\documentclass[11pt,reqno]{amsart}
\usepackage[left=1.3in, right=1.2in, top=1in, bottom=1in, includefoot]{geometry}
\usepackage{lipsum}
\usepackage{anyfontsize}
\usepackage{graphicx}
\usepackage{mathrsfs} 
\usepackage{amssymb}   
\usepackage{amsthm}  
\usepackage{bbm}
\usepackage[shortlabels]{enumitem}
\usepackage{xcolor}
\usepackage{soul}
\usepackage{scalerel}
\newtheorem{thm}{Theorem}[section]

\newtheorem{prop}[thm]{Proposition}
\newtheorem{cor}[thm]{Corollary}
\newtheorem{Def}[thm]{Definition}
\newtheorem{lemma}[thm]{Lemma}
\newtheorem{remark}[thm]{Remark}
\newtheorem{notation}[thm]{Notation}

\usepackage{pgfplots}
\pgfplotsset{compat=1.16}
\definecolor{aqua}{rgb}{0.0, 0.44, 1.0}

\newtheorem{assumption}[thm]{Assumption}

\usepackage{mathtools}
\usepackage{esint}

\newcommand{\ve}{{\varepsilon}}
\newcommand{\N}{{\mathbb{N}}}
\newcommand{\Z}{{\mathbb{Z}}}
\newcommand{\R}{{\mathbb{R}}}
\newcommand{\GD}{{\Gamma}^{\rm D}}
\newcommand{\opA}{A}
\newcommand{\RR}{{\mathbb R}}
\newcommand{\dz}{\partial_z}
\numberwithin{equation}{section}
\title{
	The F. John model and Cummins' equations for freely floating objects}

\author[D. Lannes]{David Lannes}
\address{Institut de Mathématiques de Bordeaux\\  Université de Bordeaux et CNRS UMR 5251\\ 351 Cours de la Libération \\ 33405 Talence Cedex \\ France}\email{David.Lannes@math.u-bordeaux.fr}

\author[M. O. Paulsen]{Martin Oen Paulsen} 
\address{Institut de Mathématiques de Bordeaux\\  Université de Bordeaux et CNRS UMR 5251\\ 351 Cours de la Libération \\ 33405 Talence Cedex \\ France}\email{Martin.Paulsen@math.u-bordeaux.fr}

\date{\today}
\begin{document}
	\maketitle
	\begin{abstract}  
		In this paper, we address the well-posedness theory of F. John’s problem for freely floating objects in a two-dimensional framework. This problem is a linear description of the interactions between an incompressible, irrotational free-surface fluid and a partially immersed solid object.  It is related to the Cummins equations, which are a set of coupled integro-differential equations widely used by naval engineers,       
		\begin{equation*}
			({\mathcal M}+{\mathcal M}_{\rm a})\ddot{\mathtt X}
			+   
			\rho{\mathtt g}\int_{0}^t\mathcal{K}(t-\tau)\dot{\mathtt X}(\tau) \: \mathrm{d}\tau 
			+
			{\mathcal C}{\mathtt X} 
			=  
			\mathcal{F};
		\end{equation*}
		%
		%
		here ${\mathtt X}(t)\in \R^3$ describes the three degrees of freedom of a partially immersed solid object at time $t\geq 0$ (two translations and one rotation), while $\mathcal{M}$ and $\mathcal{M}_a$ are the mass and added mass matrices, $\mathcal{K}$ is a matricial kernel that accounts for the memory effect of the fluid, $\mathcal{C}{\mathtt X}$ is the hydrostatic force, and $\mathcal{F}$ is the force resulting from the waves acting on the object. Our results provide a rigorous justification of this model and the first proof of its well-posedness. To this end, we show that F. John's problem has a Hamiltonian structure, albeit with a non-definite Hamiltonian which compels us to work in a semi-Hilbertian framework. The presence of corners in the fluid domain also induces a lack of elliptic regularity for the boundary value problem satisfied by the velocity potential. This is why the solution to F. John's problem has in general a limited regularity, even with very smooth initial data. Correspondingly, the solution ${\mathtt X}$ to Cummins' equation is in general no better than $C^3$. We demonstrate higher regularity when the contact angles are small or equal to $\pi/2$, provided that some compatibility conditions, propagated by the flow, are satisfied. This allows us to exhibit qualitative differences between the three degrees of freedom of the object. For instance, the motion of an object allowed to move only vertically can be $C^\infty$, while it cannot be better than $C^3$ if it is only allowed to translate horizontally.

	\end{abstract}

	\section{Introduction}

	\subsection{The floating body problem}

	The derivation of equations for ship motions based on linear potential flow is a historical problem that dates back to the 19th century. Two significant contributions to this field were made by Froude in 1861 \cite{Froude1861}, who examined a ship rolling in waves, and Krylov, who investigated motions with six degrees of freedom \cite{Kriloff1896,Kriloff1898}. In their work, the dynamics of the floating object were modeled as a damped harmonic oscillator, which marked an important step toward developing a general theory. However, the derivation was based on strict assumptions that neglected the effects of diffracted waves and those generated by the ship itself. A large community set out to find an appropriate formulation of the linear floating body problem, and one of the first systematic approaches was developed by John in 1949 for motions with six degrees of freedom.

	In a two-part series, John \cite{John49,John50} gave a general formulation of the linear floating body problem. The floating body problem describes the interaction between a fluid and a partially immersed solid under the assumptions:
	\begin{itemize}
		\item [-] The amplitude of the surface elevation is assumed to be small.
		\item [-] The motion of the solid is assumed to be free or forced with small amplitude. 
		\item [-] The fluid is governed by the linear Bernoulli equation.
	\end{itemize}
	Under these assumptions, John neglected the nonlinear effects due to the variations of the fluid domain to arrive at a set of equations which can be stated, in the two-dimensional case, as
	\begin{equation}\label{FJ1}
		\begin{cases}
			\partial_t \zeta - \partial_z \phi =0 &\mbox{ on } \Gamma^{\rm D}\\
			\partial_t \psi + {\mathtt g}\zeta =0 &\mbox{ on } \Gamma^{\rm D},
		\end{cases}
	\end{equation}
	where ${\mathtt g}$ is the gravity,  $\Gamma^{\rm D}$ consists of two flat finite intervals or half-lines corresponding to the part of the surface of the fluid in contact with the atmosphere when the fluid is at rest, $\zeta$ is the free surface elevation, and
	$\phi$ is the velocity potential. This latter is determined in terms of its trace $\psi$ on $\Gamma^{\rm D}$ and of the velocity ${\mathcal V}^{\rm w}$ of the floating object by solving
	\begin{equation}\label{Eq: Laplace equation}
		\begin{cases}
			\Delta \phi = 0  & \text{in}  \quad \Omega 
			\\ 
			\phi = \psi       & \text{on} \quad \Gamma^{\mathrm{D}}
			\\ 
			\mathbf{n}^{\rm w} \cdot \nabla  \phi  = \mathbf{n}^{\rm w} \cdot {\mathcal V}^{\rm w}& \text{on} \quad \Gamma^{\mathrm{w}}			\\
			\mathbf{n}^{\rm b} \cdot \nabla \phi = 0 & \text{on} \quad \Gamma^{\rm b},
		\end{cases}
	\end{equation}
	\noindent
	where $\Omega$ denotes the fluid domain at rest, while the wetted part of the object is labeled $\Gamma^{\rm w}$, and the bottom of the fluid domain is denoted $\Gamma^{\rm b}$ (see Figure \ref{Figure: Setu-up}). The first equation in \eqref{FJ1} is the linearization of the classical kinematic boundary condition at the free surface, while the second one is the linearization of the trace of Bernoulli's equation at the free surface.

	Since the object is a rigid body, its velocity ${\mathcal V}^{\rm w}$ can be expressed in terms of the time derivative of ${\mathtt X}(t)=(\tilde{x}(t),{\tilde{z}(t),\theta}(t))^{\rm T}$, where $\tilde{x}$ and $\tilde{z}$ denote the horizontal and vertical displacements of the center of mass of the object and $\theta$ its rotation. When the object is freely floating, the vector ${\mathtt X}$ is determined by Newton's equations, the acting forces and torque acting on the object being its weight and the hydrodynamical force/torque exerted by the fluid. These equations are provided in Section \ref{Sec: Hamiltonian reform}. 
	
	The work of John marked a period with several foundational works on the formulation and dynamics of the floating body problem \cite{Haskind46, Ursell49, PetersStoker54, Ursell64, Ogilvie64, Wehausen71}, culminating in the seminal work of Cummins in 1962 \cite{Cummins62}. He was able to link the formulation by Krylov and John, and proposed a description of the motion of the object as a set of coupled integro-differential equations for the degrees of freedom of the object (three in dimension two, and six in dimension three). In the present two-dimensional case, it takes the form
	\begin{equation}\label{Cummins intro}
		({\mathcal M}+{\mathcal M}_{\rm a})\ddot{\mathtt X}
		+   
		\rho{\mathtt g}\int_{0}^t\mathcal{K}(t-\tau)\dot{\mathtt X}(\tau) \: \mathrm{d}\tau 
		+
		{\mathcal C}{\mathtt X} 
		=  
		\mathcal{F},
	\end{equation}
	%
	%
	where $\mathcal{M}$ and $\mathcal{M}_a$ are the mass and added mass matrices, $\mathcal{K}$ is a matricial kernel that accounts for the memory effect of the fluid, $\mathcal{C}{\mathtt X}$ is the hydrostatic force, and $\mathcal{F}$ is the force resulting from the waves acting on the object.
	The matrices $\mathcal{M}, \mathcal{M}_a$ and the kernel ${\mathcal K}$ only depend on the geometry of the object, while the matrix ${\mathcal C}$ also depends on the position of its center of mass. The analytical description of these terms is presented in Section \ref{Sec: Hamiltonian reform}, where we revisit the derivation of F. John's model and generalize it to allow for the presence of a (possibly emerging) topography. The derivation of Cummins' equations in such general configurations is done in Subsection \ref{Subsec: Cummins}. Despite being oversimplified in many respects, these equations remain one of the primary methods used in naval engineering and marine renewable energy applications \cite{newman1988computation,Mavrakos_McIver_97,falcao2010wave,armesto2015comparative,guo2021review}, through widely used computer programs such as WaveAnalysisMIT \cite{WAMIT06}, which enables large-scale numerical simulations of floating structures with complex geometries.

	A drawback of John's model is that it neglects nonlinear, viscous, and rotational effects, which are needed to assess maximum loads on the structure in extreme sea conditions \cite{lannes_NEWS_21}. One way to include these effects is based on computational fluid dynamics (CFD) using versions of the Navier-Stokes equations \cite{Parolini05,eskilsson_CFD_15}. This approach is computationally much more expensive, as one has to deal with highly separated flows with large-scale differences between the floater and mooring systems, coupled with a non-Gaussian wave field \cite{kim_CFD_16}. For these reasons, the CFD project cannot account for a full sea state and is only used in very specific cases. 
	
	A compromise between linear theory and CFD computations was recently put forward in \cite{Lannes17}. Here, the models still neglect viscosity and vorticity but include nonlinear effects. 
	The presence of a body is accounted for in the water wave equations by applying a constraint on the surface elevation at the object's location. This approach therefore relates the floating body problem to the wide literature on congested flows \cite{Degond11,godlewski2018congested,perrin2018overview,maity2019analysis,godlewski2020congested,DalibardPerrin20}. It can be used with vertically averaged models that allow for faster numerical simulations while accounting for nonlinear effects. We direct interested readers to \cite{Bocchi20Nonlin,Bocchi2020,IguchiLannes21,BreschLannesMetivier21,BeckLannes22,Bocchi2023, BeckLannesWeynans25, IguchiLannes25} for both numerical and analytical studies of various vertically averaged models. Despite active study of these models, it remains an open question whether their solutions are close to those of the original physical system. \textit{The first step in addressing this question is to establish the well-posedness of John's problem. This is the first objective of the paper and will in turn be used to justify the Cummins equations.} Moreover, since deriving asymptotic models classically involves a loss of derivative from an asymptotic expansion, see for instance \cite{WWP}, one would need to establish high regularity of the solutions. Therefore, \textit{a second goal of the paper is to investigate the regularity properties} of the solutions to F. John's problem.

	\subsection{Related results}   While in the absence of floating objects and more generally of boundaries, the water waves equations are now well understood, the situation is far more complex in configurations where the free boundary of the fluid meets a solid boundary. T. Alazard, N. Burq and C. Zuily considered in \cite{AlazardBurqZuily_Wall_16} the case of a channel delimited by vertical walls; the surface is then the graph of a function over a finite interval. Their approach is based on a 1910 paper by Boussinesq \cite{Boussinesq1910}, where the problem is transformed into a problem on the torus using a simple reflection and periodization procedure. Extending the solution this way naturally induces singularities around the point of reflection, resulting in a domain with a Lipschitz boundary. However, by imposing the compatibility condition $\partial_x \zeta= 0$ at the contact points at $t=0$ and observing that this condition is trivially propagated by symmetry reasons, the authors were able to improve the regularity of the surface down to a point where the periodic water waves equations are known to be well posed \cite{AlazardBurqZuily14}.
	
	For non vertical walls, for instance in the presence of an emerging bottom, this argument does not work. This configuration was studied in the work of de Poyferré, \cite{dePoyferre19}, who derived \textit{a priori} estimates for the water waves equation in the case of an emerging bottom, provided that the contact angle between the free surface and the emerging bottom remains small. A similar condition is imposed by Ming and Wang \cite{MingWang20,MingWang21,MingWang24}  who proved the well-posedness in the same situation, but in the presence of surface tension. In the related context of corner singularities for the water waves equations, similar conditions also appear \cite{KinseyWu,Wu_19,CordobaEncisoGrubic}. The smallness condition on the contact angles is due to the fact that the analysis of the velocity field implies the resolution of a Laplace problem on a domain with corners. For such elliptic problems, it is known that corner singularities appear and destroy the standard elliptic regularity properties observed in smooth domains \cite{Jerison_Kenig81,Grisvard85,Dauge88,Brown94,Mazya_Kozlov_97,Brown_Lanzini_Capogna08}. There is no smallness assumption on the contact angles in the works of I. Guo and I. Tice \cite{GuoTice1,GuoTice2} who studied the contact problem for the Navier-Stokes equations with surface tension, nor in the work of E. Bocchi, A. Castro and F. Gancedo \cite{BocchiCastroGancedo25} on the Muskat problem with contact points, but this issue is still present in these references, and is the reason why the authors have to work in a low regularity setting.
	
	The issues of regularity discussed above are also at the core of the mathematical analysis of F. John's problem, and a layer of complexity is added by the fact that the free surface is not connected, which imposes to work in non standard functional spaces.
	From the time John published  the second paper in 1950 \cite{John50}, which demonstrated that the system produces a unique solution under the assumption that all motions are time-harmonic with sufficiently large frequency, and other technical assumptions, most subsequent research has concentrated on time-harmonic motions \cite{Mciver96, Mazya_LinearWW_02}. The first well-posedness result of John's problem for unsteady waves in an optimal functional setting and an analysis of the regularity properties of the solutions was only recently established for a fixed object \cite{LannesMing24}.  The case of an object having the shape of a half-disk and allowed to move only vertically in a fluid of infinite depth was recently considered in \cite{Ocqueteau_Tucsnak_25}. Using the explicit expression of the velocity potential in that configuration, the authors proved the well-posedness of F. John's problem at low regularity and assuming that the velocity potential, and not only its gradient, is in $L^2(\Omega)$. This latter assumption was also made in \cite{LannesMing24} to establish the higher regularity properties of the solution. In this paper, we address the case of general geometric configurations of a solid allowed to move along three degrees of freedom, of general initial sea states, and investigate as well the issue of the regularity of the solutions. We also derive and investigate Cummins' equations and, in particular, whether the convolution kernel ${\mathcal K}$ makes sense in all possible configurations, in spite of the possible presence of corner singularities.

	To provide a more comprehensive description of the main results, we must first write down John's equations and make the coupling with Newton's equations precise. This process will be detailed in the following subsection.

	\subsection{The Fritz John floating body problem} We consider the interactions between waves at the surface of a two-dimensional fluid of finite depth with a partially immersed solid object and a possibly emerging bottom, see Figure \ref{Figure: Setu-up}. We are interested in small disturbances of the free surface and small displacements of the solid object; the fluid domain is consequently a small perturbation of a fixed domain $\Omega$ corresponding to the equilibrium configuration where the fluid is at rest and the solid in hydrostatic equilibrium. This domain has corners formed by the emerging bottom and the immersed object.

	\begin{figure}[h]
		\centering
		\begin{tikzpicture}
			\begin{axis}[x=1cm, y=1cm,
				axis lines=middle,
				axis x line shift=1.7, 
				xlabel=\(x\),ylabel=\(z\),
				x label style={anchor=north},
				y label style={anchor=east},
				xticklabel=\empty,
				xmin=0,xmax=12,ymin=-1.7,ymax=0.8,
				xtick={},
				clip=false,
				domain=0:12, 
				smooth,
				ticks=none
				]
				
				\node[label=311:{$\color{brown}\Gamma^{\rm b}$}] (n2) at (11,-0.8) {};
				\draw [brown, very thick,  tension=0.4] plot [smooth] coordinates {(0,0) (0.5,-0.6) (1.2,-1.1) (2,-1.3) (3,-1.1) (5,-1.5) (7,-1.4)   (9.5,-1.6)  (12,-1.5)};
				
				\fill[brown!100,nearly transparent] (0,0) -- (0,-1.7) -- (0.5,-1.7) -- (0.5,-0.6) -- cycle;
				\fill[brown!100,nearly transparent] (0.5,-1.7) -- (0.5,-0.6) -- (1.2,-1.1) -- (1.2,-1.7) --cycle;
				\fill[brown!100,nearly transparent] (1.2,-1.1) -- (1.2,-1.7) -- (2,-1.7) -- (2,-1.3) -- cycle;
				\fill[brown!100,nearly transparent] (2,-1.7) -- (2,-1.3) -- (3,-1.1) -- (3,-1.7) -- cycle;
				\fill[brown!100,nearly transparent] (3,-1.1) -- (3,-1.7) -- (5,-1.7) -- (5,-1.5) -- cycle;
				\fill[brown!100,nearly transparent] (5,-1.7) -- (5,-1.5) -- (7,-1.4) -- (7,-1.7) --  cycle;
				\fill[brown!100,nearly transparent] (7,-1.4) -- (7,-1.7) -- (9.5,-1.7) -- (9.5,-1.6) --  cycle;
				\fill[brown!100,nearly transparent] (9.5,-1.7) -- (9.5,-1.6) -- (12,-1.5) -- (12,-1.7) --  cycle;			
				
				\node[label=311:{$0$}] (n2) at (-0.5,0.4) {};

				\node[label=311:{$\color{blue}\mathcal{E}_-$}] (n2) at (1.7,0.8) {};
				\draw [blue, very thick,  tension=0.4] plot [smooth] coordinates {(0,0) (4.5,0)};

				\node[label=311:{$\color{blue}\mathcal{E}_+$}] (n2) at (9.35,0.8) {};
				\draw [blue, very thick,  tension=0.4] plot [smooth] coordinates {(7.5,0) (12,0)};
				
				\node[label=311:{$\Gamma^{\rm w}$}] (n2) at (5.7,-0.5) {};
				\draw [black, very thick,  tension=0.4] plot [smooth] coordinates {(4.5,0) (4.8,-0.4) (5,-0.5) (6,-0.6) (7,-0.5) (7.2,-0.4) (7.5,0)};
				
				\draw [black!80,densely dashed,  thick, tension=0.4] plot [smooth] coordinates {(4.5,0) (4.6,+0.2) (5,+0.3) (6,0.4) (7,0.3) (7.4,0.2) (7.5,0)};
				
				\fill[gray!100,nearly transparent] (4.5,0) -- (4.5,0) -- (4.6,0.2) -- (4.8,-0.4) -- cycle;
				\fill[gray!100,nearly transparent] (4.6,0.2) -- (4.8,-0.4) -- (5,-0.5) -- (5,0.3) -- cycle;
				\fill[gray!100,nearly transparent] (5,-0.5) -- (5,0.3) -- (6,0.4) -- (6,-0.6) -- cycle;
				\fill[gray!100,nearly transparent] (6,0.4) -- (6,-0.6) -- (7,-0.5) -- (7,0.3) -- cycle;
				\fill[gray!100,nearly transparent] (7,-0.5) -- (7,0.3) -- (7.4,0.2)  -- (7.2,-0.4) -- cycle;
				\fill[gray!100,nearly transparent] (7.4,0.2)  -- (7.2,-0.4) -- (7.5,0) -- (7.5,0) -- cycle;
				
				\fill[aqua!100,nearly transparent] (0,0) -- (0,0) -- (0.5,0) -- (0.5,-0.6) -- cycle;
				\fill[aqua!100,nearly transparent] (0.5,0) -- (0.5,-0.6) -- (1.2,-1.1) -- (1.2,0) --cycle;
				\fill[aqua!100,nearly transparent] (1.2,-1.1) -- (1.2,0) -- (2,0) -- (2,-1.3) -- cycle;
				\fill[aqua!100,nearly transparent] (2,0) -- (2,-1.3) -- (3,-1.1) -- (3,0) -- cycle;		
				\fill[aqua!100,nearly transparent] (3,-1.1) -- (3,0) -- (4.5,0) -- (4.5,-1.4) -- cycle;
				
				\fill[aqua!100,nearly transparent] (4.5,0) -- (4.5,-1.4) -- (4.88,-1.48) -- (4.88,-0.455) --  cycle;
				\fill[aqua!100,nearly transparent] (4.88,-1.48) -- (4.88,-0.46) -- (6,-0.6) -- (6,-1.47) --  cycle;
				\fill[aqua!100,nearly transparent] (6,-0.6) -- (6,-1.47) -- (7.18,-1.39) -- (7.18,-0.46) --  cycle;
				\fill[aqua!100,nearly transparent] (7.18,-1.4) -- (7.18,-0.455) -- (7.5,0) -- (7.5,-1.44) --  cycle;
				\fill[aqua!100,nearly transparent] (7.5,-1.41) -- (7.5,0) -- (9.5,0) -- (9.5,-1.6) --  cycle;
				
				\fill[aqua!100,nearly transparent] (9.5,0) -- (9.5,-1.6) -- (12,-1.5) -- (12,0) --  cycle;	
				
				\node[label=311:{$\color{brown}\Gamma^{\rm b}$}] (n2) at (11,-0.8) {};
				
			\end{axis}
		\end{tikzpicture}
		\caption{The fluid domain $\Omega$ is the shaded light blue color with lower boundary $\Gamma^{\rm b}$ in dark brown. Its upper boundary is comprised of; $\Gamma^{\rm w}$ the wetted part of the object in solid black and $\Gamma^{\rm D}$ which the union of the solid blue lines $\mathcal{E}_-$ and $\mathcal{E}_+$.}
		\label{Figure: Setu-up}
	\end{figure}
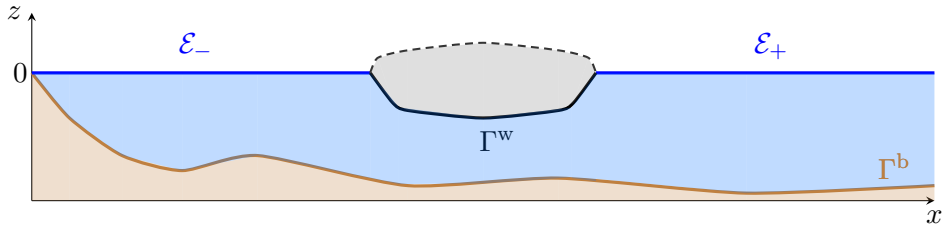
	
	To be precise, we make the following assumptions on $\Omega$ and its boundary $\Gamma$. This assumption allows several configurations represented in Figure \ref{Fig: Admissible config}. In particular, both connected components of $\GD$ can be either finite intervals or half-lines. 
	
	\begin{figure}[h!]
		\centering
		\begin{minipage}{.5\textwidth}
			\centering
			\begin{tikzpicture}
				\begin{axis}[x=1cm, y=1cm,
					axis lines=middle,
					axis x line shift=1.7, 
					xlabel=\(x\),ylabel=\(z\),
					x label style={anchor=north},
					y label style={anchor=east},
					xticklabel=\empty,
					xmin=0,xmax=6,ymin=-1.7,ymax=0.8,
					xtick={},
					clip=false,
					domain=0:12, 
					smooth,
					ticks=none
					]
					
					\draw [brown, very thick,  tension=0.4] plot [smooth] coordinates {(0,-1.5) (0.5,-1.4) (1.2,-1.3) (2,-1.3) (3,-1.1) (5,-1.5) (6,-1.4)};
					
					\fill[brown!100,nearly transparent] (0,-1.5) -- (0,-1.7) -- (0.5,-1.7) -- (0.5,-1.4) -- cycle;
					\fill[brown!100,nearly transparent] (0.5,-1.7) -- (0.5,-1.4) -- (1.2,-1.3) -- (1.2,-1.7) --cycle;
					\fill[brown!100,nearly transparent] (1.2,-1.3) -- (1.2,-1.7) -- (2,-1.7) -- (2,-1.3) -- cycle;
					\fill[brown!100,nearly transparent] (2,-1.7) -- (2,-1.3) -- (3,-1.1) -- (3,-1.7) -- cycle;
					\fill[brown!100,nearly transparent] (3,-1.1) -- (3,-1.7) -- (5,-1.7) -- (5,-1.5) -- cycle;
					\fill[brown!100,nearly transparent] (5,-1.7) -- (5,-1.5) -- (6,-1.4) -- (6,-1.7) --  cycle;
					
					\node[label=311:{$0$}] (n2) at (-0.5,0.4) {};

					\draw [blue, very thick,  tension=0.4] plot [smooth] coordinates {(0,0) (4.5/2,0)};

					\draw [blue, very thick,  tension=0.4] plot [smooth] coordinates {(7.5/2,0) (12/2,0)};
					
					\draw [black, very thick,  tension=0.4] plot [smooth] coordinates {(4.5/2,0) (4.8/2,-0.4/2) (5/2,-0.5/2) (6/2,-0.6/2) (7/2,-0.5/2) (7.2/2,-0.4/2) (7.5/2,0)};
					
					\draw [black!80,densely dashed,  thick, tension=0.4] plot [smooth] coordinates {(4.5/2,0) (4.6/2,+0.2/2) (5/2,+0.3/2) (6/2,0.4/2) (7/2,0.3/2) (7.4/2,0.2/2) (7.5/2,0)};
					
					\fill[gray!100,nearly transparent] (4.5/2,0) -- (4.5/2,0) -- (4.6/2,0.2/2) -- (4.8/2,-0.4/2) -- cycle;
					\fill[gray!100,nearly transparent] (4.6/2,0.2/2) -- (4.8/2,-0.4/2) -- (5/2,-0.5/2) -- (5/2,0.3/2) -- cycle;
					\fill[gray!100,nearly transparent] (5/2,-0.5/2) -- (5/2,0.3/2) -- (6/2,0.4/2) -- (6/2,-0.6/2) -- cycle;
					\fill[gray!100,nearly transparent] (6/2,0.4/2) -- (6/2,-0.6/2) -- (7/2,-0.5/2) -- (7/2,0.3/2) -- cycle;
					\fill[gray!100,nearly transparent] (7/2,-0.5/2) -- (7/2,0.3/2) -- (7.4/2,0.2/2)  -- (7.2/2,-0.4/2) -- cycle;
					\fill[gray!100,nearly transparent] (7.4/2,0.2/2)  -- (7.2/2,-0.4/2) -- (7.5/2,0) -- (7.5/2,0) -- cycle;
					
					\fill[aqua!100,nearly transparent] (0,0) -- (0,-1.5) -- (0.5,-1.4) -- (0.5,0) -- cycle;
					\fill[aqua!100,nearly transparent] (0.5,0) -- (0.5,-1.4) -- (1.2,-1.3) -- (1.2,0) --cycle;
					\fill[aqua!100,nearly transparent] (1.2,-1.3) -- (1.2,0) -- (2,0) -- (2,-1.3) -- cycle;
					\fill[aqua!100,nearly transparent] (2,0) -- (2,-1.3) -- (4.5/2,-1.24) -- (4.5/2,0) -- cycle;		
					
					\fill[aqua!100,nearly transparent] (4.5/2,0) -- (4.5/2,-1.24) -- (4.88/2,-1.18) -- (4.88/2,-0.455/2) --  cycle;
					\fill[aqua!100,nearly transparent] (4.88/2,-1.18) -- (4.88/2,-0.46/2) -- (6/2,-0.6/2) -- (6/2,-1.08) --  cycle;
					\fill[aqua!100,nearly transparent] (6/2,-0.6/2) -- (6/2,-1.08) -- (7.18/2,-1.2) -- (7.18/2,-0.46/2) --  cycle;
					\fill[aqua!100,nearly transparent] (7.18/2,-1.2) -- (7.18/2,-0.455/2) -- (7.5/2,0) -- (7.5/2,-1.23) --  cycle;
					\fill[aqua!100,nearly transparent] (7.5/2,-1.23) -- (7.5/2,0) -- (10/2,0) -- (10/2,-1.5) --  cycle;
					
					\fill[aqua!100,nearly transparent] (10/2,0) -- (10/2,-1.5) -- (12/2,-1.4) -- (12/2,0) --  cycle;	
					
					
				\end{axis}
			\end{tikzpicture}

		\end{minipage}%
		\begin{minipage}{.5\textwidth}
			\centering 
			\begin{tikzpicture}
				\begin{axis}[x=1cm, y=1cm,
					axis lines=middle,
					axis x line shift=1.7, 
					xlabel=\(x\),ylabel=\(z\),
					x label style={anchor=north},
					y label style={anchor=east},
					xticklabel=\empty,
					xmin=0,xmax=6,ymin=-1.7,ymax=0.8,
					xtick={},
					clip=false,
					domain=0:12, 
					smooth,
					ticks=none
					]
					
					\draw [brown, very thick,  tension=0.5] plot [smooth] coordinates {(0,0) (1,-1.4) (1.5,-1.45) (2,-1.35) (3,-1.6) (5,-1.4) (6,0)};
					
					\fill[brown!100,nearly transparent] (0,0) -- (0,-1.7) -- (1,-1.7) -- (1,-1.4) -- cycle;

					\fill[brown!100,nearly transparent] (1,-1.7) -- (1,-1.4) -- (1.2,-1.5) -- (1.2,-1.7) --cycle;
					\fill[brown!100,nearly transparent] (1.2,-1.5) -- (1.2,-1.7) -- (2,-1.7) -- (2,-1.35) -- cycle;
					\fill[brown!100,nearly transparent] (2,-1.7) -- (2,-1.35) -- (3,-1.62) -- (3,-1.7) -- cycle;
					\fill[brown!100,nearly transparent] (3,-1.625) -- (3,-1.7) -- (4.8,-1.7) -- (4.8,-1.5) -- cycle;
					\fill[brown!100,nearly transparent] (4.8,-1.5) -- (4.8,-1.7) -- (5.1,-1.7) -- (5.1,-1.36) --  cycle;
					\fill[brown!100,nearly transparent] (5.1,-1.36) -- (5.1,-1.7) -- (5.5,-1.7) -- (5.5,-0.85) --  cycle;
					\fill[brown!100,nearly transparent] (5.5,-0.85) -- (5.5,-1.7) -- (6,-1.7) -- (6,0) -- cycle;			
					
					\node[label=311:{$0$}] (n2) at (-0.5,0.4) {};

					\draw [blue, very thick,  tension=0.4] plot [smooth] coordinates {(0,0) (4.5/2,0)};

					\draw [blue, very thick,  tension=0.4] plot [smooth] coordinates {(7.5/2,0) (12/2,0)};
					
					\draw [black, very thick,  tension=0.4] plot [smooth] coordinates {(4.5/2,0) (4.8/2,-0.4/2) (5/2,-0.5/2) (6/2,-0.6/2) (7/2,-0.5/2) (7.2/2,-0.4/2) (7.5/2,0)};
					
					\draw [black!80,densely dashed,  thick, tension=0.4] plot [smooth] coordinates {(4.5/2,0) (4.6/2,+0.2/2) (5/2,+0.3/2) (6/2,0.4/2) (7/2,0.3/2) (7.4/2,0.2/2) (7.5/2,0)};
					
					\fill[gray!100,nearly transparent] (4.5/2,0) -- (4.5/2,0) -- (4.6/2,0.2/2) -- (4.8/2,-0.4/2) -- cycle;
					\fill[gray!100,nearly transparent] (4.6/2,0.2/2) -- (4.8/2,-0.4/2) -- (5/2,-0.5/2) -- (5/2,0.3/2) -- cycle;
					\fill[gray!100,nearly transparent] (5/2,-0.5/2) -- (5/2,0.3/2) -- (6/2,0.4/2) -- (6/2,-0.6/2) -- cycle;
					\fill[gray!100,nearly transparent] (6/2,0.4/2) -- (6/2,-0.6/2) -- (7/2,-0.5/2) -- (7/2,0.3/2) -- cycle;
					\fill[gray!100,nearly transparent] (7/2,-0.5/2) -- (7/2,0.3/2) -- (7.4/2,0.2/2)  -- (7.2/2,-0.4/2) -- cycle;
					\fill[gray!100,nearly transparent] (7.4/2,0.2/2)  -- (7.2/2,-0.4/2) -- (7.5/2,0) -- (7.5/2,0) -- cycle;
					
					\fill[aqua!100,nearly transparent] (0,0) -- (0,0) -- (1,-1.4) -- (1,0) -- cycle;
					\fill[aqua!100,nearly transparent] (1,0) -- (1,-1.4) -- (1.2,-1.5) -- (1.2,0) --cycle;
					\fill[aqua!100,nearly transparent] (1.2,-1.5) -- (1.2,0) -- (2,0) -- (2,-1.35) -- cycle;
					\fill[aqua!100,nearly transparent] (2,0) -- (2,-1.35) -- (4.5/2,-1.41) -- (4.5/2,0) -- cycle;		
					
					\fill[aqua!100,nearly transparent] (4.5/2,0) -- (4.5/2,-1.4) -- (4.88/2,-1.48) -- (4.88/2,-0.455/2) --  cycle;
					\fill[aqua!100,nearly transparent] (4.88/2,-1.48) -- (4.88/2,-0.46/2) -- (6/2,-0.6/2) -- (6/2,-1.61) --  cycle;
					\fill[aqua!100,nearly transparent] (6/2,-0.6/2) -- (6/2,-1.61) -- (7.18/2,-1.58) -- (7.18/2,-0.46/2) --  cycle;
					\fill[aqua!100,nearly transparent] (7.18/2,-1.58) -- (7.18/2,-0.455/2) -- (7.5/2,0) -- (7.5/2,-1.59) --  cycle;
					\fill[aqua!100,nearly transparent] (7.5/2,-1.59) -- (7.5/2,0) -- (4.8,0) -- (4.8,-1.5) --  cycle;
					\fill[aqua!100,nearly transparent] (4.8,0) -- (4.8,-1.5) -- (5,-1.4) -- (5,0) -- cycle;
					\fill[aqua!100,nearly transparent] (5,-1.4) -- (5,0) -- (5.2,0) -- (5.2,-1.25) -- cycle;
					\fill[aqua!100,nearly transparent] (5.2,0) -- (5.2,-1.28) -- (6.02,0) -- (6.02,0) -- cycle;
					
				\end{axis}
			\end{tikzpicture} 
		\end{minipage}
		
		\caption{Admissible configurations}
		\label{Fig: Admissible config}
	\end{figure}
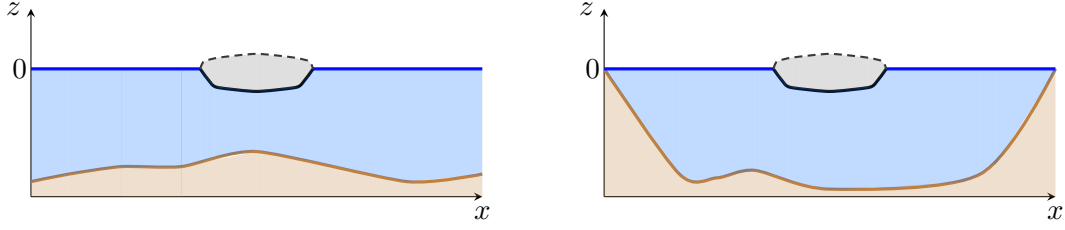

	The analysis could also be extended to the presence of several objects, as in \cite{LannesMing24}, but we just consider a single object for the sake of clarity.
	\begin{assumption}\label{Assumption domain}
		The fluid domain $\Omega$ is delimited from above and from below by two open curves $\Gamma^{\rm top}$ and $\Gamma^{\rm b}$ respectively, that do not intersect, and such that:
		\begin{itemize}
			\item [(1)] There exist $- \infty \leq x_{\rm L} < x_{\rm l} < x_{\rm r} < x_{\rm R} \leq \infty $ such that $\Gamma^{\rm top}$ is a broken line with (possibly infinite) endpoints ${\mathtt c}_{\rm L} = (x_{\rm L}, 0)$ and ${\mathtt c}_{\rm R} = (x_{\rm R}, 0)$ and vertices ${\mathtt c}_{\rm (l)} = (x_{\rm l}, 0)$, and ${\mathtt c}_{\rm (r)} = (x_{\rm r}, 0)$. The segment $({\mathtt c}_{\rm (l)},{\mathtt c}_{\rm (r)})$ which corresponds to the wetted part of the boundary of the object, is a smooth open curve denoted by  $\Gamma^{\rm w}$ and contained in the lower-half plane $\{z<0\}$. All the other segments on $\Gamma^{\rm top}$ are flat and contained in the axis $\{z=0\}$.
			\item [(2)] The bottom $\Gamma^{{\rm b}}$ is the graph on $(x_{\rm L}, x_{\rm R})$ of a smooth and bounded  function $b$ such that $b(x_{\rm L})=0$ if $-\infty < x_{\rm L}$ and $\limsup \limits_{x\rightarrow - \infty} b(x)<0$ if $x_{\rm L} = -\infty$, and similar condition at the other endpoint $x_{\rm R}$.
			\item [(3)]  At each corner, the fluid domain forms an angle strictly larger than $0$ and strictly smaller than $\pi$. 
			\item [(4)]  The boundaries are flat in the vicinity of each corner. 
		\end{itemize}
	\end{assumption}
	We also need the following notations.
	\begin{notation}In the configuration of Assumption \ref{Assumption domain}:\\
		%
		%
		%
		%
		%
		%
		- We have $\Gamma^{\rm D} = ({\mathcal E}_-\cup{\mathcal E}_+)\times \{0\}$, where $\mathcal{E}_- = (x_{\rm L}, x_{\rm l})$ and $\mathcal{E}_+ = (x_{\rm r}, x_{\rm R})$. We often identify $\Gamma^{\rm D}$ with ${\mathcal E}_-\cup{\mathcal E}_+$. \\
		- We also write $\Gamma^{\rm N}=\Gamma^{{\rm b}} \cup \Gamma^{\rm w}$, and define $\Gamma^{*} = \Gamma^{\mathrm{D}}\cup \Gamma^{\rm N}$.\\
		- We denote by ${\bf n}^{\rm w}$ the unit outward normal vector to $\Gamma^{\rm w}$, and by ${\bf n}^{\rm b}$ the unit upward (and therefore inward) vector on $\Gamma^{\rm b}$.\\
		- We denote by ${\mathtt C}$ the set of all finite contact points.

	\end{notation}

	The equations  \eqref{FJ1} and \eqref{Eq: Laplace equation} given above must be complemented by equations describing the motion of the object.
	As mentioned previously, the velocity of the solid can be expressed in terms of the time derivative of ${\mathtt X}(t)=(\tilde{x}(t),{\tilde{z}(t),\theta}(t))^{\rm T}$, where $\tilde{x}$ and $\tilde{z}$ denote the horizontal and vertical displacements of the center of mass of the object with respect to its equilibrium position $G_{\rm eq}=(x_{G,{\rm eq}},z_{G,{\rm eq}})^{\rm T}$, and $\theta$ its rotation. When the object is freely floating, the solid must obey Newton's equations. We show in Appendix \ref{AppderN} that their linear approximation takes the following compact form 
	\begin{equation}\label{Eq: Newtons law's of motion momentum}
		{\mathcal M}\dot{\mathtt V}(t)=-\rho \int_{\Gamma^{\rm w}}\partial_t \phi {\boldsymbol \kappa}-{\mathcal C}{\mathtt X},
	\end{equation}
	where ${\mathtt V}=\dot{\mathtt X}$ and ${\mathcal M}=\mbox{diag}({\mathfrak m},{\mathfrak m},{\mathfrak i})$ with ${\mathfrak m}$ the mass of the object and ${\mathfrak i}$ its moment of inertia, while ${\boldsymbol \kappa}$ and ${\mathcal C}$ are given by
	\begin{equation}\label{defkappaC}
		{\boldsymbol \kappa}=\begin{pmatrix}
			{\bf n}^{\rm w}\cdot {\bf e}_x \\ {\bf n}^{\rm w}\cdot {\bf e}_z \\ -{\bf r}_{G_{\rm eq}}^\perp \cdot {\bf n}^{\rm w}
		\end{pmatrix}
		\quad\mbox{ and }\quad
		{\mathcal C}=\begin{pmatrix}  0 & 0 & 0 \\ 
			0 & c_{zz} & -c_{z\theta} \\
			0 & -c_{\theta z} & c_{\theta\theta}
		\end{pmatrix},
	\end{equation}
	where 
	$${\bf r}_{G_{\rm eq}}(x,z)=(x-x_{G,{\rm eq}},z-z_{G,{\rm eq}})^{\rm T};$$
	the exact expression of the coefficients of ${\mathcal C}$ is not important at this point.  
	
	The system of equations formed by \eqref{FJ1}, \eqref{Eq: Laplace equation}, \eqref{Eq: Newtons law's of motion momentum}, and  the equation for ${\mathcal V}^{\rm w}$ defined in the Appendix by \eqref{eqUsol0}, namely,
	$$
	{\mathcal V}^{\rm w}=(\dot{\tilde{x}},\dot{\tilde{z}})^{\rm T}-\dot{\theta}r_{G_{\rm eq}}^\perp
	$$
	constitutes Fritz John's model \cite{John49} for the wave-structure interactions of small amplitude in the case where the object is partially immersed and freely floating. We refer to Appendix \ref{AppderN} for a self-contained and alternative derivation of this model, and its generalization to the more general configurations allowed by Assumption \ref{Assumption domain}.
	\\

	\subsection{Description of the results and organization of the paper} In this paper, we revisit the derivation of the Cummins equations \cite{Cummins62} in the general configurations allowed by Assumption \ref{Assumption domain} and give the first well-posedness result for \eqref{Cummins intro}.  The rigorous derivation of Cummins' equations is a consequence of the well-posedness of John's problem and of the regularity properties of its solutions. This well-posedness theory is based on several results that are of independent interest.
	
	Section \ref{Sec: Hamiltonian reform} is dedicated to new reformulations of John's problem, as presented in equations \eqref{FJ1}-\eqref{Eq: Newtons law's of motion momentum}. In this section, we introduce the Dirichlet-Neumann map studied in \cite{LannesMing24} and use it to reformulate \eqref{FJ1} as a set of nonlocal equations cast on $\Gamma^{\rm D}$ and inspired by the classical Zakharov-Craig-Sulem formulation of the water-waves problem \cite{Zakharov68,Craig_Sulem_Sulem_92,Craig_Sulem_93}; due to the nonlocality of the equations, there is an additional source term in the equations, which is caused by the motion of the object.  This source term, which makes the motion of the object felt on the entirety of $\GD$, is expressed in terms of Kirchhoff potentials. Through the Kirchhoff potentials, we provide an analytical formulation of the added mass effect $\mathcal{M}_a$ hidden in equation \eqref{Eq: Newtons law's of motion momentum} and identify the energy associated with the system as:
	\begin{equation*} 
		{\mathcal H}(\zeta,\psi,{\mathtt X},{\mathtt V})=\frac{1}{2} {\mathtt g}\int_{\Gamma^{\rm D}}\zeta^2+\frac{1}{2} \int_{\Gamma^{\rm D}} \psi G_0\psi+\frac{1}{2\rho}{\mathcal C} {\mathtt X}\cdot  {\mathtt X}+\frac{1}{2\rho}({\mathcal M}+{\mathcal M}_{\rm a}){\mathtt V}\cdot{\mathtt V}.
	\end{equation*}
	In Subsection \ref{Subsec: Ham struc}, we demonstrate that the aforementioned reformulation of the equations possesses a Hamiltonian structure with Hamiltonian ${\mathcal H}$. Under certain conditions on the coefficients in $\mathcal{C}$, ${\mathcal H}$ can be used to construct a semi-Hilbert space adapted to the mathematical analysis of F. John's problem. We relate these conditions on the matrix ${\mathcal C}$ to classical equilibrium criteria used in ship construction.

	In Section \ref{Sec: Wp John and just Cummins}, we develop an abstract framework for John's problem and establish its well-posedness, which  in turn is used to justify the Cummins equations. The well-posedness is obtained in the semi-Hilbert space associated with the Hamiltonian, and that can be characterized as
	\begin{equation*}
		{\mathbb X}=L^2(\Gamma^{\rm D})\times {\mathbb R}^3\times \dot{H}^{1/2}(\Gamma^{\rm D})\times {\mathbb R}^3,
	\end{equation*}
	where $\dot{H}^{1/2}(\Gamma^{\rm D})$ is the trace space associated with the Beppo-Levi space $\dot{H}^1(\Omega)$. We demonstrate in Subsection \ref{Subsec: Abstract reformulations} that the evolution operator in \eqref{FJabstract} is skew-adjoint with respect to the scalar product associated with ${\mathbb X}$. These elements are utilized in Subsection \ref{Subsec: Well-posedness of Fritz John's model} to establish the well-posedness in the energy space ${\mathbb X}$. For the proof, we apply a duality method on the semi-normed space $L^2([0,T];\mathbb{X})$, following the ideas used in \cite{LannesMing24} for a fixed object. The main difference here is that we have to deal with the lack of control of the semi-norm of ${\mathbb X}$ on the horizontal displacements, due to the degeneracy of the matrix ${\mathcal C}$. 	
	\begin{thm}
		The initial value problem for John's problem is well-posed in the energy space.
	\end{thm}
	\noindent
	Having this result at hand we turn to the derivation of Cummins equations. In Subsection \ref{Subsec: Cummins} we revisit and generalize the derivation within our setting and give an analytical description of the operators appearing in \eqref{Cummins intro}. We then demonstrate that this system is equivalent to John's formulation and present the first well-posedness result, thereby justifying the model in this context.

	\begin{thm}
		The initial value problem for the Cummins equations is well-posed in the energy space.
	\end{thm} 
	
	The second part of the paper concerns higher-order regularity of the solutions. As explained in the introduction, the main restriction arises from the regularity of the elliptic problem with corners in the domain and from the nonlocal effect between the different components of $\Gamma^{\rm D}$. In Section \ref{sectregabstract}, we prove higher regularity in an abstract functional space constructed on iterative powers of the square root of the Dirichlet-Neumann operator. It is in particular shown that the singularities of Kirchhoff potential are an obstruction to higher order regularity of the solution.

	In Section \ref{Sec: Small angles}, we show that this abstract functional framework can be characterized in terms of standard Sobolev spaces if the contact angles at each corner of the domain are assumed to be small. We begin the section with a primer on basic results on trace estimates and elliptic theory on corner domains. In particular, in Subsection \ref{Subsec: trace mapping}, we give a characterization of the trace space associated with $\dot{H}^{s+1}(\Omega)$ for $s>1/2$. 

	\begin{prop}
		We characterize a homogeneous space $\dot{H}^{s+1/2}(\Gamma^{\rm D})$ such that the trace mapping $	\mbox{Tr}^{\rm D}: \phi\in \dot{H}^{s+1}(\Omega) \mapsto \phi_{\vert_{{\Gamma}^{\rm D}}}\mbox{ on }\dot{H}^{s+1/2}(\Gamma^{\rm D})$	is well defined, continuous, onto, and admits a continuous right-inverse for $s>1/2$. 
	\end{prop}
	These elements are used in Subsection \ref{Subsec: Elliptic small angle} to prove elliptic regularity results in the homogeneous functional setting, where we rely on the theory of \cite{Grisvard85, Dauge88, DaugeNBL90} together with the Rellich estimates from \cite{LannesMing24}. The abstract functional spaces introduced in Section \ref{sectregabstract} are then  characterized in terms of standard Sobolev spaces in Subsection \ref{Subsec: Characterization mathcal H}. With these results at hand, we prove in Subsection \ref{Subsec: higher reg for small angle} a higher regularity result in Sobolev spaces for the solution of John’s problem.

	\begin{thm}
		The initial value problem for John's problem is well-posed in higher Sobolev spaces for small contact angles.
	\end{thm}

	Lastly, we consider in Section \ref{Sec: right angles} another configuration where the abstract functional spaces of Section \ref{sectregabstract} can be characterized in a simple way, namely, when all the contact angles are equal to $\pi/2$. We introduce in Subsection \ref{Subsec: mixed pb pi/2} a new function space $\dot{H}_{\rm c.c.}^{s+1/2}(\Gamma^{\rm D})$ consisting in functions of $\dot{H}^{s+1/2}(\Gamma^{\rm D})$ that satisfy some compatibility conditions. If for a function $\psi$ defined on $\GD$, we denote by $\psi^{\mathfrak h}$ its harmonic extension that has zero normal derivative on $\Gamma^{\rm N}$, we show that $\psi^{\mathfrak h}$ has maximal elliptic regularity whenever $\psi\in \dot{H}_{\rm c.c.}^{s+1/2}(\Gamma^{\rm D})$. A loose statement is given below.

	\begin{prop}
		Assume we have right angles at the contact points, $s\geq 0$ and $\psi \in \dot{H}_{\rm c.c.}^{s+1/2}(\Gamma^{\rm D})$. Then the variational solution $\psi^{\mathfrak h}\in \dot{H}^1(\Omega)$ to \eqref{Eq: psi^H2}  enjoys the usual elliptic regularity $\dot{H}^{s+1}(\Omega)$.
	\end{prop}
	A precise statement is given in Proposition \ref{Prop: Elliptic reg DN}. Moreover, we can show that whenever the compatibility conditions are not satisfied, the potential splits into a regular and a singular part. In particular, we identify a space with weak traces at the contact points $\dot{H}^{s+1/2}_{\rm tr}(\Gamma^{\rm D})$, and we have the formal statement:
	\begin{thm}\label{Thm: sing decomp intro}
		Assume we have right angles at the contact points, $s\geq 0$ and $\psi \in \dot{H}^{s+1/2}_{\rm tr}(\Gamma^{\rm D})$. Then the variational solution $\psi^{\mathfrak h}\in \dot{H}^1(\Omega)$ to \eqref{Eq: psi^H2} can be decomposed under the form
		$$
		\psi^{\mathfrak h}=\psi^{\mathfrak h}_{\rm sing}+\psi^{\mathfrak h}_{\rm reg},
		$$
		where $\psi^{\mathfrak h}_{\rm reg}\in\dot{H}^{s+1}(\Omega)$ and $\psi^{\mathfrak h}_{\rm sing}$ is explicitly defined.
	\end{thm}
	Again, we refer the reader to Theorem \ref{Thm: sing decomp mixed pb} for a precise statement and proof. This result is of independent interest and is used in this paper to study the regularity of John’s problem in the case of a vertical wall in the neighborhood of each corner. Specifically, we are able to propagate the compatibility conditions at each corner in time and show that we have higher regularity for restricted motions of the object:

	\begin{thm}
		The initial value problem for John's problem is well-posed in higher Sobolev norms for right angles at the contact points under compatibility conditions on the data, and assuming further that one of the following cases holds
		\begin{itemize}
			\item [-] The object is only allowed to move vertically.
			\item [-] The center of mass $G$ is such that $z_G=0$ and the object is only allowed to move by vertical translation and rotations around $G$.
		\end{itemize}
		
	\end{thm}
	\noindent 
	We note that the restriction on degrees of freedom of the object is due to the creation of singularities for horizontal motions. More precisely, the Kirchhoff potential corresponding to horizontal translations has a singularity which is described explicitly by Theorem \ref{Thm: sing decomp intro}, and that limits the regularity of the solutions to John's problem and Cummins' equations.

	\subsection{Notations} We provide some notations that will be used consistently throughout the paper. 
	
	\subsubsection{General notations}
	\textcolor{white}{new line}
	\\
	\noindent
	- We let $C>0$ be a universal constant of no importance, and that may change from line to line.\\
	- As a shorthand, we use the notation $a \lesssim b$ to mean $a \leq C\: b$.\\
	- For $s\in \R$, we use the notation $s^-$ to mean $s^- = s-\ve$ for any $\ve \in (0,1)$.\\
	- Let $x$ and $z$ denote the horizontal and vertical coordinates. Let $\partial_x$ and $\partial_z$ be the corresponding partial derivatives where we also write $\nabla =(\partial_x, \partial_z)^{\rm T}$.\\
	- For any tempered distribution $f \in \mathcal{S}'(\R) $ its Fourier transform will be defined by $\mathcal{F}(f)(\xi) = \hat{f} (\xi) = \int_{\R} e^{-  i x \xi } f(x) \: \mathrm{d}x$ and its inverse Fourier transform is defined by $\mathcal{F}^{-1}(f)(x) = \check{f} (x) = \frac{1}{2\pi } \int_{\R} e^{ i x \xi } f(\xi) \: \mathrm{d}\xi$.\\
	- Let $m:\R \rightarrow \mathbb{R}$ be a smooth function. Then we will use the notation $m(\rm{D})$ for a multiplier defined in frequency by $\widehat{m(\mathrm{D}) f}(\xi) = m(\xi) \hat{f}(\xi)$.

	\subsubsection{Function spaces} 
	\textcolor{white}{new line}
	\\
	\noindent
	- For all $s\in {\mathbb R}$, $s\geq 0$, we denote by $H^s(\Omega)$ the standard Sobolev space on $\Omega$.\\
	- For all $s\geq 0$, we define the Beppo-Levi space $\dot{H}^{s+1}(\Omega)$ as
	$$
	\dot{H}^{s+1}(\Omega) =\{\phi\in L^1_{\rm loc}(\Omega) \: : \:  \nabla \phi\in H^s(\Omega)^2\},
	$$
	with associated semi-norm $\Vert \phi\Vert_{\dot{H}^{s+1}(\Omega)} = \Vert \nabla \phi\Vert_{H^s(\Omega)}$.		\\ 
	- Let $I$ be a finite interval or a half-line then for all $s\geq 0$, we denote $H^{s} (I)$ the standard Sobolev space on  $I$.\\
	- For all $s\geq 0$, and $\Gamma^{\rm D} = (\mathcal{E}_- \cup \mathcal{E}_+)$, we define $H^{s} (\Gamma^{\rm D}) = H^{s} (\mathcal{E}_-) \times H^{s} (\mathcal{E}_+)$.\\
	-  $\dot{H}^s(I)$ is given in Definition \ref{Def: Hom space I}.\\ 
	-  $\dot{H}^{s}(\Gamma^{\rm D})$ is given in Definition \ref{defGN}.\\ 
	- $\dot{\mathcal H}^{n/2}(\Gamma^{\rm D})$ is given in Definition \ref{defhighreg}. \\
	- $	\dot{H}_{\rm c.c.}^{s}(\GD)$ is given in Definition \ref{Def: Hcc}.\\
	- $	\dot{H}_{\rm tr}^{s}(\GD)$ is given in Definition \ref{defT}.\\
	- $\dot{H}^{s}(\R_{2L})$ is given in Definition \ref{Def: HsR 2L}\\
	- $\dot{H}^{s}(\R)$ is given in Definition \ref{Def: HsR}\\

	\section{Hamiltonian reformulations} \label{Sec: Hamiltonian reform}
	
	In this section, we propose several reformulations of F. John's problem \eqref{FJ1}-\eqref{Eq: Newtons law's of motion momentum}. Using the Dirichlet-Neumann operator and the Kirchhoff potentials introduced in Section \ref{sectKirchhoff} and taking advantage of the added mass effect, a convenient reformulation is proposed in Section \ref{sectreform}. We then show that this formulation admits a non-canonical Hamiltonian structure, which can be transformed into a canonical Hamiltonian structure by a simple change of unknowns. In general, this Hamiltonian does not have a sign. We therefore impose conditions that are classical in naval engineering on the position of the center of mass and of the so-called meta-centers, in order to ensure that the Hamiltonian is positive, even if it remains non-definite. Under such conditions, the Hamiltonian can be used to define a semi-normed functional space adapted to the problem.

	\subsection{The Dirichlet-Neumann operator and the Kirchhoff potentials}\label{sectKirchhoff}
	
	The homogeneous Sobolev space $\dot{H}^1(\Omega)$ is a natural functional space to work with since velocity potentials that correspond to velocity fields of finite kinetic energy are in this space. We recall that it is defined as
	$$
	\dot{H}^1(\Omega)=\{\phi \in L^1_{\rm loc}(\Omega), \nabla\phi\in L^2(\Omega)^2 \},
	$$
	with associated semi-norm $\Vert \phi\Vert_{\dot{H}^1(\Omega)}=\Vert \nabla\phi\Vert_{L^2(\Omega)}$. The space $\dot{H}^{1/2}(\Gamma^{\rm D})$ corresponding to the traces on $\Gamma^{\rm D}$ of all functions in $\dot{H}^1(\Omega)$ was introduced in \cite{LannesMing24}, where its is proved that the mapping
	\begin{equation}\label{tracemapping}
		\mbox{Tr}^{\rm D}: \phi\in \dot{H}^1(\Omega) \mapsto \phi_{\vert_{{\Gamma}^{\rm D}}}\mbox{ on }\dot{H}^{1/2}(\Gamma^{\rm D})
	\end{equation}
	is well defined, continuous, onto, and admits a continuous right-inverse. It is not necessary to characterize further this functional space at this step; this will be done later in Section \ref{Sec: Wp John and just Cummins}. We can now define the Dirichlet-Neumann operator.
	\begin{Def}
		Let $\Omega$, $\Gamma^{\rm D}$ and $\Gamma^{\rm N}$ be as in Assumption \ref{Assumption domain}. The Dirichlet-Neumann operator $G_0$ is defined as
		\begin{equation*}
			G_0
			\:
			:
			\:
			\begin{cases}
				\dot{H}^{1/2}(\Gamma^{\mathrm{D}}) & \rightarrow   (\dot{H}^{1/2}(\Gamma^{\mathrm{D}}) )'
				\\
				\psi  &  \mapsto \partial_z  \psi^{\mathfrak{h}}|_{\Gamma^{\rm D}},
			\end{cases} 
		\end{equation*}
		where $(\dot{H}^{1/2}(\Gamma^{\mathrm{D}}) )'$ is the dual of $\dot{H}^{1/2}(\Gamma^{\mathrm{D}})$ and  $\psi^{\mathfrak{h}} \in \dot{H}^1(\Omega)$ is the unique variational solution of
		\begin{equation}\label{Eq: psi^H}
			\begin{cases}
				\Delta \psi^{\mathfrak{h}}= 0 \qquad\hspace{0.3cm} \text{in} \quad \Omega 
				\\ 
				\psi^{\mathfrak{h}} = \psi \qquad \hspace{0.5cm}\text{on} \quad \Gamma^{\rm D}
				\\ 
				\partial_n \psi^{\mathfrak{h}} = 0 \qquad \hspace{0.2cm}\text{on} \quad \Gamma^{\rm N}.
			\end{cases}
		\end{equation}
	\end{Def}
	\begin{remark}
		The existence of a variational solution $\psi^{\mathfrak{h}}\in \dot{H}^1(\Omega)$ to \eqref{Eq: psi^H} for $\psi \in \dot{H}^{1/2}(\Gamma^{\rm D})$ is given by Proposition 9 in \cite{LannesMing24}. Also, the definition of the normal derivative $\partial_z\psi^{{\mathfrak h}}$ on $\Gamma^{\rm D}$ should be understood in its weak variational meaning \cite{LionsMagenes73}. 
	\end{remark}
	We now turn to define the Kirchhoff potentials and the added mass matrix.
	\begin{Def}\label{defKirchhoff}
		Let $\Omega$, $\Gamma^{\rm D}$ and $\Gamma^{\rm N}$ be as in Assumption \ref{Assumption domain} and recall that ${\boldsymbol \kappa}=(\kappa_1,\kappa_2,\kappa_3)^{\rm T}\in {\mathbb R}^3$ is defined in \eqref{defkappaC}. \\
		{\bf i.} We write ${\bf K}=(K_1,K_2,K_3)^{\rm T}$, where $K_j$, $1\leq j\leq 3$ are the \emph{Kirchhoff potentials} defined as the unique solution in $H^1(\Omega)$ of the boundary value problem
		\begin{equation}\label{Eq: phi_w}
			\begin{cases}
				\Delta K_{j} = 0 & \text{in} \quad \Omega, 
				\\ 
				K_{j} = 0 & \text{on} \quad \Gamma^{\rm D},
				\\ \mathbf{n}^{\rm w} \cdot \nabla  K_{j}  = \kappa_j      & \text{on} \quad \Gamma^{\mathrm{w}},
				\\
				\mathbf{n}^{\rm b} \cdot \nabla K_j = 0 & \text{on} \quad \Gamma^{\rm b}.
			\end{cases}
		\end{equation}
		{\bf ii.} The \emph{added mass} matrix ${\mathcal M}_{\rm a}$ is defined as the $3\times 3$ matrix with entries $ \rho\int_\Omega \nabla K_i\cdot \nabla K_j$, for $1\leq i,j\leq 3$.
	\end{Def}   
	
	Using these two definitions, we can decompose the solution $\phi$ to the boundary value problem \eqref{Eq: Laplace equation} with ${\mathcal V}^{\rm w}$ derived in Appendix \ref{AppderN} equation \eqref{eqUsol0}, namely,
	$$
	{\mathcal V}^{\rm w}=(\dot{\tilde{x}},\dot{\tilde{z}})^{\rm T}-\dot{\theta}r_{G_{\rm eq}}^\perp
	$$
	to write the potential under the form
	\begin{equation}\label{decompphi}
		\phi=\psi^{\mathfrak h}+{\mathtt V}\cdot {\bf K},
	\end{equation}
	where we recall that  
	${\mathtt V}(t)=(\dot{\widetilde{x}}(t),\dot{\widetilde{z}}(t),\dot\theta(t))^{\rm T}$.
	
	\subsection{Reformulation of Fritz John's model}\label{sectreform}

	According to the decomposition \eqref{decompphi}, the quantity $\rho\int_{\Gamma^{\rm w}}\partial_t \phi {\boldsymbol \kappa}$ that appears in Newton's equation \eqref{Eq: Newtons law's of motion momentum} can be written
	$$
	\rho \int_{\Gamma^{\rm w}}\partial_t \phi {\boldsymbol \kappa}=  \rho\int_{\Gamma^{\rm w}}\partial_t \psi^{\mathfrak h} {\boldsymbol \kappa}
	+  \rho\int_{\Gamma^{\rm w}}(\dot{\mathtt V}\cdot {\bf K}) {\boldsymbol \kappa}.
	$$
	Recalling that $\kappa_j={\bf n}^{\rm w}\cdot \nabla K_j$ on $\Gamma^{\rm w}$, it follows from Green's identity and the fact that $\partial_t \psi^{\mathfrak h}=\partial_t\psi=-{\mathtt g}\zeta$ on $\Gamma^{\rm D}$ that
	$$
	\rho \int_{\Gamma^{\rm w}}\partial_t \phi {\boldsymbol \kappa}= \rho {\mathtt g} \int_{\Gamma^{\rm D}}\zeta  \partial_z {\bf K}
	+{\mathcal M}_{\rm a} \dot{\mathtt V};
	$$
	the second component corresponds to the added mass effect, and together with \eqref{decompphi}, it allows us to reformulate Fritz John's model \eqref{FJ1}-\eqref{Eq: Newtons law's of motion momentum}  under the form
	\begin{equation}\label{FJ2}
		\begin{cases}
			\partial_t \zeta - G_0\psi = {\mathtt V}\cdot \partial_z{\bf K} &\mbox{ on } \Gamma^{\rm D}\\
			\partial_t \psi + {\mathtt g}\zeta =0 &\mbox{ on } \Gamma^{\rm D},
		\end{cases}
	\end{equation}
	with ${\mathtt V}$ solving the reformulated Newton's equations
	\begin{equation}\label{Newton2}
		\begin{cases}
			{\mathtt X}'&={\mathtt V}\\
			({\mathcal M}+{\mathcal M}_{\rm a})\dot{\mathtt V}&=-\rho{\mathtt g} \int_{\Gamma^{\rm D}} \zeta \partial_z {\bf K}-{\mathcal C}{\mathtt X},
		\end{cases}
	\end{equation}
	where we recall that ${\mathtt X}=(\widetilde{x},\widetilde{z},\theta)^{\rm T}$, ${\mathtt V}=(\widetilde{x}',\widetilde{z}',\theta')^{\rm T}$,  ${\mathcal M}=\mbox{diag}({\mathfrak m},{\mathfrak m},{\mathfrak i})$, while ${\boldsymbol \kappa}$ and ${\mathcal C}$ are given by \eqref{defkappaC} and ${\mathcal M}_{\rm a}$ by Definition \ref{defKirchhoff}. The system of equations \eqref{FJ2}-\eqref{Newton2} is the reformulation of F. John's model that we work with in this article.

	\begin{remark}\label{remextforce}
		It is straightforward to generalize our analysis to the case where an external force and torque ${\mathcal F}^{\rm ext}=((F^{\rm ext})^{\rm T},T^{\rm ext})^{\rm T}$ is applied to the system. The second equation in \eqref{Newton2} must then be replaced by
		$$
		({\mathcal M}+{\mathcal M}_{\rm a})\dot{\mathtt V}=-\rho{\mathtt g} \int_{\Gamma^{\rm D}} \zeta \partial_z {\bf K}-{\mathcal C}{\mathtt X}+{\mathcal F}^{\rm ext}.
		$$
	\end{remark}   
	
	\subsection{Hamiltonian structures}\label{Subsec: Ham struc}
	
	Let us introduce the Hamiltonian ${\mathcal H}(\zeta,{\mathtt X},\psi,{\mathtt V})$ defined as
	\begin{equation}\label{defHam}
		{\mathcal H}(\zeta,\psi,{\mathtt X},{\mathtt V})=\frac{1}{2} {\mathtt g}\int_{\Gamma^{\rm D}}\zeta^2+\frac{1}{2} \int_{\Gamma^{\rm D}} \psi G_0\psi+\frac{1}{2\rho}{\mathcal C} {\mathtt X}\cdot  {\mathtt X}+\frac{1}{2\rho}({\mathcal M}+{\mathcal M}_{\rm a}){\mathtt V}\cdot{\mathtt V};
	\end{equation}
	here, we write by abuse of notation $ \int_{\Gamma^{\rm D}} \psi G_0\psi= \langle \psi, G_0\psi\rangle_{\dot{H}^{1/2}(\Gamma^{\rm D})\times \dot{H}^{1/2}(\Gamma^{\rm D})' }$. The mapping ${\mathcal H}$ is well defined and ${C}^1$ on $L^2(\Gamma^{\rm D})\times {\mathbb R}^3\times \dot{H}^{1/2}(\Gamma^{\rm D})\times {\mathbb R}^3$ and one readily checks that its variational gradient with respect to $(\zeta,{\mathtt X},\psi,{\mathtt V})$ is given in $L^2(\GD)\times \RR^3\times (\dot{H}^{1/2}(\GD))'\times \RR^3$ by
	$$
	\nabla {\mathcal H} {(\zeta,{\mathtt X},\psi,{\mathtt V})}=
	\begin{pmatrix}
		{\mathtt g}\zeta \\
		\frac{1}{\rho}{\mathcal C}{\mathtt X} \\
		G_0\psi \\
		\frac{1}{\rho}({\mathcal M}+{\mathcal M}_{\rm a}){\mathtt V}
	\end{pmatrix}.
	$$
	In order to show the (non-canonical) Hamiltonian structure of the equations \eqref{FJ2}-\eqref{Newton2}, we need to introduce the skew-symmetric matrix $J$ defined as
	$$
	J=\begin{pmatrix}
		{\bf 0}_{4\times 4} & -J_1 \\
		J_1^* & {\bf 0}_{4\times 4}
	\end{pmatrix}
	$$
	with
	$$
	J_1=\begin{pmatrix}
		1 & \rho (\partial_z {\bf K})^{\rm T}({\mathcal M}+{\mathcal M}_{\rm a})^{-1} \\
		{\bf 0}_{3\times 1} & \rho ({\mathcal M}+{\mathcal M}_{\rm a})^{-1}
	\end{pmatrix} 
	$$
	and 
	$$
	J_1^*=\begin{pmatrix}
		1 &  {\bf 0}_{1\times 3} \\
		\rho ({\mathcal M}+{\mathcal M}_{\rm a})^{-1} \ell_{\partial_z {\bf K}}& \rho ({\mathcal M}+{\mathcal M}_{\rm a})^{-1}
	\end{pmatrix},
	$$
	where $\ell_{\partial_z {\bf K}}:L^2(\GD)\to \RR^3$ is defined as $\ell_{\partial_z {\bf K}}\zeta=\int_{\GD}\zeta \dz{\bf K}$.
	It is now straightforward to check the Hamiltonian structure of the equations.
	\begin{prop}
		The equations \eqref{FJ2}-\eqref{Newton2} can be equivalently written under the form
		$$
		\partial t \begin{pmatrix}  \zeta \\ {\mathtt X} \\ \psi  \\ {\mathtt V} \end{pmatrix}
		+J \nabla {\mathcal H} {(\zeta,{\mathtt X},\psi, {\mathtt V})}=0;
		$$
		in particular, for regular solutions, the Hamiltonian ${\mathcal H}$ is a conserved quantity.
	\end{prop}
	\begin{remark}\label{remstab}
		The Hamiltonian ${\mathcal H}$ corresponds to the total energy of the fluid-solid system in the linear approximation considered here. Its first two components are obviously positive; the last one is also positive since ${\mathcal M}$ is a definite positive matrix and ${\mathcal M}_{\rm a}$ is a Gram-matrix and is therefore symmetric positive. As seen in Remark \ref{analC}, the third component is not always positive. However, under the condition \begin{equation}\label{stabcond}
			c_{zz}c_{\theta \theta}-c_{z\theta}^2>0,
		\end{equation}
		one has ${\mathtt X}\cdot {\mathcal C}{\mathtt X}\geq c_0 ( \widetilde{z}^2 + \theta^2)$ for some $c_0>0$. From the conservation of the Hamiltonian for regular solutions, it then follows that $\vert \zeta \vert_{L^2(\Gamma^{\rm D})}$, $\vert \psi \vert_{\dot{H}^{1/2}}$, $\widetilde{z}$, $\theta$ and ${\mathtt V}$ remain bounded for all times. This is not necessarily the case if \eqref{stabcond} is not satisfied. Condition \eqref{stabcond} corresponds to the stability condition in the classical static stability of ships \cite{Moore67}. Even under this condition, the conservation of the Hamiltonian does not imply a control of the vertical displacement $\tilde{x}$.
	\end{remark}
	
	It was remarked in \cite{Daalen} with formal variational arguments that the wave-structure interactions should have a canonical Hamiltonian structure. This can be deduced from the above analysis for Fritz John's model. Indeed, since $-{\mathtt g}\zeta=\partial_t \psi$ on $\Gamma^{\rm D}$, the second equation in \eqref{Newton2} can be written
	$$
	\big[ ({\mathcal M}+{\mathcal M}_{\rm a}){\mathtt V}+\rho \int_{\Gamma^{\rm D}}\psi \partial_z{\bf K}\big]'=-{\mathcal C}{\mathtt X},
	$$
	which motivates the introduction of the variable ${\mathtt W}:= ({\mathcal M}+{\mathcal M}_{\rm a}){\mathtt V}+\rho \int_{\Gamma^{\rm D}}\psi \partial_z{\bf K}$. Using this variable instead of ${\mathtt V}$, the model \eqref{FJ2}-\eqref{Newton2} can written
	\eqref{FJ1}-\eqref{Eq: Newtons law's of motion momentum}  under the form
	\begin{equation}\label{FJ3}
		\begin{cases}
			\partial_t \zeta - G_0\psi = ({\mathcal M}+{\mathcal M}_{\rm a})^{-1}\big({\mathtt W}-\rho \int_{\Gamma^{\rm D}}\psi \partial_z{\bf K}\big)\cdot \partial_z{\bf K} &\mbox{ on } \Gamma^{\rm D},\\
			\partial_t \psi + {\mathtt g}\zeta =0 &\mbox{ on } \Gamma^{\rm D}
		\end{cases}
	\end{equation}
	with ${\mathtt V}$ solving the reformulated Newton's equations
	\begin{equation}\label{Newton3}
		\begin{cases}
			\dot{\mathtt X}&=({\mathcal M}+{\mathcal M}_{\rm a})^{-1}\big({\mathtt W}-\rho \int_{\Gamma^{\rm D}}\psi \partial_z{\bf K}\big)\\
			\dot{\mathtt W} &=-{\mathcal C}{\mathtt X}.
		\end{cases}
	\end{equation}
	The following proposition shows that \eqref{FJ3}-\eqref{Newton3} has a canonical Hamiltonian structure with canonical conjugate variables $(\zeta,{\mathtt X}^{\rm T})^{\rm T}$ and $(\psi,{\mathtt W}^{\rm T})^{\rm T}$ for the Hamiltonian $\widetilde{\mathcal H}(\zeta,{\mathtt X},\psi,{\mathtt W})$ deduced from 
	${\mathcal H}(\zeta,{\mathtt X},\psi,{\mathtt V})$ by replacing ${\mathtt V}=({\mathcal M}+{\mathcal M}_{\rm a})^{-1}\big({\mathtt W}-\rho \int_{\Gamma^{\rm D}}\psi \partial_z{\bf K}\big)$,
	\begin{align*}
		\widetilde{\mathcal H}(\zeta,{\mathtt X},\psi,{\mathtt W})=&
		\frac{1}{2} {\mathtt g}\int_{\Gamma^{\rm D}}\zeta^2+ \frac{1}{2} \int_{\Gamma^{\rm D}} \psi G_0\psi+\frac{1}{2\rho}{\mathcal C} {\mathtt X}\cdot  {\mathtt X}\\
		&+\frac{1}{2\rho}({\mathcal M}+{\mathcal M}_{\rm a})^{-1}  \big({\mathtt W}-\rho \int_{\Gamma^{\rm D}}\psi \partial_z{\bf K}\big)\cdot \big({\mathtt W}-\rho \int_{\Gamma^{\rm D}}\psi \partial_z{\bf K}\big).
	\end{align*}
	\begin{prop}
		The equations \eqref{FJ3}-\eqref{Newton3} can be equivalently written under the form
		$$
		\partial t \begin{pmatrix}  \zeta \\ {\mathtt X}  \\ \psi \\ {\mathtt W} \end{pmatrix}
		+ \begin{pmatrix} {\bf 0}_{4\times 4} & -{\rm Id}_{4\times 4} \\
			{\rm Id}_{4\times 4} &  {\bf 0}_{4\times 4}
		\end{pmatrix}
		\nabla \widetilde{\mathcal H} {(\zeta,{\mathtt X},\psi,{\mathtt W})}=0.
		$$
		in particular, for regular solutions, the Hamiltonian $\widetilde{\mathcal H}$ is a conserved quantity.
	\end{prop}
	\begin{remark}
		{\bf i.} In \cite{Daalen}, the canonical variable ${\mathtt W}$ is defined as ${\mathtt W}={\mathcal M}{\mathtt V}+\int_{\Gamma^{\rm w}}(\partial_t \phi){\boldsymbol \kappa}$; using the same kind of computations as for the derivation of \eqref{FJ2}-\eqref{Newton2}, one can check that these two definitions are equivalent.\\
		{\bf ii.} Though in canonical Hamiltonian form, the equations \eqref{FJ3}-\eqref{Newton3} are more complicated than \eqref{FJ2}-\eqref{Newton2}, and we therefore prefer the latter for the mathematical analysis.
	\end{remark}

	\section{Well-posedness of F. John's model in the energy space and Justification of Cummins equations} \label{Sec: Wp John and just Cummins}

	In this section, we introduce a semi-normed functional space based on the Hamiltonian introduced in the previous section; this functional setting is well adapted to study the well-posedness of  F. John's model in Subsection \ref{Subsec: Well-posedness of Fritz John's model}. Building on this result, we revisit the derivation made by Cummins \cite{Cummins62} in Subsection \ref{Subsec: Cummins} to derive the Cummins equations and establish the well-posedness of these equations.  However, before we turn to these results, we introduce the functional setting. 
	
	We have already mentioned the space $\dot{H}^{1/2}(\GD)$ in Section \ref{sectKirchhoff} as the range of the trace mapping  $\Phi\in \dot{H}^1(\Omega)\mapsto \Phi_{\vert_{\GD}}$. This space was characterized in \cite{LannesMing24}. To state this characterization and to define the higher regular spaces $\dot{H}^{s+1/2}(\GD)$, for $s\geq 0$, we first need to introduce the space $\dot{H}^{r}(I)$, for $r\geq 0$ and $I$ a finite interval or a half-line. 
	\begin{Def}\label{Def: Hom space I}
		Let $I$ be a finite interval or a half-line.\\
		{\bf i.}  We define $\widetilde{L}_{\rm loc}^1(I) = \{f \in L_{\rm loc}^1(I) \: : \: \exists \tilde{f} \in L_{\rm loc}^1(I), \tilde{f}_{|_{I}} = f \}$.\\
		{\bf ii.} For $0<r<1$, we define $\dot{H}^r(I)$ as
		$$
		\dot{H}^{r}(I)=\{f \in \widetilde{L}_{\rm loc}^1(I), | f|_{\dot{H}^r(I)}< \infty \}
		$$
		with
		$$
		| f|_{\dot{H}^r(I)}^2=
		\begin{cases}
			\displaystyle \int_I \int_I\frac{|  f(x)-  f(y)|^2}{|x-y|^{2r+1}}{\rm d}x{\rm d}y &\mbox{ if } I \mbox{ is a finite interval},\\
			\displaystyle \int_I\int_{I\cap B(y,1)}\frac{|  f(x)-  f(y)|^2}{|x-y|^{2r+1}}{\rm d}x{\rm d}y &\mbox{ if } I \mbox{ is a half-line}.
		\end{cases}
		$$
		{\bf iii.} For all $r \geq 1$, we just take $| f|_{\dot{H}^r(I)}=|\partial_x f\vert_{H^{r-1}(I)}$.
	\end{Def}
	
	\begin{remark}\label{exprequivnorm}
		{\bf i.} When $I$ is a half-line, then $|f|_{\dot{H}^r(I)}$ differs from the standard homogeneous semi-norm defined as $  |f|_{{H}_{\rm hom}^r(I)}^2=\int_I \int_I\frac{|  f(x)-  f(y)|^2}{|x-y|^{2r+1}}{\rm d}x{\rm d}y$. This is due to the screening condition $y\in I\cap B(x,1)$ in the second integral defining the semi-norm. We refer to \cite{Strichartz16,LeoniTice19,StevensonTice20} for an in depth analysis of screened Sobolev spaces.\\
		{\bf ii.} For all $r>0$, we can write $r=k+\sigma$, with $k\in {\mathbb N}$ and $0<\sigma<1$. The semi-norm of $\dot{H}^{r}(I)$ can be equivalently defined as
		$$
		| f|_{\dot{H}^r(I)}^2:=
		\begin{cases}
			\displaystyle\sum_{j=0}^{k-1} | \partial_x^{j+1}  f |_{L^2(I)}^2
			&\mbox{ if } \sigma=0,\\
			\displaystyle\sum_{j=0}^{k-1} | \partial_x^{j+1}  f |_{L^2(I)}^2+\vert \partial_x^k f \vert_{\dot{H}^\sigma(I)}& \mbox{ if } 0<\sigma<1.
		\end{cases}
		$$
	\end{remark}
	Following \cite{LannesMing24}, we also define such homogeneous spaces on domains with various connected components, such as $\Gamma^{\rm D}$. The important fact in the definition below is that the semi-norm associated with $\dot{H}^r(\Gamma^{\rm D})$ relates all the connected components through the constants $\overline{f}_j$. Note also that the choice of the finite interval to define the average when ${\mathcal E}$ is a half-line is of no importance, as shown in \cite{LannesMing24}.
	\begin{Def}\label{defGN}
		Let $\Omega$ be as in Assumption \ref{Assumption domain}; with the same notations, we identify $\Gamma^{\rm D}$ with ${\mathcal E}_-\cup {\mathcal E}_+$.  For  $f\in \widetilde{L}^1_{\rm loc}(\Gamma^{\rm D})$, we denote by $\overline{f}_\pm$ the average of $f$ over ${\mathcal E}_\pm$ if it is a finite interval; if ${\mathcal E}_\pm$ is a half-line, then $\overline{f}_\pm$ denotes the average over any non-empty finite interval contained in ${\mathcal E}_\pm$. 
		For all $s\geq 0$, we then define
		$$
		\dot{H}^{s+1/2}(\Gamma^{\rm D})=\{f\in  \widetilde{L}^1_{\rm loc}(\Gamma^{\rm D}),  \quad f_{\vert_{{\mathcal E}_\pm}}\in \dot{H}^{s+1/2}({\mathcal E}_\pm)\},
		$$
		endowed with the semi-norm
		$$
		\vert f\vert_{\dot{H}^{s+1/2}(\Gamma^{\rm D})}=\vert f\vert_{\dot{H}^{s+1/2}({\mathcal E}_-)}
		+\vert f\vert_{\dot{H}^{s+1/2}({\mathcal E}_+)}+
		\vert \overline{f}_{+}-\overline{f}_-\vert.
		$$
	\end{Def}	
	
	We now show that we may view F. John's model as the evolution of a skew-symmetric operator in this functional setting.
	
	\subsection{Abstract reformulation of F. John's model}\label{Subsec: Abstract reformulations}
	We show here that there exists a semi--Hilbert space ${\mathbb X}$ such that F. John's model \eqref{FJ2}-\eqref{Newton2} can be reformulated in compact form as
	\begin{equation}\label{FJabstract}
		\partial_t {\bf U}+{\bf A}{\bf U}=0
	\end{equation}
	with ${\bf U}=(\zeta,{\mathtt X}^{\rm T},\psi,{\mathtt V}^{\rm T})^{\rm T}\in {\mathbb X}$ and ${\bf A}$ an unbounded skew-symmetric operator on ${\mathbb X}$. We define ${\mathbb X}$ as
	\begin{equation}\label{defX}
		{\mathbb X}=L^2(\Gamma^{\rm D})\times {\mathbb R}^3\times \dot{H}^{1/2}(\Gamma^{\rm D})\times {\mathbb R}^3,
	\end{equation}
	endowed with the semi-scalar product
	$$
	\qquad \langle {\bf U}_1 , {\bf U}_2 \rangle_{\mathbb X}=
	\frac{1}{2} {\mathtt g}\int_{\Gamma^{\rm D}} \zeta_1 \zeta_2
	+ \frac{1}{2} {\mathtt g}\int_{\Gamma^{\rm D}}  \psi_1 G_0\psi_2 
	+\frac{1}{2\rho}{\mathcal C}{\mathtt X}_1\cdot {\mathtt X}_2
	+\frac{1}{2\rho}({\mathcal M}+{\mathcal M}_{\rm a}){\mathtt V}_1\cdot {\mathtt V}_2,
	$$
	for all  ${\bf U}_1, {\bf U}_2 \in {\mathbb X}$, and we write $\Vert \cdot \Vert_{\mathbb X}$ the associated semi-norm;  in particular, for all ${\bf U}\in {\mathbb X}$, one has  $\Vert {\bf U} \Vert_{\mathbb X}^2={\mathcal H}({\bf U})$, with ${\mathcal H}$ the Hamiltonian defined in \eqref{defHam}. According to the comments made in Remark \ref{remstab}, the following assumption is needed to ensure that $\Vert {\bf U} \Vert_{\mathbb X}^2$ is always positive.
	\begin{assumption}\label{hypStab}
		The coefficients of the matrix ${\mathcal C}$ defined in \eqref{defkappaC} are such that \eqref{stabcond} is satisfied, namely,
		$$
		c_{zz}c_{\theta \theta}-c_{z\theta}^2>0.
		$$
	\end{assumption}
	\begin{remark}\label{Remark: semi norm}
		Even under this assumption, $\Vert \cdot \Vert_{\mathbb X}$ is only a semi-norm; indeed, $\Vert {\bf U}\Vert_{\mathbb X}=0$ if and only if there are constants $C_1$ and $C_2$ such that $\psi=C_1$ and $\widetilde{x}=C_2$.
	\end{remark}
	We also define the operator ${\bf A}$ as
	\begin{equation}\label{defA}
		{\bf A}=\begin{pmatrix}
			{\bf 0}_{4\times 4} & A_{1} \\
			A_2 & {\bf 0}_{4\times 4}
		\end{pmatrix}
	\end{equation}
	with
	$$
	A_1=\begin{pmatrix}
		-G_0 & -\partial_z {\bf K}^{\rm T}  \\
		{\bf 0}_{3\times 1} & -{\rm Id}_{3\times 3}
	\end{pmatrix}
	\quad\mbox{ and }\quad
	A_2=\begin{pmatrix}
		{\mathtt g} & {\bf 0}_{1\times 3}  \\
		({\mathcal M}+{\mathcal M}_{\rm a})^{-1}\ell_{\partial_z {\bf K}} & ({\mathcal M}+{\mathcal M}_{\rm a})^{-1}{\mathcal C}
	\end{pmatrix},
	$$
	where we recall that $\ell_{\dz{\bf K}}\zeta=\int_{\GD} \zeta \dz {\bf K}$.
	The following proposition shows that ${\bf A}$ is skew-symmetric for the scalar product associated with ${\mathbb X}$.
	\begin{prop}\label{Prop: Skew-symmetry}
		Let ${\bf A}$ be as defined in \eqref{defA}.\\
		{\bf i. } The equations \eqref{FJ2}-\eqref{Newton2} can be written under the abstract form \eqref{FJabstract}.\\
		{\bf ii.} The domain of ${\bf A}$ defined as $D({\bf A})=\{ {\bf U}\in {\mathbb X}, {\bf A}{\bf U}\in {\mathbb X}\}$ is given by
		$$
		D({\bf A})=    {H}^{1/2}(\Gamma^{\rm D})\times {\mathbb R}^3  \times \dot{H}^1(\Gamma^{\rm D})\times {\mathbb R}^3.
		$$
		{\bf iii.} The operator ${\bf A}$ is skew-symmetric for the scalar product associated with ${\mathbb X}$, that is, 
		$$
		\forall {\bf U}_1, {\bf U}_2 \in {D}({\bf A}), \qquad
		\langle {\bf A}{\bf U}_1,{\bf U}_2 \rangle_{\mathbb X}=
		-\langle {\bf U}_1, {\bf A}{\bf U}_2 \rangle_{\mathbb X}.
		$$
	\end{prop}
	\begin{proof}
		The first and third points follow from direct computations. For the second point, the only non-trivial things to check are that 
		\begin{itemize}
			\item(i) If $\zeta\in L^2(\Gamma^{\rm D})$ then $\int_{\Gamma^{\rm D}} \zeta \partial_z {\bf K}$ is well defined.
			\item(ii) If $\psi\in \dot{H}^{1/2}(\Gamma^{\rm D})$ and ${\mathtt V}\in {\mathbb R}^3$ are such that $-G_0\psi -\partial_z{\bf K}\cdot {\mathtt V} \in L^2(\Gamma^{\rm D})$ then $\psi \in \dot{H}^{1}(\Gamma^{\rm D})$.
		\end{itemize}
		The first point follows immediately if we can prove that each component $\partial_zK_j$, $1\leq j\leq 3$ of $\partial_z{\bf K}$ is in $ L^2(\Gamma^{\rm D})$. Now, by definition, $K_j$ solves the mixed boundary value problem \eqref{Eq: phi_w}, with smooth Neumann data on $\Gamma^{\rm N}$ and homogeneous Dirichlet data on $\Gamma^{\rm D}$. Since the angles at the corners between $\Gamma^{\rm N}$ and $\Gamma^{\rm D}$ are assumed to be strictly smaller than $\pi$, we know by the elliptic regularity theory on corner domains (see for instance \cite{Dauge88}) that $K_j\in H^{3/2+\epsilon}(\Omega)$ for some $\epsilon>0$. It follows from the trace theorem that $\partial_z K_j \in H^\epsilon(\Gamma^{\rm D})\subset L^2((\Gamma^{\rm D}))$, which proves the claim and therefore point (i).		
		
		Since we now know that $\partial_z{\bf K}\cdot {\mathtt V}\in L^2(\Gamma^{\rm D})$, we are left to prove that if $\psi\in \dot{H}^{1/2}(\Gamma^{\rm D})$
		is such that $G_0\psi \in L^2(\Gamma^{\rm D})$ then $\psi\in \dot{H}^1(\Gamma^{\rm D})$. This was proved using Rellich's identities in Proposition 3.8 of \cite{LannesMing24}.
	\end{proof}

	\subsection{Well-posedness of Fritz John's model in the energy space}
	\label{Subsec: Well-posedness of Fritz John's model}
	We now apply the theory from the previous section to demonstrate the well-posedness of the floating body problem described by \eqref{FJabstract} within the energy space defined by \eqref{defX}.

	The following theorem proves the well-posedness in ${\mathbb X}$ of the initial value problem associated with \eqref{FJabstract}, with further time regularity if the initial data is in $D({\bf A})$, where we recall that by Proposition \ref{Prop: Skew-symmetry}, one has $D({\bf A})=    {H}^{1/2}(\Gamma^{\rm D})\times {\mathbb R}^3  \times \dot{H}^1(\Gamma^{\rm D})\times {\mathbb R}^3$.	
	\begin{thm}\label{thm: Wp energy space}
		Let $\Omega$ and $\Gamma$ be as in Assumption \ref{Assumption domain}, $\mathbb{X}$  defined by \eqref{defX}, and suppose the coefficients of $\mathcal{C}$ satisfy Assumption \ref{hypStab}. Then for any $\mathbf{U}^{\rm in} = (\zeta^{\rm in}, ({\mathtt X}^{\rm in})^{\rm T},\psi^{\rm in},({\mathtt V}^{\rm in})^{\rm T})^{\rm T}  \in \mathbb{X}$, there is a unique solution  
		\begin{equation*}
			\mathbf{U} = (\zeta, {\mathtt X}^{\rm T}, \psi, {\mathtt V}^{\rm T})^{\rm T} \in C(\R_+;\mathbb{X}) 
		\end{equation*} 
		to \eqref{FJabstract} with initial condition ${\bf U}(0)={\bf U}^{\rm in}$.
		Moreover, we have that ${\mathtt V}=\dot{X} \in  C^1(\R_+; \R^3)$, and the solution conserves the energy
		\begin{equation}\label{Eq: Energy est}
			\forall \: t \geq 0 \quad   \Vert {\bf U}(t)\Vert_{\mathbb X}=\Vert {\bf U}^{\rm in}\Vert_{\mathbb X}.
		\end{equation}
		If in addition ${\bf U}^{\rm in}\in D({\mathbf A})$, then  ${\bf U}\in C^1(\R_+;\mathbb{X})$ and ${\mathtt V}=\dot{X} \in C^2(\R_+; \R^3)$. 
	\end{thm}
	\begin{remark}
		{\bf i.} We recall that the conserved energy coincides with the Hamiltonian, that is, $\mathcal{H}(\mathbf{U})(t) = \Vert {\bf U}(t)\Vert_{\mathbb X}^2$.\\
		{\bf ii.} By Duhamel's formula, the theorem can be extended to the non-homogeneous case
		\begin{equation}\label{FJnonhom}
			\begin{cases}
				\partial_t {\bf U}+{\bf A}{\bf U}={\bf F},\\
				{\bf U}_{\vert_{t=0}}={\bf U}^{\rm in}.
			\end{cases}
		\end{equation}
		with ${\bf F}\in C({\mathbb R}_+;{\mathbb X})$; the solution then satisfies the energy estimate
		$$
		\Vert {\bf U}(t)\Vert_{\mathbb X}\leq \Vert {\bf U}^{\rm in}\Vert_{\mathbb X}+\int_0^T \Vert {\bf F}(t)\Vert_{\mathbb X}\:{\rm d}t.
		$$
		This allows for instance to cover the case where an external force is applied to the floating object, see Remark \ref{remextforce}.
	\end{remark} 
	
	\begin{proof}
		The proof is inspired by Theorem 4 in \cite{LannesMing24}; the main difference is that the adherence of zero in the semi-normed space ${\mathbb X}$ is isomorphic to $\RR^2$ here, while it was isomorphic to $\RR$ in \cite{LannesMing24}. This means that there is an underdeterminacy of one additional constant, corresponding to horizontal translations. We mainly focus on this aspect of the proof and only sketch the other points: a priori estimates and a duality argument for the construction of a weak solution, and an approximation argument to deduce the regularity of the solutions. The first step is provided by the lemma, which shows that regular enough solutions satisfy an energy estimate; such solutions are called strong solutions.
		\begin{lemma}
			For any $T>0$, there exists $C>0$ such that for all $\mathbf{V} \in H^1([0,T];\mathbb{X})\cap L^2([0,T];D({\bf A}))$ there holds, for all $0\leq t \leq T$,
			\begin{equation}\label{Eq: A priori estimate}
				|\mathbf{V}(t)|_{\mathbb{X}}^2 
				\leq C\big(
				|\mathbf{V}(0)|_{\mathbb{X}}^2 
				+
				\int_0^t 	|(\partial_t +\mathbf{A})\mathbf{V}(s)|_{\mathbb{X}}^2 \: \mathrm{d}s\big).
			\end{equation}
			Moreover, if $(\partial_t +\mathbf{A})\mathbf{V}=0$ then $|\mathbf{V}(t)|_{\mathbb{X}}=|\mathbf{V}(0)|_{\mathbb{X}}$ for all $0\leq t \leq T$.
		\end{lemma}
		
		\begin{proof}
			Since ${\bf V}$ has enough regularity to use the skew-symmetry of the operator ${\bf A}$, this is a direct consequence of the third point of  Proposition \ref{Prop: Skew-symmetry}.
		\end{proof}
		The next step is to construct a weak solution to the non-homogeneous initial boundary value problem \eqref{FJnonhom}. This is where we have to lift the underdeterminacy related to horizontal translations.	
		\begin{lemma}\label{Lemma: weak sol}
			Let $T>0$, ${\bf F}\in L^2([0,T];{\mathbb X})$ and ${\bf U}^{\rm in} \in  {\mathbb X}$. There exist a weak solution ${\bf U} \in L^2([0,T]; {\mathbb X})$ to  \eqref{FJnonhom}, which also solves \eqref{FJnonhom} almost everywhere.
		\end{lemma}
		\begin{proof}
			The space $L^2([0,T];{\mathbb X})$ is not a Hilbert space. The adherence $K$ of zero for the $L^2([0,T];{\mathbb X})$ semi-norm consists, by Remark \ref{Remark: semi norm}, of all $(\zeta,{\mathtt X}^{\rm T}, \psi, {\mathtt V}^{\rm T})^{\rm T}\in L^2([0,T];{\mathbb X})$ such that $\zeta, \widetilde{z},\theta,{\mathtt V}$ vanish identically, while  $\psi=C_1$ and $\widetilde{x}=C_2$ with $C_1,C_2\in L^2([0,T];{\mathbb R})$. This means that if ${\bf V}_1,{\bf V}_2\in L^2([0,T];{\mathbb X})$ are such that $\Vert {\bf V}_1-{\bf V}_2\Vert_{\mathbb X}=0$ then ${\bf V}_1={\bf V}_2$ up to two functions of time, while in the case of a fixed object considered in \cite{LannesMing24} it is up to only one function of time. The quotient space $L^2([0,T];{\mathbb X})/K$ is a Hilbert space and, as in Lemma 9 in \cite{LannesMing24}, one can use a duality method based on the energy estimate of Lemma \ref{Eq: A priori estimate} to obtain that there exists $\widetilde{\bf U}\in L^2([0,T]; {\mathbb X})$ such that for all ${\bf V} \in H^1([0,T];{\mathbb X})\cap L^2([0,T];D({\bf A}))$ such that ${\bf V}(T)=0$, there holds
			$$
			\int_0^T \langle {\bf F}(t),{\bf V}(t)\rangle_{\mathbb X}{\rm d}t +\langle {\bf U}^{\rm in},{\bf V}(0)\rangle_{{\mathbb X}}=
			\int_0^T \langle \widetilde{\bf U}(t),(-\partial_t -{\bf A}){\bf V}(t)\rangle_{{\mathbb X}}{\rm d}t.
			$$
			However, since $L^2([0,T];{\mathbb X})$ is only semi-normed, this does not imply that $\widetilde{U}$ solves \eqref{FJnonhom} almost everywhere; we only get that there is an element ${\bf k}\in K$ such that
			$$
			\partial_t \widetilde{\bf U}+{\bf A}\widetilde{\bf U}={\bf F}+{\bf k}
			$$
			almost everywhere. Remark now that for all $\widetilde{\bf k}\in K$, one has ${\bf A}\widetilde{\bf k}=0$ and therefore
			$$
			\partial_t (\widetilde{\bf U}+\widetilde{\bf k})+{\bf A}(\widetilde{\bf U}+\widetilde{\bf k})={\bf F}+\partial_t\widetilde{\bf k}+{\bf k}
			$$
			holds almost everywhere. Choosing $\widetilde{\bf k} = \widetilde{\bf k}(0)-\int_0^t{\bf k}(s){\rm d}s$, with $\widetilde{\bf k}(0)={\bf U}^{\rm in}-\widetilde{\bf U}(0)$ (which is constant since $\Vert {\bf U}^{\rm in}-\widetilde{\bf U}(0)\Vert_{\mathbb X}=0$), the function ${\bf U}=\widetilde{\bf U}+\widetilde{\bf k}$ is a weak solution satisfying the conditions of the lemma.  
		\end{proof}

		For the conclusion, one can show, using mollifiers in time as in Lemma 10 of \cite{LannesMing24}, that a weak solution furnished by Lemma \ref{Lemma: weak sol} is the limit in $C([0,T];{\mathbb X})$ of a sequence of strong solutions, from which one deduces that it is unique, belongs to $C([0,T];{\mathbb X})$, and satisfies the energy estimate. In order to prove that ${\mathtt V}$ belongs to $ C^1$ rather than $ C^0$, we use the second equation in \eqref{Newton2} to observe first that $\dot{\mathtt V}\in C^0$ and therefore ${\mathtt V}\in C^1$. 
		
		The proof of the additional regularity under the assumption that ${\bf U}^{\rm in}\in D({\bf A})$ is a particular case of Corollary \ref{correg} below, and we therefore refer to the proof of this corollary. 
		
	\end{proof}

	\subsection{Derivation and well-posedness  of the Cummins equations}\label{Subsec: Cummins}

	In this section, we rigorously derive the Cummins equations \cite{Cummins62} from John's problem. In doing so, we make explicit the excitation force associated with the wave field and prove the well-posedness of these integro-differential equations by using our result on Fritz John's model. To be precise, we  
	%
	%
	introduce the radiative potentials ${\boldsymbol \zeta}^{\rm rad}=(\zeta^{\rm rad}_1,\zeta^{\rm rad}_2,\zeta^{\rm rad}_3)^{\rm T} \in C(\R^+;L^2(\GD)^3)$ and ${\boldsymbol \psi}^{\rm rad}=(\psi^{\rm rad}_1,\psi^{\rm rad}_2,\psi^{\rm rad}_3)^{\rm T} \in C(\R^+;\dot{H}^{1/2}(\GD)^3)$ defined as the solution of the initial value problems
	$$ 
	\begin{cases}
		\partial_t \zeta^{\rm rad}_j -G_0\psi^{\rm rad}_j=0,\\
		\partial_t \psi^{\rm rad}_j+{\mathtt g}\zeta^{\rm rad}_j=0
	\end{cases}
	\quad\mbox{ and }\quad
	(\zeta^{\rm rad}_j,\psi^{\rm rad}_j)_{\vert_{t=0}}=(\partial_z K_j,0),
	$$
	for $j=1,2,3$ and where ${\bf K}=(K_1,K_2,K_3)^{\rm T}$ are the Kirchhoff potentials introduced in Definition \ref{defKirchhoff}. We note that since $\partial_z K_j\in L^2(\GD)$ (see the proof of Proposition \ref{Prop: Skew-symmetry}), the solution $(\zeta^{\rm rad}_j,\psi^{\rm rad}_j)$ to the above system is well defined and unique in $C(\R^+;L^2(\GD)\times \dot{H}^{1/2}(\GD))$ by Theorem 4.5 of \cite{LannesMing24}, whence our assertion that ${\boldsymbol \zeta}^{\rm rad}_j \in C(\R^+;L^2(\GD)^3)$. Similarly, for all $(\zeta^{\rm in},\psi^{\rm in})\in L^2(\GD)\times \dot{H}^{1/2}(\GD)$, we define the exciting wave field $(\zeta^{\rm exc},\psi^{\rm exc})\in C(\R^+; L^2(\Gamma^{\rm D})\times  \dot{H}^{1/2}(\Gamma^{\rm D}))$ as the solution to 
	$$
	\begin{cases}
		\partial_t \zeta^{\rm exc} -G_0\psi^{\rm exc}=0,\\
		\partial_t \psi^{\rm exc}+{\mathtt g}\zeta^{\rm exc}=0
	\end{cases}
	\quad\mbox{ with }\quad
	(\zeta^{\rm exc},\psi^{\rm exc})_{\vert_{t=0}}=(\zeta^{\rm in},\psi^{\rm in}).
	$$
	Consider now ${\bf U}=(\zeta,{\mathtt X},\psi,{\mathtt V})\in C(\R^+;{\mathbb X})$ to be the unique solution furnished by Theorem \ref{thm: Wp energy space} to the Cauchy problem associated with \eqref{FJ2}-\eqref{Newton2} and with initial data ${\bf U}^{\rm in}=(\zeta^{\rm in},{\mathtt X}^{\rm in},\psi^{\rm in},{\mathtt V}^{\rm in})\in {\mathbb X}$. With the notations just introduced and using the Duhamel formulation associated with \eqref{eqA}, we can decompose the surface elevation as
	\begin{equation}\label{decompzeta}
		\zeta(t,\cdot)=\zeta^{\rm exc}(t,\cdot)+\int_0^t {\boldsymbol \zeta}^{\rm rad}(t-\tau)\cdot \dot{\mathtt X}(\tau){\rm d}\tau
	\end{equation}
	and plug it into \eqref{Newton2} to obtain the Cummins equation,
	\begin{equation}\label{Cummins}
		({\mathcal M}+{\mathcal M}_{\rm a})\ddot{\mathtt X}+\rho {\mathtt g}\int_0^t {\mathcal K}(t-\tau)\dot{\mathtt X}(\tau){\rm d}\tau+{\mathcal C}{\mathtt X}={\mathcal F}^{\rm exc},
	\end{equation}
	where the kernel ${\mathcal K}\in C(\R^+;{\mathcal M}_3(\R))$ and the excitation force ${\mathcal F}^{\rm exc}\in C(\R^+;\R^3)$ are given by 
	\begin{equation}\label{defkernel}
		{\mathcal K}(t)=\int_{\GD} \partial_z{\bf K}(x,0)\otimes {\boldsymbol \zeta}^{\rm rad}(t,x){\rm d}x
		\quad\mbox{ and }\quad
		{\mathcal F}^{\rm exc}(t)=-\rho{\mathtt g}\int_{\GD}\zeta^{\rm exc}(t,x)\partial_z{\bf K}(x,0){\rm d}x,
	\end{equation}
	\begin{remark}
		Recalling that by construction $\partial_z{\bf K}(x,z=0)={\bf \zeta}^{\rm rad}(t=0,x)$, we can rewrite the kernel as
		$$
		{\mathcal K}(t)=\int_{\GD} {\boldsymbol \zeta}^{\rm rad}(0,x)\otimes {\boldsymbol \zeta}^{\rm rad}(t,x){\rm d}x;
		$$
		this shows that for $t\geq 0$ small enough, one has ${\mathcal K}(t)>0$.
	\end{remark}
	\begin{remark}
		The original derivation of \cite{Cummins62} was done under the assumption of the fluid being initially at rest, with an arbitrary and unspecified exciting force, and without topography. Here, the exciting force corresponds to the wave field that would be generated by the initial data $\zeta^{\rm in}$ and $\psi^{\rm in}$ if the solid were fixed, which is the situation considered in \cite{LannesMing24}. Our result also generalizes the Cummins equations in the presence of non trivial topography, with a possibly emerging bottom.
	\end{remark}
	%
	The Cummins equations \eqref{Cummins} are integro-differential equations on the position/rotation vector ${\mathtt X}$. The following theorem proves its well-posedness.
	\begin{thm}\label{Theorem: Cummins sim F. John}
		Let $\Omega$ and $\Gamma$ be as in Assumption \ref{Assumption domain}, and suppose the coefficients of $\mathcal{C}$ satisfy Assumption \ref{hypStab}. Let $(\zeta^{\rm in},\psi^{\rm in})\in L^2(\GD)\times \dot{H}^{1/2}(\GD)$ and ${\mathcal K}\in C(\R^+;{\mathcal M}_3(\R))$ and ${\mathcal F}_{\rm exc}\in C(\R^+;\R^3)$ be as defined in \eqref{defkernel}. Then for any ${\mathtt X}^{\rm in}\in \R^3$ and ${\mathtt V}^{\rm in}\in \R^3$, there exists a unique solution ${\mathtt X}\in C^2(\R^+;\R^3)$ to \eqref{Cummins} satisfying the initial conditions ${\mathtt X}(0)={\mathtt X}^{\rm in}$ and $\dot{\mathtt X}(0)={\mathtt V}^{\rm in}$. 
		
		If in addition $(\zeta^{\rm in},\psi^{\rm in})\in H^{1/2}(\GD)\times \dot{H}^{1}(\GD)$, then ${\mathtt X}\in C^3(\R^+;\R^3)$.
		
	\end{thm}

	\begin{proof}
		We have proved in the derivation of the Cummins equations that if ${\bf U}=(\zeta,{\mathtt X},\psi,{\mathtt V})\in C(\R^+;{\mathbb X})$ solves \eqref{FJ2}-\eqref{Newton2} then ${\mathtt X}$ solves the Cummins equation. The reverse is true, since if ${\mathtt X}$ solves the Cummins equation, one just has to take ${\mathtt V}=\dot{\mathtt X}$, $\zeta$ given by \eqref{decompzeta} and $\psi$ by a similar decomposition formula to obtain a solution to \eqref{FJ2}-\eqref{Newton2}. The result therefore follows from Theorem \ref{thm: Wp energy space}. 
	\end{proof}

	In general, the solution to the Cummins equation is no more regular than $C^2(\R^+;\R^3)$ if $(\zeta^{\rm in},\psi^{\rm in})\in L^2(\GD)\times \dot{H}^{1/2}(\GD)$ and $C^3(\R^+;\R^3)$ if $(\zeta^{\rm in},\psi^{\rm in})\in H^{1/2}(\GD)\times \dot{H}^{1}(\GD)$.  This is proved in Section \ref{Sec: right angles}, where we also show that there is a singularity in the Kirchhoff potentials for horizontal translations at higher regularity. We first study the issue of higher order regularity in an abstract functional framework.

	\section{Abstract higher order regularity in the general case}\label{sectregabstract}

	In order to study higher order regularity of solutions to \eqref{FJabstract}, it is convenient to introduce a scale of functional spaces adapted to the problem.
	We know by Corollary 2 in \cite{LannesMing24} that $G_0$ admits a self-adjoint realization on $L^2(\Gamma^{\rm D})$ with domain $H^1(\Gamma^{\rm D})$. Since it is also positive, it admits a square root, denoted $G_0^{1/2}$, which is such that $\vert G_0^{1/2} \psi \vert_{L^2(\Gamma^{\rm D})}^2=\langle \psi, G_0\psi \rangle_{\dot{H}^{1/2}\times (\dot{H}^{1/2})'}$. Since this latter expression makes sense for $\psi\in \dot{H}^{1/2}(\Gamma^{\rm D})$, we still use  the notation $\vert G_0^{1/2} \psi \vert_{L^2(\Gamma^{\rm D})}$ for such $\psi$ even though they are not necessarily in $L^2(\Gamma^{\rm D})$. More generally, if $n$ is an odd integer, we write
	$$
	\forall n=2l+1,\qquad \vert G_0^{n/2}\psi \vert_{L^2(\Gamma^{\rm D})}^2=\langle G_0^l\psi, G_0 (G_0^l \psi) \rangle_{\dot{H}^{1/2}\times (\dot{H}^{1/2})'},
	$$
	which is well defined whenever $G_0^l \psi \in \dot{H}^{1/2}(\Gamma^{\rm D})$; by abuse of language, we then say that $G_0^{n/2}\psi\in L^2(\Gamma^{\rm D})$. We can now introduce the appropriate functional spaces.
	\begin{Def}\label{defhighreg}
		Let $\Omega$ and $\Gamma$ be as in Assumption \ref{Assumption domain}, and $n\in {\mathbb N}^*$. \\
		{\bf i.} We define
		\begin{align*}
			\dot{\mathcal H}^{n/2}(\Gamma^{\rm D}) &=\{ f\in \dot{H}^{1/2}(\Gamma^{\rm D}), \quad \forall 1\leq j\leq n, \quad G_0^{j/2}f \in L^2(\Gamma^{\rm D})\},\\
			{\mathcal H}^{n/2}(\Gamma^{\rm D}) &= L^2(\Gamma^{\rm D})\cap \dot{\mathcal H}^{n/2}(\GD),
		\end{align*}
		respectively endowed with their canonical semi-norm and norm.\\
		{\bf ii.} The space ${\mathbb X}^n$ is defined as
		$$
		{\mathbb X}^n={\mathcal H}^{n/2}(\Gamma^{\rm D})\times {\mathbb R}^3\times \dot{\mathcal H}^{(n+1)/2}(\Gamma^{\rm D})\times {\mathbb R}^3
		$$
		and is endowed with the semi-norm
		$$
		\Vert {\bf U}\Vert_{{\mathbb X}^n}=\Vert {\bf U}\Vert_{\mathbb X}+\sum_{j=1}^n ( \vert G_0^{j/2}\zeta \vert_{L^2(\Gamma^{\rm D})} +\vert G_0^{(j+1)/2}\psi \vert_{L^2(\Gamma^{\rm D})}),
		$$
		where ${\bf U}=(\zeta,{\mathtt X}^{\rm T},\psi, {\mathtt V}^{\rm T})^{\rm T}$.
	\end{Def}
	
	We implement the strategy used in \cite{LannesMing24}. In particular, we use the energy estimate \eqref{Eq: Energy est} to deduce a control on higher-order time derivatives and trade them for higher-order spatial regularity for initial data in ${\mathbb X}^n$. However, due to the motion of the object there is an additional source of singularity compared with the case of a fixed object considered in \cite{LannesMing24}. More precisely, the regularity of the Kirchhoff potentials ${\bf K}=(K_1,K_2,K_3)^{\rm T}$ introduced in Definition \ref{defKirchhoff} is a limitation to the regularity of the solution. 
	\begin{cor}\label{correg}
		Let $\Omega$, and $\Gamma$ be as in Assumption \ref{Assumption domain}, and suppose the coefficients of $\mathcal{C}$ satisfy Assumption \ref{hypStab}. Let also $n\in {\mathbb N}^*$ and assume that 
		$$
		(\partial_z{\bf K})_{\vert_{\Gamma^{\rm D}}} \in {\mathcal H}^{(n-1)/2}(\Gamma^{\rm D}). 
		$$
		Then for all ${\bf U}^{\rm in}\in {\mathbb X}^n$, there exists a unique solution ${\bf U}\in \cap_{j=0}^n C^j({\mathbb R}_+;{\mathbb X}^{n-j})$
		to \eqref{FJabstract} with initial condition ${\bf U}(0)={\bf U}^{\rm in}$,
		and one also has  ${\mathtt V}=\dot{\mathtt X}\in C^{n+1}({\mathbb R}_+;{\mathbb R}^3)$; moreover, there exists $C>0$ independent of ${\bf U}^{\rm in}$ such that
		$$
		\forall t\geq 0, \qquad
		\sum_{j=0}^n \Vert \partial_t^j {\bf U}\Vert_{{\mathbb X}^{n-j}}(t) +\vert {\mathtt V}^{(n+1)}(t)\vert \leq C \Vert {\bf U}^{\rm in}\Vert_{{\mathbb X}^n}.
		$$
	\end{cor}

	\begin{proof} 
		It is convenient to isolate the equations satisfied by $U = (\zeta, \psi)^{\rm T}$, namely, \eqref{FJ2}, which we can rewrite in compact form as
		\begin{equation}\label{eqA}
			\qquad \quad 
			\partial_t U+\opA U=F, 
		\end{equation}
		with
		\begin{equation*}
			\opA 
			= 
			\begin{pmatrix}
				0 & -G_0
				\\ 
				{\mathtt g} & 0
			\end{pmatrix}
			\quad\mbox{ and }\quad
			F
			=
			\begin{pmatrix}
				{\mathtt V}\cdot \partial_z{\bf K}
				\\
				0
			\end{pmatrix}.
		\end{equation*}	
		We will use the fact for all $U^{\rm in}\in L^2(\Gamma^{\rm D})\times \dot{H}^{1/2}(\Gamma^{\rm D})$ and  $F\in C([0,T];L^2(\Gamma^{\rm D})\times \dot{H}^{1/2}(\Gamma^{\rm D}))$, there exists a unique solution $U\in C([0,T];L^2(\Gamma^{\rm D})\times \dot{H}^{1/2}(\Gamma^{\rm D}) )$ to \eqref{eqA} with initial value $U^{\rm in}$; this is shown in Theorem 4 of \cite{LannesMing24} and can also easily be deduced from the proof of Theorem \ref{thm: Wp energy space}.
		Let us also notice that if $U$ is a regular solution to \eqref{eqA}, one can use the equation to express time derivatives in terms of space derivatives, so that for all $n\geq 1$, there holds				%
		$$
		\partial_t^n U = (-\opA)^n U + \sum \limits_{k=0}^{n-1} (-\opA)^k( \partial_t^{n-1-k}F).
		$$
		In particular, at $t=0$, one should have $\partial_t^n U =(\partial_t^n U)^{\rm in}$, with
		$$
		(\partial_t^n U)^{\rm in}=(-\opA)^n U^{\rm in} + \sum \limits_{k=0}^{n-1} (-\opA)^k( \partial_t^{n-1-k}F)_{\vert_{t=0}}.
		$$			
		We divide the proof of the corollary in two steps. We first prove the time regularity, namely, that $\mathbf{U}\in C^n([0,T];{\mathbb X})$ and that moreover $ {\mathtt V}\in  C^{n+1}({\mathbb R})$. In a second step, we deduce space regularity. \\
		
		\noindent
		{\bf Step 1.}				
		Under the assumptions of the corollary, we prove by a finite induction that for all $0\leq j \leq n$, one has
		\begin{equation}\label{HR}
			\mathbf{U}\in C^j([0,T];{\mathbb X}) \quad \mbox{ and }\quad {\mathtt V}  \in C^{j+1}({\mathbb R}).
		\end{equation}
		For $j=0$, this is exactly the result stated in Theorem \ref{thm: Wp energy space}. Let $0\leq j\leq n-1$ and assume that \eqref{HR} holds for all integers $j'\leq j$; we now prove that it also holds for $j+1$. From \eqref{HR}$_{j}$, we know that ${\mathtt V}\in C^{j+1}({\mathbb R})$;  since in addition $\partial_z{\bf K}\in L^2(\Gamma^{\rm D})$ (this was established in the proof of Proposition \ref{Prop: Skew-symmetry}) we deduce from the above expression for $F$ that $\partial_t^{j+1}F\in C([0,T]; L^2 \times \dot{H}^{1/2})$. Under the assumptions made in the statement of the corollary, we also get that $(\partial_t^{j+1}U)^{\rm in}\in L^2\times\dot{H}^{1/2}$. As explained above, there exists therefore a unique solution $U_{j+1}\in C([0,T];L^2\times \dot{H}^{1/2} )$ to the initial value problem
		\begin{equation}\label{eqUjp1}
			\begin{cases}
				\partial_t U_{j+1}+AU_{j+1}=\partial_t^{j+1} F,\\
				(U_{j+1})_{\vert_{t=0}}=(\partial_t^{j+1} U)^{\rm in}.
			\end{cases}
		\end{equation}
		The function $\widetilde{U}$ defined as
		$$
		\widetilde{U}(t,x)=U^{\rm in}(x)+\dots+\frac{1}{(j)!}t^j(\partial_t^{j}U)^{\rm in}(x)+\int_0^t \frac{(t-s)^j}{j!}U_{j+1}(s,x){\rm d}s
		$$	
		belongs therefore to $C^{j+1}([0,T]; L^2(\Gamma^{\rm D})\times \dot{H}^{1/2}(\Gamma^{\rm D}) )$. Since moreover $\partial_t^{j+1}\widetilde{U}=U_{j+1}$, one deduces after substitution in \eqref{eqUjp1} that
		$$
		\begin{cases}
			\partial_t (\partial_t^{j+1} \widetilde{U})+A (\partial_t^{j+1}\widetilde{U})=\partial_t^{j+1} F,\\
			(\partial_t^{j+1}\widetilde{U})_{\vert_{t=0}}=(\partial_t^{j+1} U)^{\rm in}.
		\end{cases}
		$$
		Integrating $j+1$ times in time and remarking that by construction all the boundary terms cancel at $t=0$, one gets that $\widetilde{U}$ solves the same initial value problem as $U$, so that, by uniqueness, $U=\widetilde{U}$, so that $U\in C^{j+1}([0,T];L^2(\Gamma^{\rm D})\times \dot{H}^{1/2}(\Gamma^{\rm D}) )$. The fact that ${\mathtt V}\in  C^{j+2}$ is then obtained as in the proof of Theorem \ref{thm: Wp energy space}, using \eqref{Newton2}. This ends the proof of the induction. \\
		
		\noindent
		{\bf Step 2.} Space regularity. From the previous step, we have a solution $\mathbf{U}\in C^n([0,T];{\mathbb X})$ to \eqref{FJabstract}  with initial data ${\bf U}^{\rm in}$. We readily deduce that for all $0\leq j+k\leq n$, $(-{\bf A})^k \partial_t^j {\bf U} \in C([0,T];{\mathbb X})$ also solves \eqref{FJabstract}  but with initial data $(-{\bf A})^{j+k} {\bf U}^{\rm in}$, and satisfies the energy estimate, namely,
		$$
		\forall t\geq 0, \qquad \Vert (-{\bf A})^k \partial_t^j {\bf U}(t)\Vert_{\mathbb X}=\Vert (-{\bf A})^{j+k} {\bf U}^{\rm in}\Vert_{\mathbb X}.
		$$ 
		Now, from the definition of ${\bf A}$ and using the assumption that $(\partial_z{\bf K})_{\vert_{\Gamma^{\rm D}}}\in {\mathcal H}^{(n-1)/2}(\Gamma^{\rm D})$, one readily checks that ${\bf A}: {\mathbb X}^m\to {\mathbb X}^{j-1}$ is a well defined and continuous operator for all $1\leq m\leq n$. It follows that 
		\begin{equation}\label{piku}
			\forall t\geq 0, \qquad \Vert (-{\bf A})^k \partial_t^j {\bf U}(t)\Vert_{\mathbb X} \leq \Vert (-{\bf A})^{j+k} {\bf U}^{\rm in}\Vert_{\mathbb X}\lesssim \Vert {\bf U}^{\rm in}\Vert_{{\mathbb X}^n}.
		\end{equation}
		In order to obtain the estimate of the corollary, it is therefore enough to prove that
		\begin{equation}\label{ineqspacereg}
			\forall 1\leq k \leq n, \quad \Vert {\bf U} \Vert_{{\mathbb X}^k}\lesssim \sum_{0\leq k'\leq k} \Vert (-{\bf A})^{k'}{\bf U}\Vert_{{\mathbb X}}.
		\end{equation}
		In order to prove this inequality, let us first remark that
		\begin{align*}
			\Vert {\bf U} \Vert_{{\mathbb X}^k}
			&=\Vert {\bf U}\Vert_{{\mathbb X}}+\sum_{1\leq k'\leq k}\big(\vert G_0^{(k'-1)/2}\zeta\vert_{\dot{H}^{1/2}}+\vert G_0^{(k'-1)/2}G_0\psi\vert_{L^2}\big)\\
			&\lesssim \Vert {\bf U}\Vert_{{\mathbb X}}+\sum_{1\leq k'\leq k}\big(\vert G_0^{(k'-1)/2}\zeta\vert_{\dot{H}^{1/2}}+\vert G_0^{(k'-1)/2}(G_0\psi+\partial_z{\bf K}\cdot {\mathtt V}\vert_{L^2}\big)+\Vert \partial_z {\bf K}\Vert_{{\mathcal H}^{n-1/2}} |{\mathtt V}|.
		\end{align*}
		Remarking that the first and third components of $(-{\bf A}){\bf U}$ are respectively given by $G_0\psi+\partial_z{\bf K}\cdot {\mathtt V}$ and $-{\mathtt g}\zeta$, this implies that
		$$
		\Vert {\bf U} \Vert_{{\mathbb X}^k} \lesssim \Vert (-{\bf A}){\bf U} \Vert_{{\mathbb X}^{k-1}}+(1+\Vert \partial_z {\bf K}\Vert_{{\mathcal H}^{n-1/2}})\Vert {\bf U}\Vert_{{\mathbb X}}.
		$$
		Iterating this inequality, and using the assumption that $\Vert \partial_z {\bf K}\Vert_{{\mathcal H}^{n-1/2}}<\infty$, one readily obtains \eqref{ineqspacereg}, which concludes the proof of the corollary. 
		
	\end{proof}

	\section{The case of small angles at the contact points} 
	\label{Sec: Small angles}
	
	In the previous section, we established higher order regularity, but in the abstract functional space ${\mathbb X}^n$. In general, these spaces are of finite codimension and defined through the cancellation of various linear forms used to define the coefficients of the singular terms in the singular expansion of variational solutions for elliptic equations in corner domains; see for instance \cite{DalibardMR} in a different context. When the contact angles are small, the analysis is simpler and we can characterize ${\mathbb X}^n$ in terms of more standard Sobolev spaces. To do so, we need some basic properties on the trace mapping $\mathrm{Tr}^{\rm D}: \dot{H}^{1}(\Omega)\rightarrow \dot{H}^{1/2}(\Gamma^{\rm D})$ defined in \eqref{tracemapping}. This is done in Subsection \ref{Subsec: trace mapping}. Moreover, we need to revisit elliptic estimates provided in \cite{Grisvard85, DaugeNBL90} for small angles in the homogeneous functional setting, this is done in Subsection \ref{Subsec: Elliptic small angle}. The characterization of the spaces ${\mathcal H}^{n/2}$ and $\dot{\mathcal H}^{(n+1)/2}$, and therefore of ${\mathbb X}^n$, in terms of standard Sobolev spaces is then derived in Subsection  \ref{Subsec: Characterization mathcal H}.
	These results are then used in Subsection \ref{Subsec: higher reg for small angle} where we show higher order regularity of John's problem when the contact angles are small. Lastly, we stated the corresponding results for the Cummins equations in Subsection \ref{Subsec: Cummins higher reg for small angle}. 

	\subsection{Properties of the trace mapping $\mathrm{Tr}^{\rm D}$}\label{Subsec: trace mapping}
	
	We prove here that the restriction to $\dot{H}^{s+1}(\Omega)$ of the trace mapping   $\mathrm{Tr}^{\rm D}: \dot{H}^{1}(\Omega)\rightarrow \dot{H}^{1/2}(\Gamma^{\rm D})$ defined in \eqref{tracemapping}, is continuous with values in $\dot{H}^{s+1/2}(\Gamma^{\rm D})$ for all $s>1/2$,.
	\begin{prop}\label{Prop: trace est}
		Let $\Omega$ and $\Gamma$ be as in Assumption \ref{Assumption domain}, and let $s>1/2$.  The trace mapping $\mathrm{Tr}^{\rm D}: F \in \dot{H}^{s+1}(\Omega)\mapsto F_{|_{\Gamma^{\rm D}}} \in  \dot{H}^{s+1/2}(\Gamma^{\rm D})$ is well defined and continuous.
	\end{prop}
	\begin{proof}
		For  $s >  1/2$  one has by Proposition 5 in \cite{LannesMing24} that $\dot{H}^{s+1/2}(\Gamma^{\rm D}) = \dot{H}^{1/2}(\Gamma^{\rm D})\cap \big{(}\dot{H}^{s+1/2}(\mathcal{E}_-)\times \dot{H}^{s+1/2}(\mathcal{E}_+)\big{)}$. 
		Since $\mathrm{Tr}^{\rm D}: \dot{H}^{1}(\Omega)\rightarrow \dot{H}^{1/2}(\Gamma^{\rm D})$ is well defined and continuous by Theorem 1 in \cite{LannesMing24}, and recalling that $\vert f \vert_{\dot{H}^{s+1/2}(\mathcal{E}_\pm)}=\vert \partial_x f\vert_{H^{s-1/2}({\mathcal E}_\pm)}$, the result follows from the estimate $\vert \partial_x F_{|_{\Gamma^{\rm D}}} \vert_{H^{s-1/2}({\mathcal E}_\pm)}\lesssim \Vert \partial_x F\Vert_{H^s(\Omega)}$ which is a consequence of the trace theorem in standard Sobolev spaces since $s>1/2$.
	\end{proof}  
	The fact that the mapping considered in Proposition \ref{Prop: trace est} is onto is a direct consequence of the following proposition, whose proof is omitted as it is a straightforward generalization of the proof of Theorem 2.21 in \cite{LannesMing24}, which corresponds to the special case $s=0$.     
	\begin{prop}\label{Prop: extension high reg}
		Let $\Omega$ and $\Gamma$ be as in Assumption \ref{Assumption domain}, and let $s\geq 0$. Then there exist a continuous mapping $\mathbf{E}: \dot{H}^{s+1/2}(\Gamma^{\rm D}) \rightarrow \dot{H}^{s+1}(\Omega)$ such $\mathrm{Tr}^{\rm D} \circ \mathbf{E} = \mathrm{Id}_{\dot{H}^{s+1/2}(\Gamma^{\rm D}) \rightarrow \dot{H}^{s+1/2}(\Gamma^{\rm D})}$.
	\end{prop}

	\subsection{Elliptic estimates for corner domains with small angles}\label{Subsec: Elliptic small angle}
	We now present two elliptic regularity results for the problem with mixed boundary conditions and to the Neumann problem in the case of small angles. We extend classical results provided by \cite{Dauge88,Grisvard85} to unbounded domains and, in the case of the mixed problem, we allow for Dirichlet data in $\dot{H}^{s+1/2}(\Gamma^{\rm D})$ -- and not only in the standard Sobolev space ${H}^{s+1/2}(\Gamma^{\rm D})$.  
	\begin{prop}\label{Prop: DN high reg small angle}      Let $\Omega$ and $\Gamma$ be as in Assumption \ref{Assumption domain}, and let $s>1/2$. If the angles at each corner point between $\Gamma^{\rm N}$ and $\Gamma^{\rm D}$  are strictly smaller than $\frac{\pi}{2s}$ then for all  
		$\psi \in \dot{H}^{s+1/2}(\Gamma^{\rm D})$, the variational solution $\psi^{\mathfrak h}$ to the mixed boundary value problem \eqref{Eq: psi^H} is in $\dot{H}^{s+1}(\Omega)$ and there is a constant $C>0$ independent of $\psi$ such that 
		\begin{equation}\label{Eq: est on u small angle mixed pb}
			\frac{1}{C} |\psi |_{\dot{H}^{s+1/2}(\Gamma^{\rm D})} \leq \| \nabla \psi^{\mathfrak h}\|_{H^{s} (\Omega )} \leq C |\psi|_{\dot{H}^{s+1/2}(\Gamma^{\rm D})}.
		\end{equation}
		In particular,  one has $G_0\psi\in {H}^{s-1/2}(\GD)$ and here exists another constant $C'>0$ independent of $\psi$ such that
		\begin{equation}\label{Eq: control of G_0 hi reg}
			\vert G_0 \psi\vert_{{H}^{s-1/2}(\GD)}\leq C' \vert \psi \vert_{\dot{H}^{s+1/2}(\GD)}.
		\end{equation} 
	\end{prop}
	\begin{remark}
		The estimate \eqref{Eq: control of G_0 hi reg} also holds for $s=1/2$ as a consequence of Rellich identity, see Proposition 3.7 in \cite{LannesMing24}.
	\end{remark}

	\begin{proof} 
		The extension ${\bf E}\psi\in \dot{H}^{s+1}(\Omega)$ of $\psi$ furnished by Proposition \ref{Prop: extension high reg}, satisfies by the trace theorem
		\begin{equation}\label{Eq: est on extension}
			|\mathbf{n} \cdot (\nabla {\bf E}\psi)_{|_{\Gamma^{\rm N}}}|_{H^{s - 1/2}(\Gamma^{\rm N})} \lesssim \|\nabla {\bf E}\psi\|_{H^{s}(\Omega)} \lesssim |\psi|_{\dot{H}^{s + 1/2}(\Gamma^{\rm D})},
		\end{equation}
		Then $v = \psi^{\mathfrak h}-{\bf E}\psi$ solves
		\begin{equation*}
			\begin{cases}
				\Delta v =  -\Delta {\bf E}\psi  &\mbox{ in }\Omega\\
				v = 0 & \mbox{ on }\Gamma^{\rm D}\\
				\partial_{\rm n}v = - \partial_{\rm n}{\bf E}\psi  & \mbox{ on }\Gamma^{\rm N};
			\end{cases}
		\end{equation*}
		since the angles of the contact angles are smaller than $\frac{\pi}{2s}$, we can conclude by Theorem 8.13 and Section 8.C of \cite{DaugeNBL90} (see also Theorem 5.1.3.1 in \cite{Grisvard85}) that $v \in H^{s+1}(\Omega)$ and that one has 
		\begin{align*}
			\|  v \|_{H^{s+1}(\Omega)}& \lesssim \|\nabla {\bf E}\psi\|_{H^s(\Omega)} + |\partial_{\rm n} {\bf E}\psi|_{H^{s-1/2}(\Gamma^{\rm N})}\notag
			\\
			& 
			\lesssim  \vert \psi\vert_{\dot{H}^{s+1/2}(\GD)}.
		\end{align*}
		Together with Proposition \ref{Prop: trace est}, this easily implies the second inequality in \eqref{Eq: est on u small angle mixed pb}. The first one follows from the fact that $\vert \psi\vert_{\dot{H}^{s+1/2}}=\vert \partial_x\psi\vert_{H^{s-1/2}}$ and the trace theorem applied to the function $\partial_x\psi^{\mathfrak h}$ in standard Sobolev spaces since $s>1/2$.
		
		Finally, \eqref{Eq: control of G_0 hi reg} follows directly from the estimate  \eqref{Eq: est on u small angle mixed pb} and the trace theorem.
		%
		%
		%
		%
		%
	\end{proof}

	The next result concerns the reverse inequality, where we show improved regularity of the Dirichlet data through the control of the Dirichlet-Neumann map. A similar result was provided in \cite{LannesMing24} for $s = n/2$ with $n\in \N^{\ast}$ (see Proposition 3.10 and proof of Corollary 4.14). However, when $s\geq 1$ it was also assumed that $\psi \in L^2(\Gamma^{\rm D})$ when $\Gamma^{\rm D}$ is unbounded. This assumption is removed here.
	\begin{prop}\label{Prop: NN high reg small angle}
		Let $\Omega$ and $\Gamma$ be as in Assumption \ref{Assumption domain}, and let $s\geq 1$. If the angles at each corner point between $\Gamma^{\rm N}$ and $\Gamma^{\rm D}$  are strictly smaller than $\frac{\pi}{s}$ then
		for all $\psi \in \dot{H}^{1/2}(\GD)$ such that $G_0\psi\in H^{s-1/2}(\GD)$, one actually has $\psi \in \dot{H}^{s+1/2}(\GD)$ and there holds 
		\begin{equation}\label{Eq: psi higher reg small angle}
			|\psi|_{\dot{H}^{s+1/2}(\Gamma^{\rm D})}  \leq C\big{(} |\psi|_{\dot{H}^{1/2}(\Gamma^{\rm D})}+|G_0 \psi|_{H^{s-1/2}(\Gamma^{\rm D})}\big{)},
		\end{equation}
		for some constant $C>0$ independent of $\psi$. 
	\end{prop}
	\begin{remark}
		The case $s=1/2$ is proved using a Rellich identity in Proposition 3.8 in \cite{LannesMing24}.
	\end{remark}
	\begin{proof}[Proof of Proposition \ref{Prop: NN high reg small angle}] We prove the proposition in the configuration of Figure \ref{Figure: Setu-up}, where $x_{\rm L}$ is finite and $x_{\rm R}=\infty$; the adaptation to the other admissible configurations is straightforward.
		The proof is based on elliptic estimates for a Neumann problem, and in order to work in a homogeneous functional setting, we need the following Poincaré-type inequality.
		
		\begin{lemma}\label{Lemma: Poincare type}
			Let $\Omega$ and $\Gamma$ be as in Assumption \ref{Assumption domain}, with $x_{\rm L}>-\infty$ and $x_{\rm R}=\infty$. Suppose further, $a>x_{\rm r}$ such that $\Gamma^{\rm w}$ is located on the half-plane $\{x<a\}$. Let also  $s\geq 1$. If $u \in \dot{H}^{s+1/2}(\Omega)$ is harmonic then its rescaled vertical average $\widetilde{u}$ defined for all $x>a$ as
			\begin{equation*} 
				\widetilde{ u }(x) = \int_{-1}^{0}  u (x,-b(x)z) \: \mathrm{d}z,
			\end{equation*}
			satisfies the estimate
			\begin{equation}
				\|u - \widetilde{u}\|_{{L^2(\Omega \cap  \{x>a\})}} + \|\partial_x^{2}\widetilde{u}\|_{H^{s-1}(\Omega \cap  \{x>a\})} \color{black} \leq C \| \nabla u\|_{H^{s-1/2}(\Omega)},
			\end{equation}
			for some constant $C>0$ independent of $u$.
		\end{lemma}
		
		\begin{proof}[Proof of Lemma  \ref{Lemma: Poincare type}]
			The first term is bounded by the gradient in $L^2(\Omega)$ using  a simple version of the Poincaré inequality,
			\begin{equation*}
				\|u - \widetilde{u}\|_{{L^2(\Omega \cap  \{x>a\})}} \lesssim  \| \nabla u \|_{{L^2(\Omega)}}.
			\end{equation*}
			For the second term, we use the fact that $u$ is harmonic on $\Omega\cap \{x>a\}$ to get that the function $u^\sharp$ defined on $(a,\infty)\times (-1,0)$ by $u^\sharp(x,z)=u(x,-b(x)z)$ satisfies an elliptic equation of the form
			$$
			\partial_x^2 u^\sharp + \partial_z (p_1(z,b)\cdot \nabla u^\sharp)+p_2(z,b)\cdot \nabla u^\sharp =0,
			$$
			where $p_1$ and $p_2$ are two vectors satisfying
			$$
			p_1(0,b)=0 \quad \mbox { and }\quad \Vert p_j \Vert_{W^{l-1}(\Omega\cap \{x>a\})}\leq C(\vert b\vert_{W^{l,\infty}},\inf_{x>a} \frac{1}{b(x)}),
			$$
			the exact expression of these two vectors being of no importance. Since $\widetilde{u}(x)=\int_{-1}^0 u^\sharp(x,z){\rm d}z$, we deduce that
			\begin{align*}
				\partial_x^{2}\widetilde{u}= \int_{-1}^0 \partial_x^2 u^\sharp 
				= \big[ p_1(-1,b)(\nabla u^\sharp)_{\vert_{z=-1}}\big] - \int_{-1}^0 p_2\cdot \nabla u^\sharp,
			\end{align*}
			where we used the fact that $p_1(0,b)=0$. Using the trace theorem to control the first term of the right-hand side, we get that
			$$
			\Vert \partial_x^{2}\widetilde{u}\Vert_{H^{s-1}(\Omega\cap \{x>a\})}
			\lesssim \Vert \nabla u \Vert_{H^{s-1/2}(\Omega)}+ \Vert \nabla u \Vert_{H^{s-1}(\Omega)},
			$$
			where we used also the fact that $\Vert \nabla u^\sharp \Vert_{H^r((a,\infty)\times (-1,0))}\lesssim \Vert \nabla u \Vert_{H^r(\Omega)}$, with $r=s-1/2,s$.
			To conclude with, it follows from the above, and from the fact the that $H^{s-1/2}(\Omega)$ is continuously embedded in $H^{s-1}(\Omega)$.
		\end{proof}

		To prove the estimate for $s\geq 1$, it is enough to show that if $\phi$ is variational solution of  
		\begin{equation*}
			\begin{cases}
				\Delta \phi =  0  &\mbox{ in }\Omega\\
				\partial_{\rm n} \phi = G_0\psi  & \mbox{ on }\Gamma^{\rm D}\\
				\partial_{\rm n}\phi = 0  & \mbox{ on }\Gamma^{\rm N},
			\end{cases}
		\end{equation*}
		then it satisfies the following bound
		\begin{equation}\label{Eq: grad psi h in Hs}
			\| \nabla \phi\|_{H^{s} (\Omega )} \leq C \big{(} 		\| \nabla \phi\|_{H^{s-1/2} (\Omega )} + \vert G_0 \psi \vert_{H^{s-1/2}(\Gamma^{\rm D})} \big{)}.
		\end{equation} 
		Indeed, since $\vert G_0 \psi \vert_{\dot{H}^{1/2}(\GD)'} \lesssim \vert \psi\vert_{\dot{H}^{1/2}(\GD)}$, we have that $\| \nabla \phi\|_{L^2(\Omega )} \lesssim  \vert  \psi \vert_{\dot{H}^{1/2}(\Gamma^{\rm D})}$, and we may use induction and  the interpolation inequality $\|\nabla u\|_{H^{s-1/2}(\Omega)}\lesssim \|\nabla u\|_{H^{s-1}(\Omega)}^{1/2}\|\nabla u\|_{H^{s}(\Omega)}^{1/2}$  to find that
		\begin{equation*}
			\| \nabla \phi\|_{H^{s} (\Omega )} \leq C \big{(}  \vert  \psi \vert_{\dot{H}^{1/2}(\Gamma^{\rm D})} + 	\vert G_0 \psi \vert_{H^{s-1/2}(\Gamma^{\rm D})} \big{)},
		\end{equation*}
		for $s\geq 1$. Using the trace estimate in Proposition \ref{Prop: trace est}, we deduce that $\phi_{\vert_{\GD}}$ is bounded in $\dot{H}^{1/2}(\GD)$ by the right-hand side of \eqref{Eq: psi higher reg small angle}. Now, we just have to observe that $\phi=\psi^{\mathfrak h}+c_0$ for some constant $c_0\in {\mathbb R}$ to obtain that $\psi=\psi^{\mathfrak h}_{\vert_{\GD}}=\phi_{\vert_{\GD}}+c_0$ also satisfies \eqref{Eq: psi higher reg small angle}.

		To prove \eqref{Eq: grad psi h in Hs}, we introduce a smooth cut-off function $\chi_{+}$   such that $\chi_{+}=1$ for $x<a$ and $\chi_{+}=0$ for $x>a+1$, for some $a>x_{\rm r}$.  Moreover, to keep control of $\nabla \phi$, we will introduce corrections of the mean. In particular, denoting $\Omega_a=\Omega\cap \{x<a+1\}$, we write $\langle  \phi\rangle$ the average of $ \phi$ over $\Omega_a$, 
		\begin{equation} \label{defavv}
			\langle  \phi \rangle = \frac{1}{|\Omega_a |} \int_{\Omega_a} \phi(x,z) \: \mathrm{d}x \mathrm{d}z,
		\end{equation} 
		for $x>a$, we also recall that  $\widetilde{ \phi}$ denotes the rescaled vertical average of $ \phi$. We can now decompose $ \phi$ as 
		\begin{align*}
			\phi
			= & \big[\chi_{+} \langle  \phi\rangle + (1-\chi_{+})\widetilde{ \phi}]
			+  
			\chi_{+} (\phi-	\langle \phi\rangle) 
			+
			(1-\chi_{+}) (\phi - \widetilde{\phi})
			\\ 
			= : & \phi_{\rm 0}+
			\phi_{\rm I} + \phi_{\rm II}.
		\end{align*} 

		\noindent
		{\bf Step 0.} \textit{Estimate on $ \phi_{\rm 0}=\chi_{+} \langle  \phi\rangle + (1-\chi_{+})\widetilde{ \phi}$}. We now use that $\phi_{\rm 0}$ only depends on the horizontal variable and observe that	for $s\geq 1$ there holds,
		\begin{align*}
			\Vert \nabla \phi_{\rm 0} \Vert_{H^{s}(\Omega)} 
			\lesssim
			\|  \nabla \phi\|_{{H}^{s-1}(\Omega)}
			+
			\| \langle \phi\rangle -\widetilde{\phi}\|_{L^2(\Omega\cap \mathrm{supp}(\partial_x \chi_+))}   +
			\|\partial_x^{2}\widetilde{\phi}\|_{H^{s-1}(\Omega \cap  \{x>a\})}.
		\end{align*}
		Then using that $\partial_x\chi_{+}$ is supported in $\{a<x<a+1\}$, the usual Poincaré inequality $	\Vert \langle \phi \rangle - \phi\Vert_{L^2(\Omega_a )}  \leq C  	\Vert  \nabla \phi\Vert_{L^2(\Omega_a )} $,  and Lemma \ref{Lemma: Poincare type}  we find that
		\begin{equation*}
			\Vert \nabla \phi_{\rm 0} \Vert_{{H}^{s}(\Omega)} \leq C 	\Vert \nabla \phi\Vert_{{H}^{s-1/2}(\Omega)}.
		\end{equation*}
		\\  
		\noindent
		{\bf Step 1.} \textit{Estimate of $\phi_{\rm I}= \chi_{+} (\phi -	\langle \phi \rangle) $.}  Using the support of $\chi_+$, there is a bounded domain $\widetilde{\Omega}\subset \Omega$ whose boundary $\partial \widetilde{\Omega} = \Gamma^{\rm D} \cap \{x<a+2\} \cup \widetilde{\Gamma}^{\rm N}$  coincides with the boundary of $\Omega$ for $x \in  \mathrm{supp}(\chi_{+})$ and is smooth for $x\in(a+1,a+2)$. Then we have that $\phi_{\rm I}$ satisfies
		$$
		\begin{cases}
			\Delta \phi_{\rm I}= f_{\rm I}  &\mbox{ in }\widetilde{\Omega} \\
			\partial_{\rm n} \phi_{\rm I} =   \chi_+G_0 \psi & \mbox{ on } \Gamma^{\rm D} \cap \{x<a+2\},\\
			\partial_{\rm n} \phi_{\rm I} = 0 & \mbox{ on } \widetilde{\Gamma}^{\rm N} ,
		\end{cases}
		$$
		with  
		$$f_{\rm I}=  -\chi_{+}''(\phi-\langle \phi \rangle) - 2 \chi_{+}'\partial_x\phi.$$
		On bounded domains with contact angles smaller than $\frac{\pi}{s}$ we may again apply  \cite{Grisvard85,DaugeNBL90} to find that
		\begin{align*}
			\Vert \nabla \phi_{\rm I}\Vert_{H^s(\Omega_a)} 
			\lesssim &
			\Vert f_{\rm I}\Vert_{H^{s-1}(\Omega_a)}
			+
			\vert \chi_+G_0 \psi \vert_{H^{s-1/2}(\Gamma^{\rm D} \cap \{x<a+2\})}
			\\ 
			\lesssim &
			\|\phi\|_{H^{s-1/2}(\Omega)}  
			+
			\vert G_0 \psi \vert_{H^{s-1/2}(\Gamma^{\rm D})},
		\end{align*}
		where we also used Lemma \ref{Lemma: Poincare type}  in the last estimate.
		\\

		\noindent
		{\bf Step 2.} \textit{Estimate of $\phi_{\rm II}=(1-\chi_{+}) (\phi - \widetilde{\phi)}$.} Let $\widetilde{b}$ be an extension of $b_{\vert_{x>a}}$ to $\R$ such that $\widetilde{b}$ never vanishes and is equal to some negative constant $b_0$ for $x<a-1$. We may extend the domain $\Omega \cap \{x>a\}$ to the strip $\mathcal{S} = \{(x,z), \widetilde{b}(x)<z<0\}$. Still denoting by $\phi_{\rm II}$ the extension by $0$ of $\phi_{\rm II}$ to ${\mathcal S}$, we have
		$$
		\begin{cases}
			\Delta \phi_{\rm II}= f_{\rm II}  &\mbox{ in }\mathcal{S}\\
			\partial_{\rm n} \phi_{\rm II} = (1- \chi_+)G_0 \psi & \mbox{ on } \R\times \{0\},\\
			\partial_{\rm n} \phi_{\rm II}= 0 & \mbox{ on } \widetilde{\Gamma}^{\rm b},
		\end{cases}
		$$
		with  $\widetilde{\Gamma}^{\rm b}=\{(x,z), z=\widetilde{b}(x)\}$ and
		$$f_{\rm II}=  -\chi_{+}''(\phi-\widetilde{\phi}) - 2 \chi_{+}'\partial_x(\phi-\widetilde{\phi}) +(1-\chi_{+}) \partial_x^2\widetilde{\phi}.$$
		Applying classical regularity estimates for this elliptic problem on a smooth strip gives
		\begin{align*}
			\Vert \nabla \phi_{\rm II}\Vert_{H^s(\mathcal S)}\lesssim
			\Vert f_{\rm II}\Vert_{H^{s-1}(\Omega)}
			+
			\vert (1- \chi_+)G_0 \psi \vert_{H^{s-1/2}(\R)}.
		\end{align*}
		To conclude, we use the support of $\chi_+$ and we estimate $f_{\rm II}$ using Lemma \ref{Lemma: Poincare type} to deduce that
		$$
		\Vert \nabla \phi_{\rm II} \Vert_{H^s(\Omega)}\lesssim
		\|\nabla \phi\|_{H^{s-1/2}(\Omega)} 
		+
		\vert G_0 \psi \vert_{H^{s-1/2}(\Gamma^{\rm D})},
		$$
		which concludes the proof of \eqref{Eq: grad psi h in Hs}.
		
	\end{proof}

	\subsection{Characterization of  $\dot{\mathcal H}^{n/2}(\Gamma^{\rm D}) $ and ${\mathcal H}^{n/2}(\Gamma^{\rm D})$ for small angles}\label{Subsec: Characterization mathcal H}
	
	The spaces $\dot{\mathcal H}^{n/2}(\Gamma^{\rm D}) $ and ${\mathcal H}^{n/2}(\Gamma^{\rm D}) $ are related, but not always identical, to $\dot{H}^{n/2}(\Gamma^{\rm D}) $ and ${H}^{n/2}(\Gamma^{\rm D}) $. They are always the same for $n=1,2$; when $n$ is higher, this also remains true if a smallness assumption is made on the contact angles. In general, they are different due to the lack of elliptic regularity in corner domains, see for instance Section \ref{Sec: right angles} when the contact angles are $\pi/2$.
	\begin{prop}\label{propHn}
		Let $\Omega$ and $\Gamma$ be as in Assumption \ref{Assumption domain}.\\
		{\bf i.} For $n=1,2$, one has $\dot{\mathcal H}^{n/2}(\Gamma^{\rm D})= \dot{H}^{n/2}(\Gamma^{\rm D})$ and ${\mathcal H}^{n/2}(\Gamma^{\rm D})={H}^{n/2}(\Gamma^{\rm D})  $.\\
		{\bf ii.} If $n\geq 3$ and the angles at the contact points between $\Gamma^{\rm N}$ and $\Gamma^{\rm D}$ are all strictly smaller than $\pi/(n-1)$, then the same conclusion holds. 
	\end{prop}		
	\begin{remark}
		In the case $n=3$, one has $\dot{\mathcal H}^{3/2}(\GD)\subset \dot{H}^{3/2}(\GD)$. However, the reverse inclusion is false. See Section \ref{Sec: Characterization with vertical walls} below when the contact angles are $\pi/2$. In fact, we will see in the proof below that for $n\geq 3$ that there holds $\dot{\mathcal H}^{n/2}(\GD)\subset \dot{H}^{n/2}(\GD)$ if the angles at the contact points are strictly less than $\frac{2\pi}{n-1}$.

	\end{remark}

	\begin{proof}
		
		For the proof we note that $n=1$ is obvious, and the case $n=2$ was established in the proof of Proposition \ref{Prop: Skew-symmetry}. By induction it is therefore enough to assume that $\dot{H}^{j/2}(\GD)= \dot{\mathcal H}^{j/2}(\GD)$ for $j=1,\dots, n-1$ and show that this implies $\dot{H}^{n/2}(\GD)= \dot{\mathcal H}^{n/2}(\GD)$.

		Let us first prove that $\dot{H}^{n/2}(\GD)\subset \dot{\mathcal H}^{n/2}(\GD)$. To that end, we let $\psi\in \dot{H}^{n/2}(\GD)$ and use \eqref{Eq: control of G_0 hi reg}:
		\begin{equation*}
			\vert G_0 \psi\vert_{{H}^{(n-2)/2}(\GD)}\leq C' \vert \psi \vert_{\dot{H}^{n/2}(\GD)},
		\end{equation*} 
		for $n\geq 3$ and angles strictly less than $\frac{\pi}{n-1}$. Since we have by the induction assumption that ${H}^{(n-2)/2}(\GD)={\mathcal H}^{(n-2)/2}(\GD)$, we also have that $G_0^{k/2}(G_0\psi)\in L^2(\GD)$ for all $1\leq k\leq n-2$ or equivalently that $G_0^{k/2}\psi\in L^2(\Gamma^{\rm D})$ for all $3\leq k\leq n$. Since this property also holds for $k=1,2$, we have proved that $\psi \in 	\dot{\mathcal H}^{n/2}(\GD)$. 
		
		Next, we turn to the proof of  $\dot{\mathcal H}^{n/2}(\GD)\subset \dot{H}^{n/2}(\GD)$. In this case, we suppose $\psi \in \dot{H}^{1/2}(\GD)$ such that $G_0^{k/2}\psi\in L^2(\GD)$ for all $1\leq k\leq n$. This implies that $G_0^{k/2}(G_0\psi)\in L^2(\GD)$ for all $1\leq k\leq n-2$, that is, $G_0\psi \in \dot{\mathcal H}^{(n-2)/2}(\GD)$ and, using the induction assumption, $G_0\psi \in \dot{H}^{(n-2)/2}(\GD)$. Since we also get from the case $n=2$ that $G_0\psi\in L^2(\GD)$, we deduce that 
		$G_0\psi \in {H}^{(n-2)/2}(\GD)$. We can therefore use Proposition \ref{Prop: NN high reg small angle} to conclude that $\psi \in  \dot{H}^{n/2}(\GD)$ for angles strictly smaller than $\frac{2\pi}{n-1}$. Combining the two inclusions, concludes the proof of the proposition.
	\end{proof}

	Now, with this proposition at hand we can provide higher regularity result in standard Sobolev spaces for the solution provided by Theorem \ref{thm: Wp energy space}.

	\subsection{Higher order well-posedness result of John's problem in ${\mathbb X}^n$}\label{Subsec: higher reg for small angle}
	
	Finally, we remark that when the contact angles are small enough, then the assumption on ${\bf K}$ made in the statement of Corollary \ref{correg} is automatically satisfied, and the estimate can be provided in standard Sobolev spaces rather than ${\mathbb X}^n$. We write
	$$
	{\mathbb X}^n_{\rm Sob}=H^{n/2}(\Gamma^{\rm D})\times {\mathbb R}^3\times \dot{H}^{(n+1)/2}(\Gamma^{\rm D})\times {\mathbb R}^3,
	$$
	where ${\mathcal H}^{n/2}(\Gamma^{\rm D})$ and $\dot{\mathcal H}^{(n+1)/2}(\Gamma^{\rm D})$ in the definition of ${\mathbb X}^n$ (see Definition \ref{defhighreg}) have been replaced by the Sobolev spaces ${H}^{n/2}(\Gamma^{\rm D})$ and $\dot{H}^{(n+1)/2}(\Gamma^{\rm D})$.
	\begin{cor}\label{correg2}
		Let $\Omega$, and $\Gamma$ be as in Assumption \ref{Assumption domain}, and assume the coefficients of $\mathcal{C}$ satisfy Assumption \ref{hypStab}. Let also $n\in {\mathbb N}^*$, and suppose that the angles at the contact points between $\Gamma^{\rm N}$ and $\Gamma^{\rm D}$ are smaller than $\pi/n$. Then for all ${\bf U}^{\rm in}\in {\mathbb X}_{\rm Sob}^n$, there exists a unique solution ${\bf U}\in \cap_{j=0}^n C^j({\mathbb R}_+;{\mathbb X}_{\rm Sob}^{n-j})$
		to \eqref{FJabstract} with initial condition ${\bf U}(0)={\bf U}^{\rm in}$,
		and one also has  ${\mathtt V}=\dot{\mathtt X}\in C^{n+1}({\mathbb R}_+;{\mathbb R}^3)$; moreover, there exists $C>0$ independent of ${\bf U}^{\rm in}$ such that
		$$
		\forall t\geq 0, \qquad
		\sum_{j=0}^n \Vert \partial_t^j {\bf U}\Vert_{{\mathbb X}_{\rm Sob}^{n-j}}(t) +\vert {\mathtt V}^{(n+1)}(t)\vert \leq C \Vert {\bf U}^{\rm in}\Vert_{{\mathbb X}_{\rm Sob}^n}.
		$$
	\end{cor}
	\begin{remark}
		The smallness assumption on the contact angles is a \emph{sufficient} condition to ensure that $(\partial_z{\bf K})_{\vert_{\Gamma^{\rm D}}}\in {\mathcal H}^{(n-1)/2}(\Gamma^{\rm D})$, but it is not necessary, as shown in Section \ref{Sec: right angles} when the contact angles are $\pi/2$.
	\end{remark}
	\begin{remark}
		The result is a generalization of Corollary 4.14 in \cite{LannesMing24}, where the object is assumed to be at rest. It also removes the unnecessary assumption made in \cite{LannesMing24} that $\psi \in L^2(\Gamma^{\rm D})$ when $\GD$ is unbounded.
	\end{remark}
	\begin{proof}
		This is a direct consequence of Corollary \ref{correg}. Indeed, under the assumption that the contact angles are smaller than $\pi/n$, we can deduce from Proposition \ref{propHn} that ${\mathcal H}^{n/2}(\Gamma^{\rm D})={H}^{n/2}(\Gamma^{\rm D})$ and ${\mathcal H}^{(n+1)/2}(\Gamma^{\rm D})={H}^{(n+1)/2}(\Gamma^{\rm D})$. The assumption that $(\partial_z{\bf K})_{\vert_{\Gamma^{\rm D}}}\in {\mathcal H}^{(n-1)/2}(\Gamma^{\rm D})$ can then be restated as
		$(\partial_z{\bf K})_{\vert_{\Gamma^{\rm D}}}\in {H}^{(n-1)/2}(\Gamma^{\rm D})$ and is therefore a consequence of Proposition \ref{Prop: DN high reg small angle} under the smallness assumption on the angles.
	\end{proof}

	\subsection{Higher order well-posedness of the Cummins equtions  in ${\mathbb X}^n$}\label{Subsec: Cummins higher reg for small angle}
	Lastly, we comment on the regularity of the Cummins equations in the same setting. This follows directly from Corollary \ref{correg}.
	\begin{cor}\label{corregCummins}
		Let $\Omega$, and $\Gamma$ be as in Assumption \ref{Assumption domain}, and suppose the coefficients of $\mathcal{C}$ satisfy Assumption \ref{hypStab}. Let also $n\in {\mathbb N}^*$ and assume that 
		$$
		(\partial_z{\bf K})_{\vert_{\Gamma^{\rm D}}} \in {\mathcal H}^{(n-1)/2}(\Gamma^{\rm D})
		\quad\mbox{ and }\quad
		(\zeta^{\rm in},\psi^{\rm in})\in {\mathcal H}^{n/2}(\GD)\times \dot{\mathcal H}^{(n+1)/2}(\GD).
		$$
		Then the kernel ${\mathcal K}$ and the exciting force ${\mathcal F}_{\rm exc}$ defined in 		\eqref{defkernel} satisfy
		${\mathcal K}\in C^{n}(\R^+;{\mathcal M}_3(\R))$ and ${\mathcal F}_{\rm exc}\in C^{n}(\R^+;\R^3)$, 
		and one has ${\mathtt X}\in C^{n+2}(\R^+;\R^3)$. 
	\end{cor}
	\begin{remark}
		{\bf i.} As observed in Corollary \ref{correg2}, if the contact angles are strictly smaller than $\pi/n$, then one can replace ${\mathcal H}^{(n-1)/2}(\GD)$, ${\mathcal H}^{n/2)}(\GD)$ and $\dot{\mathcal H}^{(n+1)/2}(\GD)$ in the statement by the Sobolev spaces ${H}^{(n-1)/2}(\GD)$, ${H}^{n/2)}(\GD)$ and $\dot{H}^{(n+1)/2}(\GD)$.\\
		{\bf ii.} The spaces ${\mathcal H}^{n/2}(\GD)$ and $\dot{\mathcal H}^{(n+1)/2}(\GD)$ are also characterized in the next section in the case where the contact angles are $\pi/2$; they differ from standard Sobolev spaces by the presence of compatibility conditions. 
	\end{remark} 
	
	\section{The case of right angles at the contact points}\label{Sec: right angles}
	
	We have established in Corollary \ref{correg} the propagation of the ${\mathbb X}^n$ regularity under an assumption on the Kirchhoff potentials. However, it is in general not obvious to translate this regularity into a regularity measured in standard Sobolev spaces. In the previous section, this was done in the case of small angles. We pursue here this program in the case where the contact angles are assumed to be equal to $\pi/2$. In such a configuration, we are able to describe the corner singularities of the harmonic extension $\psi^{\mathfrak h}$ and to characterize the spaces ${\mathbb X}^n$ for all $n\in {\mathbb N}$.

	We derive in Subsection \ref{Subsec: mixed pb pi/2} new elliptic estimates in the homogeneous functional setting for the harmonic extension $\psi^{\mathfrak h}$; in particular, we provide in Subsection \ref{Subsubsec: Non comp data pi/2} a singularity decomposition of this harmonic extension that could be of independent interest. These results are then used in  Subsection \ref{Subsec: Neumann pb pi/2} to establish higher order ellipticity properties of the Dirichlet-Neumann operator $G_0$. The space ${\mathcal H}^{n/2}$ and $\dot{\mathcal H}^{(n+1)/2}$ involved in the definition of ${\mathbb X}^n$ can then be easily characterized in Subsection \ref{Sec: Characterization with vertical walls}. The applications of these results to the space regularity of the solution to F. John's problem is then proposed in  Subsection \ref{Subsec: John high reg pi/2}. In particular, the singularity decomposition is used to show the lack of regularity in some cases.

	\subsection{Higher order  regularity for the mixed problem}\label{Subsec: mixed pb pi/2}
	
	We are interested here in the mixed problem \eqref{Eq: psi^H} that defines the harmonic extension $\psi^{\mathfrak h}$, namely, 
	\begin{equation}\label{Eq: psi^H2}
		\begin{cases}
			\Delta \psi^{\mathfrak{h}}= 0 \qquad\hspace{0.3cm} \text{in} \quad \Omega ,
			\\ 
			\psi^{\mathfrak{h}} = \psi \qquad \hspace{0.5cm}\text{on} \quad \Gamma^{\rm D},
			\\ 
			\partial_n \psi^{\mathfrak{h}} = 0 \qquad \hspace{0.2cm}\text{on} \quad \Gamma^{\rm N}.
		\end{cases}
	\end{equation}
	As already said, it was shown in \cite{LannesMing24} that this problem admits variational solutions in $\dot{H}^1(\Omega)$ for data $\psi\in\dot{H}^{1/2}(\GD)$. 
	Let us now consider a more regular data $\psi \in \dot{H}^{s+1/2}(\GD)$, with $s\geq 0$ and $\dot{H}^{s+1/2}(\GD)$ as defined in Section \ref{Sec: Wp John and just Cummins}. Due to the corner singularities of $\Omega$, it is in general false that $\nabla\psi^{\mathfrak h}\in H^s(\Omega)$. Indeed, it is well known \cite{BorsukKondratiev,Grisvard85,Dauge88,Grisvard92,Mazya_Kozlov_97} that singular terms arise in general at the corners. It is possible to impose conditions on the data in order to cancel the coefficients of these singularities, hereby ensuring that $\nabla\psi^{\mathfrak h}\in H^s(\Omega)$, but such conditions are nonlocal and it is difficult to assess whether a given boundary data $\psi$ satisfies them (see \cite{Grisvard85,DaugeNBL90}, and \cite{DalibardMarbachRax_25} in a related context). 
	
	However, the case of contact angles equal to $\pi/2$ is a special case, not covered for instance in \cite{Grisvard85}. In terms of the Mellin transform, the singularity does not come from the spectral problem but from the lack of injectivity modulo polynomials in the sense of \cite{Dauge88}. We show below with elementary tools that the conditions to remove the singularities can be reduced to simple compatibility conditions. We are also able to provide an explicit description of the singularities when these conditions are not satisfied.
	
	\subsubsection{The case of compatible data}
	
	We recall that ${\mathtt C}$ denotes the set of all finite contact points.
	\begin{Def}\label{Def: Hcc}
		Let $\Omega$ be as in Assumption \ref{Assumption domain} and assume moreover that all the angles at the contact points are equal to $\pi/2$. Let $s\geq 0$.\\
		{\bf i.} We say that $\psi\in \dot{H}^{s+1/2}(\GD)$  is compatible at order $s$ if the following holds for all ${\mathtt c}=(x_{\mathtt c},0)\in {\mathtt C}$,
		$$
		\begin{cases}
			\partial_x^{2l+1} \psi ({\mathtt c})=0 & \mbox{ for all }l\in {\mathbb N} \mbox{ such that } 2l+1 <s
			\\
			| \cdot-x_{\mathtt c}|^{-1/2}\partial_x^s \psi \in L^2(\GD) & \mbox{ if } s\in 2{\mathbb N}+1.
		\end{cases}
		$$
		{\bf ii.} We denote by $\dot{H}^{s+1/2}_{\rm c.c.}(\GD)$ the set of all $\psi\in \dot{H}^{s+1/2}(\GD)$ that are compatible at order $s$, endowed with the semi-norm
		$$
		\vert \psi\vert_{\dot{H}_{\rm c.c.}^{s+1/2}(\GD)}=\begin{cases}
			\vert \psi \vert_{\dot{H}^{s+1/2}(\GD)} &\mbox{ if }s\not\in 2{\mathbb N}+1,\\
			\vert \psi \vert_{\dot{H}^{s+1/2}(\GD)} + \sum_{{\mathtt c}\in {\mathtt C}}\big\vert |\cdot - x_{\mathtt c}|^{-1/2}\partial_x^s \psi \big\vert_{L^2(\GD)}   &\mbox{ if }s\in 2{\mathbb N}+1.
		\end{cases}
		$$
		{\bf iii.} We similarly define ${H}^{s+1/2}_{\rm c.c.}(\GD)$ by replacing  $\dot{H}^{s+1/2}(\GD)$ by ${H}^{s+1/2}(\GD)$ everywhere in point {\bf ii}.
	\end{Def}	
	
	The proposition below shows that if the data $\psi\in \dot{H}_{\rm c.c.}^{s+1/2}(\GD)$ then $\nabla\psi^{\mathfrak h}$ has the maximal regularity, namely $\nabla\psi^{\mathfrak h}\in H^s(\Omega)$. 
	\begin{prop}\label{Prop: Elliptic reg DN}
		Let $\Omega$, $\Gamma^{\rm D}$ and $\Gamma^{\rm N}$ be as in Assumption \ref{Assumption domain}, and assume moreover that all the contact angles are equal to $\pi/2$.
		Let $s\geq 0$ and  $\psi\in \dot{H}_{\rm c.c.}^{s+1/2}(\Gamma^{\rm D})$. Then the variational solution $\psi^{\mathfrak h}\in \dot{H}^1(\Omega)$ to 
		\eqref{Eq: psi^H2} satisfies $\nabla\psi^{\mathfrak h}\in H^s(\Omega)$ and there exists a constant $C>0$ independent of $\psi$ such that
		$$
		\Vert \nabla\psi^{\mathfrak h}\Vert_{H^s(\Omega)} \leq C \vert \psi \vert_{\dot{H}^{s+1/2}_{\rm c.c.}(\GD)}.
		$$ 
		In particular, one has for any $s\geq 1/2$ and $G_0\psi\in {H}^{s-1/2}_{\rm c.c.}(\GD)$ that there exists another constant $C'>0$ independent of $\psi$ such that
		$$
		\vert G_0\psi\vert_{{H}^{s-1/2}_{\rm c.c.}(\GD)}\leq C' \vert \psi \vert_{\dot{H}^{s+1/2}_{\rm c.c.}(\GD)}.
		$$
	\end{prop}
	\begin{proof}
		The higher regularity result is based on two lemmas. The first one constructs an extension $\psi^{\rm ext}$ of $\psi$ which is harmonic near the corners and satisfies the Neumann boundary condition on $\Gamma^{\rm N}$ near the corners, and such that the mapping $\psi \mapsto \psi^{\rm ext}$ is well defined and continuous with respect to homogeneous semi-norms.
		The proof of this lemma is based on an extension of $\psi$ to the real line and an analysis of the singularities of the Poisson kernel associated with this extension. The proof is postponed to Appendix \ref{appcompatible} for the sake of clarity. We denote $\Omega_{x_{\rm l},x_{\rm r}}$ the portion of the fluid domain located below the segment $(x_{\rm l},x_{\rm r})\times\{0\}$,
		$$
		\Omega_{x_{\rm l},x_{\rm r}}=\Omega \cap \{(x,z), x_{\rm l}<x<x_{\rm r}\}.
		$$
		\begin{lemma}\label{LMell1}
			Under the assumptions of the proposition, there exists $h_1>0$ such that for all $s\geq 0$ and $\psi\in \dot{H}_{\rm c.c.}^{s+1/2}(\Gamma^{\rm D})$, there exists a function $\psi^{\rm ext}\in \dot{H}^{s+1}(\Omega)$, harmonic in $\Omega\backslash\overline{\Omega_{x_{\rm l},x_{\rm r}}}$,  such that $\psi^{\rm ext}=\psi$ on $\Gamma^{\rm D}$ and $\partial_n \psi^{\rm ext}=0$ on $\Gamma^{\rm N}\cap \{z>-h_1\}$, and such that
			\begin{equation}\label{Eq: Est on nabla psi ext in Hk}
				\Vert \nabla \psi^{\rm ext}\Vert_{H^s(\Omega)}
				\leq C \vert \psi\vert_{\dot{H}_{\rm c.c.}^{s+1/2}(\Gamma^{\rm D})},
			\end{equation}
			for some constant $C$ independent of $\psi$.
		\end{lemma}

		Having the extension in hand, we will now use it to study $u = \phi - \psi^{\rm ext}$:
		\begin{equation}\label{Eq: Laplace for u = phi - psi_ext}
			\begin{cases}
				\Delta u = h  &\mbox{ in }\Omega\\
				u =0 & \mbox{ on }\Gamma^{\rm D}\\
				\partial_{\rm n}u =g  & \mbox{ on }\Gamma^{\rm N},
			\end{cases}
		\end{equation}
		where $h = -\Delta \psi^{\rm ext}$ and $g = -	\partial_{\rm n} \psi^{\rm ext}$ which are both supported away from the corner. In other words, we simply need to prove that the elliptic regularity follows if the data is supported far from the corner. We state this as a general result in the next lemma.
		\begin{lemma}\label{LMell2}
			Under the assumptions of the proposition, let $s \geq 1$,  and $h\in H^{s-1}(\Omega)$ and $g\in H^{s-1/2}(\Gamma^{\rm N})$. Let also $u \in H^1(\Omega)$ be the variational solution of the elliptic boundary value problem
			$$
			\begin{cases}
				\Delta u =h  &\mbox{ in }\Omega\\
				u =0 & \mbox{ on }\Gamma^{\rm D}\\
				\partial_{\rm n} u = g  & \mbox{ on }\Gamma^{\rm N}.
			\end{cases}
			$$
			Then if $g$ is supported on $\Gamma^{\rm N}\cap \{z < -h_1\}$ for some $h_1>0$ and if $h$ is supported in $\Omega_{x_{\rm l},x_{\rm r}}:$, one has $u\in H^{s+1}(\Omega)$ and moreover
			$$
			\Vert u \Vert_{H^{s+1}(\Omega)} \leq C \big( \Vert h\Vert_{H^{s-1}(\Omega_{x_{\rm l},x_{\rm r}})}+ \vert g\vert_{H^{s-1/2}(\Gamma^{\rm N})} ),
			$$
			for some $C>0$ independent of $g$ and $h$.
		\end{lemma}
		\begin{proof}[Proof of the lemma]
			We do the proof in the case where $-\infty <x_{\rm L}$ so that ${\mathcal E}_-$ is finite, and $x_{\rm R}=\infty$ so that ${\mathcal E}_+$ is infinite. The other configurations are treated similarly. Note also that by interpolation, it is sufficient to prove the lemma for $s=k$, with $k\in {\mathbb N}^*$.

			Taking $h_1>0$ smaller if necessary, we can assume that $\Gamma^{\rm N}\cap \{z>-h_1\}$ is a union of vertical segments. Let $\theta$ be a smooth compactly supported function on ${\mathbb R}_-$ such that $\theta(z)=1$ for $-\frac{h_1}{2}\leq z\leq 0$ and  $\theta(z)=0$ for
			$z\leq - h_1$, and let $\chi_{-}$ be as in the proof of Lemma \ref{LMell1}. We can then write
			\begin{align*}
				u
				& =
				\theta \chi_{-} u+ \theta (1-\chi_{-}) u +(1-\theta) u
				\\ 
				& =:
				u_{\rm I}+u_{\rm II}+u_{\rm III}.
			\end{align*} 
			We prove the three components are in $H^{k+1}(\Omega)$ and satisfy the estimate of the lemma. To do so, we proceed by induction and show that if $u\in H^{l}(\Omega)$ for some $1\leq l< k$, then the three components are in $u\in H^{l+1}(\Omega)$.\\

			\noindent
			{\bf Step 1.} \textit{Estimate of $u_{\rm I}$}. The function $u_{\rm I}$ solves the elliptic boundary value problem
			$$
			\begin{cases}
				\Delta u_{\rm I}=h_{\rm I}  &\mbox{ in } (x_{\rm L},x_{\rm l})\times (-h_1,0)\\
				u_{\rm I}=0 & \mbox{ on }   (x_{\rm L},x_{\rm l})\times \{0\} \quad\mbox{ and }\quad (x_{\rm L},x_{\rm l})\times \{-h_1\} \\
				\partial_{\rm n}u_{\rm I}=0  & \mbox{ on } \{x_{\rm L}\}\times (-h_1,0) \cup \{x_{\rm l}\}\times (-h_1,0),
			\end{cases}
			$$
			with
			$$h_{\rm I}=2\theta'  \partial_z u + \theta''  u,$$
			where we used the fact that $\chi_{-}$ is constant and $h$ identically vanishes on $\Omega_- = \Omega\cap \{x\in {\mathcal E}_-\}$. Letting $L=x_{\rm l}-x_{\rm L}$, the strategy is to use the  Boussinesq symmetrization introduced in Section \ref{sectBouss} to extend $u_{\rm I}$ and $h_{\rm I}$ to the periodic strip  ${\mathcal S}_{-} = {\mathbb R}/{(2L{\mathbb Z})}\times (-h_1,0)$ where the elliptic theory is standard.\\
			The Boussinesq symmetrizations of $u_{\rm I}$ and $h_{\rm I}$, respectively denoted by $u_{\rm I}^{\rm B}$ and $h_{\rm I}^{\rm B}$ are obtained by taking the even reflection of $u_{\rm I}$ and $h_{\rm I}$ across $\{x=x_{\rm L}\}$ and then the $2L$-periodization of the resulting functions. Remark now that on $\{x=x_{\rm l}\}$, one has  for all integer $m$ such that $2m+1<l-1$,
			\begin{align*}
				\partial_x^{2m+1}h_{\rm I}&=2\theta'  \partial_z \partial_x^{2m+1}u + \theta''  \partial_x^{2m+1}u\\
				&=2\theta'  (-1)^m\partial_z \partial_z^{2m}\partial_xu + \theta'' (-1)^m \partial_z^{2m}\partial_xu\\
				&=0,
			\end{align*}
			where we used the fact that $u$ is harmonic on ${\mathcal S}_{-}$ and that $\partial_{\rm n} u=\partial_x u=0$ on
			$	\{x_{\rm l}\}\times (-H_1,0)$. We can then deduce as in Section \ref{sectBouss} that $h_{\rm I}^{\rm B} \in H^{l-1}({\mathcal S}_-)$. Since $u_{\rm I}^{\rm B}$ is in $H^1({\mathcal S}_-)$ and is obviously a variational solution of 
			$$
			\begin{cases}
				\Delta u_{\rm I}^{\rm B}=h_{\rm I}^{\rm B}  &\mbox{ in } {\mathcal S}_-,\\
				u_{\rm I}=0 & \mbox{ on }   {\mathbb R}/(2L{\mathbb Z})\times \{0\} \quad\mbox{ and }\quad  {\mathbb R}/(2L{\mathbb Z})\times \{-h_1\} ,
			\end{cases}
			$$
			we can use standard elliptic regularity on this periodic strip to deduce that $u_{\rm I}^{\rm B} \in H^{l+1}({\mathcal S}_-)$ and therefore that $u_{\rm I}\in H^{l+1}(x_{\rm L},x_{\rm l})\times (-h_1,0)$, with the estimate
			$$
			\Vert u_{\rm I}\Vert_{H^{l+1}(\Omega)}\lesssim  \Vert u_{\rm I}\Vert_{H^{l}(\Omega)}.
			$$

			\noindent
			{\bf Step 2.} \textit{Estimate of $u_{\rm II}$}. Since $u_{\rm I}$ solves an elliptic value problem on the half strip, namely,
			$$
			\begin{cases}
				\Delta u_{\rm II}=h_{\rm II}  &\mbox{ in } (x_{\rm r},\infty)\times (-h_1,0)\\
				u_{\rm II}=0 & \mbox{ on }   (x_{\rm r}, \infty)\times \{0\} \mbox{ and }\quad (x_{\rm r}, \infty)\times \{-h_1\},\\
				\partial_{\rm n}u_{\rm II}=0 & \mbox{ on } \{x_{\rm r}\}\times (-h_1,0),
			\end{cases}
			$$
			with $h_{\rm II}=2 \theta' \partial_z u +  \theta''  u$, the method of Step 1 can easily be adapted to get
			$$
			\Vert u_{\rm II} \Vert_{H^{l+1}(\Omega)} \lesssim \Vert u\Vert_{H^{l}(\Omega)}.
			$$
			{\bf Step 3}. \textit{Estimate of $u_{\rm III}$}. There exists a smooth domain $\widetilde{\Omega}\subset \Omega$ whose boundary coincides with the boundary of $\Omega$ for $z<-\frac{h_1}{2}$ and a function $g^{\rm ext} \in H^{k-1/2}(\partial\widetilde{\Omega})$  that coincides with $(1-\theta)g$ for  $z<-\frac{h_1}{2}$ and is equal to $0$ for $z>-h_1/2$; the function $u_{\rm III}$ then solves a Neumann problem on a smooth domain, namely,
			$$
			\begin{cases}
				\Delta u_{\rm III}=h_{\rm III}  &\mbox{ in } \widetilde{\Omega}\\
				\partial_{\rm n}u_{\rm III}=g^{\rm ext}  & \mbox{ on }\partial\widetilde{\Omega},
			\end{cases}
			$$
			where $h_{\rm III}=h+2\theta'  \partial_z u + \theta''u$ is in $H^{l-1}(\widetilde{\Omega})$ if $u \in H^{l}(\Omega)$. Using standard elliptic regularity theory and the Poincaré inequality, we deduce that $u_{\rm III} \in H^{l+1}(\widetilde{\Omega})$ and therefore in $H^{l+1}({\Omega})$ since it identically vanishes for $z>-\frac{h_1}{2}$. Moreover, we have that $g^{\rm ext}$ satisfies
			$$\vert g^{\rm ext}\vert_{H^{l-1/2}(\partial\widetilde{\Omega})}\lesssim \vert g\vert_{H^{k-1/2}(\Gamma^{\rm N})},$$ 
			and since $h$ is supported in $\Omega_{x_{\rm l},x_{\rm r}}$, one readily gets the estimate
			$$
			\Vert u_{\rm III} \Vert_{H^{l+1}(\Omega)} \lesssim \Vert h\Vert_{H^{k-1}(\Omega_{x_{\rm l},x_{\rm r}})}+ \vert g\vert_{H^{k-1/2}(\Gamma^{\rm N})}
			+\Vert u \Vert_{H^{l}(\Omega)}.
			$$
			{\bf Step 4}. \textit{Conclusion}. From the three previous step, we conclude that $u \in H^{l+1}(\Omega)$ if it is in $H^{l}(\Omega)$, if $1\leq l < k$. Starting from the variational solution ($l=1$), a finite induction and the estimates of the three steps yield the result.
		\end{proof}

		We can now proceed to the proof of the proposition. We make the decomposition $\phi=\psi^{\rm ext}+u$ where $u$ solves \eqref{Eq: Laplace for u = phi - psi_ext} with $h=-\Delta\psi^{\rm ext}$ satisfies the assumption of Lemma \ref{LMell2} as well as
		$$
		\Vert h\Vert_{H^{s-1}(\Omega_{x_{\rm l},x_{\rm r}})}
		\leq
		\Vert \nabla\psi^{\rm ext}\Vert_{H^{s}(\Omega)};
		$$	 
		since moreover $g=-\partial_{\rm n}\psi^{\rm ext}$ we get  by a classical trace estimate  that
		\begin{align*}
			\vert \partial_{x}\psi^{\rm ext}\vert_{H^{s-1/2}(\Gamma^{\rm N})}\lesssim \| \nabla \psi^{\rm ext}\|_{H^{s}(\Omega)}.
		\end{align*} 
		By Lemma \ref{LMell2}, we therefore get that for all $s\geq 1$, one has
		$$
		\Vert u\Vert_{H^{s+1}(\Omega)} \lesssim \| \nabla \psi^{\rm ext}\|_{H^{s}(\Omega)}.
		$$
		This estimate is also true for $s=0$ which corresponds to the variational estimate; it holds also for $s\in (0,1)$ by interpolation. With Lemma \ref{LMell1}, we therefore get that for all $s\geq 0$ we have
		$$
		\Vert u\Vert_{H^{s+1}(\Omega)} \lesssim \vert \psi \vert_{\dot{H}^{s+1/2}_{\rm c.c.}}.
		$$
		%
		%
		which concludes the proof of the first part of Proposition \ref{Prop: Elliptic reg DN}.
		
		For the second part, we recall from Proposition 3.8 in \cite{LannesMing24} that it holds in a more general setting when $s=1/2$. For $s>1/2$ we use that $\partial_x \psi^{\mathfrak h}=0$ on $\Gamma^{\rm N}\cap \{z>-h_1\}$, together with the trace theorem on polygons (see more specifically Theorem 1.5.2.4 in \cite{Grisvard85} for the special case of a quadrant) that $G_0\psi=(\partial_z\psi^{\mathfrak h})_{\vert_{\GD}}$ belongs to ${H}_{\rm c.c.}^{s-1/2}(\GD)$ and that $\vert G_0\psi \vert_{{H}_{\rm c.c.}^{s-1/2}(\GD)}\lesssim \Vert \nabla \psi^{\mathfrak h}\Vert_{H^s(\Omega)}$; together with the first point of the proof, this concludes the proof of the proposition.
		
	\end{proof}

	\subsubsection{The case of non compatible data}  \label{Subsubsec: Non comp data pi/2}
	
	We have shown in the previous section that for all $s\geq 0$, the variational solution $\psi^{\mathfrak h}$ to the elliptic problem \eqref{Eq: psi^H2} satisfies $\nabla\psi^{\mathfrak h}\in H^{s}(\Omega)^2$ if $\psi \in \dot{H}^{s+1/2}_{\rm c.c.}(\GD)$. We recall that if $2l+1\leq s<2l+3$, with 	$l\in {\mathbb N}$, then $\psi \in  \dot{H}^{s+1/2}_{\rm c.c.}(\GD)$ if $\psi \in \dot{H}^{s+1/2}(\GD)$ and $\psi^{(2k+1)}({\mathtt c})=0$ for all $1\leq k\leq l$ and all finite corner ${\mathtt c}$ and where, as usual, in the case $s=2l+1$, the condition $\psi^{(s)}({\mathtt c})=0$ must be understood in the weak sense $|x-{\mathtt c}|^{-1/2}\psi^{(s)}\in L^2(\GD)$.

	In this section, we consider boundary data $\psi \in \dot{H}^{s+1/2}(\GD)$ such that the traces $\psi^{(2k+1)}({\mathtt c})$ are well defined, but not necessarily zero. We denote by $\dot{H}^{s+1/2}_{\rm tr}(\GD)$ the space of such functions. Note also that the definition below makes sense in the case $s\in 2{\mathbb N}+1$ because the constants $\alpha_{\mathtt c}$ involved in the definition are obviously unique. We also recall that ${\mathtt C}$ denotes the set of all finite contact points.
	\begin{Def}
		Let $\Omega$ and $\Gamma^{\rm D}$ be as in Assumption \ref{Assumption domain}, and assume moreover that all the contact angles are equal to $\pi/2$.
		Let also $s\geq 0$ and introduce the space $\dot{H}^{s+1/2}_{\rm tr}(\GD)$ defined as
		$$
		\dot{H}^{s+1/2}_{\rm tr}(\GD)=\dot{H}^{s+1/2}(\GD) \quad\mbox{ if }\quad s\not\in 2{\mathbb N}+1,
		$$
		endowed with its canonical semi-norm, and
		\begin{align*}
			\dot{H}^{2l+3/2}_{\rm tr}(\GD)=
			\{\psi \in \dot{H}^{2l+3/2}(\GD), &\forall {\mathtt c}=(x_{\mathtt c},0)\in {\mathtt C}, \exists \alpha_{\mathtt c}\in \R,\\
			&  |x-x_{\mathtt c}|^{-1/2}(\partial_x^{2l+1}\psi-\alpha_{\mathtt c})\in L^2(\GD)\},
		\end{align*}
		for all $l\in {\mathbb N}$,  and endowed with the semi-norm 
		$$
		\vert \psi \vert_{\dot{H}^{2l+3/2}_{\rm tr}(\GD)}=\vert \psi \vert_{\dot{H}^{2l+3/2}(\GD)}+\sum_{{\mathtt c}\in {\mathtt C}} \big( \vert x-x_{\mathtt c}|^{-1/2}(\partial_x^{2l+1}\psi-\alpha_{\mathtt c})\vert_{L^2(\GD)}+\vert \alpha_{\mathtt c}\vert\big);
		$$
		in the latter case, we write by abuse of notation $\alpha_{\mathtt c}=\partial_x^{2l+1}\psi({\mathtt c})$.
	\end{Def}
	Clearly, the space $\dot{H}^{s+1/2}_{\rm c.c.}(\GD)$ corresponds to the elements in $\dot{H}^{s+1/2}_{\rm tr}(\GD)$ such that for all finite corner ${\mathtt c}$, one has $\partial_x^{2l+1}\psi({\mathtt c})=0$ for all $l\in {\mathbb N}$ such that $2l+1\leq s$. In order to write the supplementary space of  $\dot{H}^{s+1/2}_{\rm c.c.}(\GD)$ in  $\dot{H}^{s+1/2}_{\rm tr}(\GD)$, we introduce some notations.
	\begin{notation}\label{notachi}
		{\bf i.} We denote by $\chi_{[0]}$ a smooth, even, compactly supported function, equal to $1$ in a neighborhood of the origin, and with support contained in $(-L/2,L/2)$, with $L \leq \min \{|{\mathcal E}_- |,|{\mathcal E}_+| \}$.  
		\\
		{\bf ii.} For all $l\in {\mathbb N}$, we define $\chi_{[2l+1]}$ as
		$$
		\forall x\in {\mathbb R},  \qquad \chi_{[2l+1]}(x)=\frac{1}{(2l+1)!} |x|^{2l+1}\chi_{[0]}(x);
		$$
		the function $\chi_{[2l+1]}$ admits right and left derivatives of any order at $0$, and $\partial_x^k\chi_{[2l+1]}(0^\pm)=0$ except if $k=2l+1$, in which case $\partial_x^k\chi_{[2l+1]}(0^\pm)=\pm 1$.
	\end{notation}

	We can now define the supplementary space  ${\mathbb V}^{s+1/2}$ of  $\dot{H}^{s+1/2}_{\rm c.c.}(\GD)$ in  $\dot{H}^{s+1/2}_{\rm tr}(\GD)$.
	\begin{Def}\label{defT}
		Let $\Omega$ and $\Gamma^{\rm D}$ be as in Assumption \ref{Assumption domain}, and assume moreover that all the contact angles are equal to $\pi/2$.
		Let also $s\geq 0$.\\ 
		{\bf i.} We define the finite dimensional space ${\mathbb V}^{s+1/2}$ as
		$$
		{\mathbb V}^{s+1/2}={\rm span}\{ \chi_{[2l+1]}(\cdot-x_{\mathtt c})_{\vert_{\GD}}, {\mathtt c}=(x_{\mathtt c},0)\in {\mathtt C}, 2l+1\leq s\},
		$$
		endowed with the norm $\vert f\vert_{{\mathbb V}^{s+1/2}}=\sum_{{\mathtt c}\in {\mathtt C}}\sum_{2l+1\leq s} |\partial_x^{2l+1} f(x_{\mathtt c}) \vert$.\\
		{\bf ii.} We define on $\dot{H}_{\rm tr}^{s+1/2}(\GD)$ a mapping ${\mathcal T}$ by
		$$
		{\mathcal T}: \begin{array}{lcl}
			\dot{H}_{\rm tr}^{s+1/2}(\GD) & \to & {\mathbb V}^{s+1/2}\\
			\psi &\mapsto & \sum_{{\mathtt c}\in {\mathtt C}}\sum_{2l+1\leq s}\sigma({\mathtt c}) \partial_x^{2l+1}\psi(x_{\mathtt c})\chi_{[2l+1]}(\cdot-x_{\mathtt c})_{\vert_{\GD}},
		\end{array}
		$$
		where $\sigma({\mathtt c})=1$ if ${\mathtt c}$ is the left extremity of a connected component of $\GD$, that is, if $x_{\mathtt c}=x_{\rm L}$ or $x_{\mathtt c}=x_{\rm r}$, while $\sigma({\mathtt c})=-1$ if $x_{\mathtt c}=x_{\rm R}$ or $x_{\mathtt c}=x_{\rm l}$.
	\end{Def}
	The proof of the following proposition is straightforward.
	\begin{prop}\label{propVs}
		Let $\Omega$ and $\Gamma^{\rm D}$ be as in Assumption \ref{Assumption domain}, and assume moreover that all the contact angles are equal to $\pi/2$.
		Let also $s\geq 0$. Then one has
		$$
		\dot{H}^{s+1/2}_{\rm tr}(\GD)=\dot{H}^{s+1/2}_{\rm c.c.}(\GD)\oplus {\mathbb V}^{s+1/2},
		$$
		and for al $\psi\in \dot{H}^{s+1/2}_{\rm tr}(\GD)$, the corresponding decomposition is given by
		$$
		\psi = \psi_{\rm c.c.}+{\mathcal T}\psi \quad \mbox{ with }\quad \psi_{\rm c.c.}=\psi -{\mathcal T}\psi \in H^{s+1/2}_{\rm c.c.}(\GD)
		\quad \mbox{ and }\quad {\mathcal T}\psi \in {\mathbb V}^{s+1/2}.
		$$
		In particular, there is a constant $C>0$ independent of $\psi$ such that 
		$$
		\frac{1}{C}\vert \psi\vert_{\dot{H}^{s+1/2}_{\rm tr}(\GD)} \leq  \vert \psi_{\rm c.c.}\vert_{\dot{H}^{s+1/2}_{\rm c.c.}(\GD)}
		+\vert {\mathcal T}\psi\vert_{{\mathbb V}^{s+1/2}} \leq C \vert \psi\vert_{\dot{H}^{s+1/2}_{\rm tr}(\GD)}. 
		$$
	\end{prop}

	The singularities in the variational solution $\psi^{\mathfrak h}$ arise if ${\mathcal T}\psi\neq 0$. They are expressed in terms of the singular functions introduced in the following definition.
	\begin{Def}\label{Def: Generating function}
		{\bf i.} We define on $\R$ smooth functions $G_{2l-1}$, $l\in {\mathbb N}$,  through the relations
		$$
		G_{-1}(x)=- \frac{1}{\pi} \frac{1}{1+x^2}\quad\mbox{ and for all }l\in {\mathbb N}^*, \quad
		\begin{cases}
			\frac{{\rm d}^{2l}}{({\rm d}x)^{2l}}G_{2l-1} =2G_{-1}\\
			\frac{{\rm d}^k}{({\rm d}x)^k}G_{2l-1}(0)=0, \quad 0\leq k\leq 2l-1.
		\end{cases}
		$$
		{\bf ii.} For all $l\in {\mathbb N}$ and $z<0$, we define 
		$$
		\forall x\in \R, \qquad G_{2l-1}[z](x)=z^{2l-1}G_{2l-1}(\frac{x}{z}).
		$$
		{\bf iii.} For all $l\in {\mathbb N}$, we denote by ${\mathcal S}_{2l+1}$ de function defined for all $x\in \R$ and $z<0$ by
		$$
		{\mathcal S}_{2l+1}(x,z)=
		G_{2l+1}[z](x)+2G_{-1}(0)\ln (-z)\sum_{0\leq k'\leq l} (-1)^{l-k'}\frac{x^{2k'}z^{2(l-k')+1}}{(2k')!(2(l-k')+1)!}.
		$$
		{\bf iv.} For all ${\mathtt c}\in {\mathtt C}$, we denote by $\Theta_{\mathtt c}$ a smooth function with compact support in $\overline{\Omega}$, equal to $1$ in the neighborhood of ${\mathtt c}$ and whose support is small enough to have  $\Theta_{\mathtt c}\chi_{[0]}(\cdot -x_{\mathtt c})=\Theta_{\mathtt c}$.
		
	\end{Def} 
	

	We can now state the main result of this section.
	\begin{thm}\label{Thm: sing decomp mixed pb}
		Let $\Omega$ and $\Gamma^{\rm D}$ be as in Assumption \ref{Assumption domain}, and assume moreover that all the contact angles are equal to $\pi/2$.
		Let also $s\geq 0$ and $\psi \in \dot{H}^{s+1/2}_{\rm tr}(\Gamma^{\rm D})$.Then the variational solution $\psi^{\mathfrak h}\in \dot{H}^1(\Omega)$ to \eqref{Eq: psi^H2} can be decomposed under the form
		$$
		\psi^{\mathfrak h}=\psi^{\mathfrak h}_{\rm sing}+\psi^{\mathfrak h}_{\rm reg}
		$$
		where $\psi^{\mathfrak h}_{\rm sing} \in H^{2l+2^-}(\Omega)^2$ for $l \in \N$ such that $2l+1\leq s$ and is defined by
		\begin{equation}\label{Eq: Sing psi}
			\forall (x,z)\in \Omega, \qquad \psi^{\mathfrak h}_{\rm sing}(x,z)= \sum_{{\mathtt c}\in {\mathtt C}}\sum_{2l+1\leq s}\sigma({\mathtt c})\partial_x^{2l+1}\psi(x_{\mathtt c}) \Theta_{{\mathtt c}}(x,z){\mathcal S}_{2l+1}(x-x_{\mathtt c},z),
		\end{equation}
		with $\sigma({\mathtt c})$ as defined in Definition \ref{defT},
		while $\nabla\psi^{\mathfrak h}_{\rm reg}\in {H}^{s}(\Omega)^2$ and satisfies, for some constant $C>0$ independent of $\psi$,
		$$
		\Vert \nabla\psi^{\mathfrak h}_{\rm reg}\Vert_{H^s(\Omega)}\leq C \vert \psi \vert_{\dot{H}^{s+1/2}_{\rm tr}(\GD)}.
		$$ 
	\end{thm}
	
	\begin{remark}\label{Remark: Sing decomp}
		One can compute
		$$
		G_1(x)=-\frac{2}{\pi}\big(x\arctan(x)-\frac{1}{2}\ln(x^2+1)\big)
		$$
		and therefore
		$$
		G_1[z](x)=-\frac{2}{\pi}\big(x\arctan(\frac{x}{z})-\frac{1}{2}z\ln((\frac{x}{z})^2+1)\big).
		$$
		Since $G_{-1}(0)=-\frac{1}{\pi}$, the first singular function ${\mathcal S}_1$ takes the form
		\begin{align*}
			{\mathcal S}_1(x,z)&=G_1[z](x) - \frac{2}{\pi} z\ln (-z) \\
			&=\frac{1}{\pi}\big(z\ln(x^2+z^2)- 2x\arctan(\frac{x}{z})\big).
		\end{align*}  
	\end{remark}


	\begin{proof} 
		According to Proposition \ref{propVs}, we can write
		$$
		\psi^{\mathfrak h}=\psi_{\rm c.c.}^{\mathfrak h}+({\mathcal T}\psi)^{\mathfrak h}.
		$$
		We know by Proposition \ref{Prop: Elliptic reg DN} that $\nabla\psi_{\rm c.c.}^{\mathfrak h}\in H^s(\Omega)$ and that $\Vert \nabla\psi_{\rm c.c.}^{\mathfrak h}\Vert_{H^s(\Omega)}\lesssim \vert \psi_{\rm c.c.}\vert_{\dot{H}^{s+1/2}_{\rm c.c.}}$, so that we just need to decompose $({\mathcal T}\psi)^{\mathfrak h}$ into a singular and a regular part belonging to $H^{s+1}(\Omega)$. By definition of ${\mathcal T}\psi$, we just need to perform such a decomposition on the functions ${\mathbf X}_{[2l+1]}^{\mathtt c}:=\big(\chi_{[2l+1]}(\cdot-x_{\mathtt c})\big)^{\mathfrak h}$, $l\in {\mathbb N}$, where ${\mathtt c}=(x_{\mathtt c},0)$ is a finite corner. 
		Let us for instance assume that ${\mathtt c}={\mathtt c}_{\rm r}=(x_{\rm r},0)$ is the contact point at the right of the object; the other contact points can be treated similarly. Without loss of generality, we can assume that $x_{\rm r}=0$. \\
		Proceeding as in the proof of Lemma \ref{LMell1}, one gets that
		$$
		{\mathbf X}^{{\mathtt c}_{\rm r}}_{[2l+1]} - (1-\chi_-)\exp(z|D|)\chi_{[2l+1]}(\cdot ) \in H^\infty(\Omega).
		$$
		We also get from classical arguments (Lemma 5.1.1.1 in \cite{Grisvard85} for instance) that the singularities are located at the corners, namely, $(1-\sum_{{\mathtt c}\in {\mathtt C}}\Theta_{\mathtt c}) {\mathbf X}^{{\mathtt c}_{\rm r}}_{[2l+1]}\in H^\infty(\Omega)$; observing that 
		$(\sum_{{\mathtt c}\in {\mathtt C}}\Theta_{\mathtt c})(1-\chi_-)=\Theta_{{\mathtt c}_{\rm r}}$, we deduce that
		\begin{equation}\label{decompreg}
			{\mathbf X}^{{\mathtt c}_{\rm r}}_{[2l+1]}-\Theta_{{\mathtt c}_{\rm r}}\exp(z|D|)\chi_{[2l+1]}(\cdot ) \in H^\infty(\Omega).
		\end{equation}
		We now need the following lemma; we recall that ${\mathcal S}_{2l+1}(x,z)$ has been defined in Definition \ref{Def: Generating function}.
		\begin{lemma}
			Let $l \in \N$. Then there exists a function $R_{2l+1}\in C^\infty(\{(x,z)\in \R^2, z\leq 0\})$ such that there holds	
			$$
			\big(\exp(z | D|) \chi_{[2l+1]}\big)(x)={\mathcal S}_{2l+1}(x,z) 
			+R_{2l+1}(x,z).
			$$
		\end{lemma}

		\begin{proof}[Proof of the lemma]
			Recall that for all $z<0$, we have the explicit Fourier transform
			$$
			{\mathcal F}\big[ 
			G_{-1}{[z]}(\cdot)\big](\xi)=\exp( z|\xi|),
			$$
			so that
			$$
			\exp(z | D|) \chi_{[2l+1]} =  
			G_{-1}{[z]}*\chi_{[2l+1]}.
			$$
			Remarking that $\partial_x^{2l+2}\chi_{[2l+1]} =2\delta +r_{2l+1}$, where $\delta$ denotes the Dirac mass and $r_{2l+1}$ a smooth compactly supported function,
			we obtain that
			\begin{align*}
				\partial_x^{2l+2}\big(\exp(z | D|) \chi_{[2l+1]}\big)(x)
				&=
				(G_{-1}{[z]}(\cdot)* \big(\partial_x^{2l+2}\chi_{[2l+1]})\big)(x) \\
				&=
				g(x,z),
			\end{align*}
			where $g(x,z):=2G_{-1}[z](x)+ (G_{-1}{[z]}*r_{2l+1})(x)$. Let us now introduce the induction relation for sequences of functions $(a_{2n+1})_{n\geq -1}$ defined on the half-line $\{z<0\}$,
			\begin{equation}\label{induc}
				a_{-1}(z)=2G_{-1}[z](0),\quad \mbox{ and for all}\quad  n\in {\mathbb N}, \quad a_{2n+1}''(z)=-a_{2n-1}(z);
			\end{equation}
			note that if $(a_{2n+1})_n$ and $(\widetilde{a}_{2n+1})_n$ are two sequences of functions satisfying \eqref{induc}, then for all $n\in {\mathbb N}$, $a_{2n+1}-\widetilde{a}_{2n+1}$ is a polynomial in $z$ of order $2n+1$.
			We want to prove by induction that there exists a sequence of smooth functions $a_{2k+1}$ satisfying \eqref{induc} and such that for all $0\leq k\leq l+1$, one has for all $z<0$,
			$$
			\partial_x^{2(l+1-k)}\big(\exp(z | D|) \chi_{[2l+1]}\big)=\sum_{0\leq k'<k} a_{2(k-k')-1}(z)\frac{x^{2k'}}{(2k')!}+I^{2k}(g(\cdot,z)),
			$$
			where we denote by $I$ the integration operator defined for all $f\in L^1_{\rm loc}(\R)$ by $I(f)(x)=\int_0^x f$.
			If the above property is true for $k\leq l$, then integrating twice with respect to $x$ and remarking that
			$\partial_x^{2(l-k)+1}\big(\exp(z | D|) \chi_{l}\big)(0)=0$ for all $z<0$ because $\partial_x^{2(l-k)}\big(\exp(z | D|) \chi_{l}\big)$ is regular and even, one obtains
			\begin{equation}\label{induchyp}
				\partial_x^{2(l-k)}\big(\exp(z | D|) \chi_{l}\big)=
				\alpha(z)+\sum_{0\leq k'<k} a_{2(k-k')-1}\frac{x^{2k'+2}}{(2k'+2)!}+I^{2k+2}(g(\cdot,z)),
			\end{equation}
			where $
			\alpha(z)=\partial_x^{2(l-k)}\big(\exp(z | D|) \chi_{l}\big)(0)$.  			
			Since the left-hand side is harmonic, and			 using the fact that since $g$ is also harmonic
			and satisfies $\partial_xg(0,z)=0$, as can easily be checked from the definition of $g$, one has
			$$
			\Delta [I^{2(k+1)}(g(\cdot,z))](x)=\frac{x^{2k}}{(2k)!}g(0,z),
			$$
			we obtain after applying the Laplace operator to the above identity that
			\begin{align*}
				0=&\sum_{0\leq k'<k} a_{2(k-k')-1}\frac{x^{2k'}}{(2k')!}
				+\alpha''(z)+\sum_{0\leq k'<k} a''_{2(k-k')-1}\frac{x^{2(k'+1)}}{(2(k'+1))!}+\frac{x^{2k}}{(2k)!} g(0,z).
			\end{align*}
			Recalling that $g(0,z)=2G_{-1}[z](0)=a_{-1}(z)$, we can rearrange this identity as
			\begin{align*}
				\alpha''(z)&=-\big[a_{2k-1}+\sum_{1\leq k'\leq k} (a_{2(k-k')-1}+a''_{2(k-k')+1})\frac{x^{2(k')}}{(2(k'))!}\big]\\
				&=-a_{2k-1},
			\end{align*}
			since the coefficients $a_{-1},\dots a_{2k-1}$ satisfy \eqref{induc}. It follows that if we set $a_{2k+1}=\alpha$, then \eqref{induc} is satisfied up to $n=k+1$. The relation \eqref{induchyp} then shows that the induction assumption is satisfied at $k+1$. Since it is also trivially satisfied at $k=0$, we deduce that it holds for all $1\leq k\leq l+1$. In particular, for $k=l+1$,
			$$
			\big(\exp(z | D|) \chi_{[2l+1]}\big)=\sum_{0\leq k'\leq l} a_{2(l+1-k')-1}\frac{x^{2k'}}{(2k')!}+G_{2l+1}[z]+I^{2l+2}(G_{[-1]}[z]*r_{2l+1}).
			$$
			Now, considering successive primitives of the function $z^{-1}$, one gets that any solution of the induction relation \eqref{induc}   is necessarily of the form
			$$
			a_{2n+1}(z)=2G_{-1}(0)(-1)^{n+1}\frac{z^{2n+1}}{(2n+1)!}\ln (-z)+p_{2n+1}(z),
			$$
			for all $n\in {\mathbb N}$, $a_{2n+1}$ and where $p_{2n+1}$ is a polynomial of order $2n+1$. Since $I^{2l+2}(G_{[-1]}[z]*r_{2l+1})$ is a smooth function on the closed lower half-plane, this concludes the proof of the lemma.
		\end{proof}
		It follows from \eqref{decompreg} and the lemma that
		$$
		{\mathbf X}^{{\mathtt c}_{\rm r}}_{[2l+1]} - \Theta_{{\mathtt c}_{\rm r}}{\mathcal S}_{2l+1}(\cdot-{\mathtt c}_{\rm r}) \in H^\infty(\Omega),
		$$
		from which one easily concludes that \eqref{Eq: Sing psi} holds true. Moreover, we observe from the lemma that $\psi^{\mathfrak{h}}_{\rm sing}$ has the same regularity as $\Theta_{{\mathtt c} }\exp(z | D|) \chi_{[2l+1]}$ and since $\partial_x^{2l+1}\chi_{[2l+1]} \in H^{1/2^-}(\R)$, we deduce that $\psi^{\mathfrak{h}}_{\rm sing} \in H^{2l+2^-}(\Omega)$ for $2l+1\leq s$. 
		
	\end{proof}

	\subsection{Higher order ellipticity of the Dirichlet-Neumann operator}\label{Subsec: Neumann pb pi/2}
	
	It was proved in \cite{LannesMing24} that if $\psi\in \dot{H}^{1/2}(\Gamma^{\rm D})$ is such that $G_0\psi\in L^2(\GD)$, then one has $\psi\in \dot{H}^1(\GD)$, with the estimate
	$$
	\vert \psi \vert_{\dot{H}^1(\GD)}\leq C \big( \vert \psi \vert_{\dot{H}^{1/2}(\GD)}+ \vert G_0\psi \vert_{L^2(\GD)}\big).
	$$
	This result was established through a Rellich identity which does not allow us to extend this ellipticity property to more regular spaces. The following proposition shows that $G_0\psi$ remains elliptic of order one on $H_{\rm c.c.}^{s-1/2}(\GD)$ for all $s\geq 1/2$; the above inequality corresponds to the particular case $s=1/2$.

	\begin{prop}\label{Prop: Elliptic reg Neumann}
		Let $\Omega$, $\Gamma^{\rm D}$ and $\Gamma^{\rm N}$ be as in Assumption \ref{Assumption domain}, and assume moreover that all the contact angles are equal to $\pi/2$.
		Let also $s\geq 1$.  For all $ \psi\in \dot{H}^{1/2}(\Gamma^{\rm D})$ such that $G_0\psi \in H_{\rm c.c.}^{s-\frac{1}{2}}(\Gamma^{\rm D})$,  one actually has  $\psi \in \dot{H}_{\rm c.c.}^{s+1/2}(\Gamma^{\rm D})$ and there holds
		$$
		\vert \psi\vert_{\dot{H}_{\rm c.c.}^{s+1/2}(\GD)}\leq C \big( \vert \psi \vert_{\dot{H}^{1/2}(\GD)}+ \vert G_0\psi \vert_{{H}_{\rm c.c.}^{s-1/2}(\GD)}\big),
		$$
		for some constant $C>0$ independent of $\psi$. 			
	\end{prop}  
	
	\begin{proof}
		As in proof of Proposition \ref{Prop: Elliptic reg DN}, we consider the case where ${\mathcal E}_-= (x_{\rm L},x_{\rm l})$  is finite while  ${\mathcal E}_+ = (x_{\rm r}, \infty)$  is a half-line; the adaptation to the other possible configurations is straightforward. The proof relies on a Neumann extension property and on a regularity result for an elliptic problem with Neumann boundary conditions, with data that vanish near the corners.

		The proof of the following extension lemma is postponed to Appendix \ref{appcompatible} for the sake of clarity.
		\begin{lemma}\label{Lemma: extension Neumann} 
			Under the assumptions of the proposition, 			there exists $h_1>0$ such that for all $s\geq 1/2$ and  all  $ f\in H_{\rm c.c.}^{s - \frac{1}{2}}(\Gamma^{\rm D})$, there exists a function $F\in H^{s+1}(\Omega)$ such that $\Delta F$ is in $H^s(\Omega)$ and vanishes outside $\Omega_{x_{\rm l},x_{\rm r}}$, and  such that $\partial_{\rm n}  F= f$ on $\Gamma^{\rm D}$ and $\partial_{\rm n} F = 0$ on $ \Gamma^{\rm N}\cap \{z > -h_1\}$, and satisfying
			\begin{equation}\label{Eq: Est on nabla F ext in Hk} 
				\| F \|_{H^{s+1}(\Omega) } 
				\leq C | f|_{H^{s-1/2}_{\rm c.c.}(\Gamma^{\rm D})}  
			\end{equation}
			for a constant $C>0$ independent of $f$.
		\end{lemma}

		Using the extension provided in the lemma, with $f = G_0\psi$, we may now study $u = \psi^{\mathfrak h} - F$,
		which solves an elliptic problem of the form
		\begin{equation}\label{BVPNeumann}
			\begin{cases}
				\Delta u = \nabla \cdot \mathbf{h}  &\mbox{ in }\Omega,\\
				\partial_{\rm n} u =0 & \mbox{ on }\Gamma^{\rm D},\\
				\partial_{\rm n}u =n\cdot {\bf h}  & \mbox{ on }\Gamma^{\rm N},
			\end{cases}
		\end{equation}
		where $n=n^{\rm w}$ on $\Gamma^{\rm w}$ and $n=n^{\rm b}$ on $\Gamma^{\rm b}$. If $\nabla\cdot {\bf h} \in L^2(\Omega)$, then the normal traces ${\bf h}\cdot n$ on $\GD$ and $\Gamma^{\rm N}$ are well defined in the weak sense. In particular,  the normal trace ${\bf h}\cdot n$ on $\Gamma^{\rm D}$ makes sense in $\dot{H}^{1/2}(\GD)'$ so that the problem \eqref{BVPNeumann} is variational in $\dot{H}^1(\Omega)$ and the solution satisfies
		\begin{equation}\label{estk0}
			\Vert \nabla u\Vert_{L^2(\Omega)} \lesssim \Vert {\bf h}\Vert_{L^2(\Omega)} 
			+\Vert \nabla\cdot {\bf h}\Vert_{L^2(\Omega)}.
		\end{equation} 
		The following lemma shows that elliptic regularity is true for \eqref{BVPNeumann} in a configuration where $\nabla\cdot{\bf h}$ and ${\bf h}\cdot n$ are supported far from the corners. 
		
		\begin{lemma}\label{Lemma: Neumann Neumann away from corner}
			Under the assumptions of the proposition, there exists $h_1>0$ such that for all $s\geq 1$ and  $ \mathbf{h}\in H^{s}(\Omega)$ such that $\nabla \cdot \mathbf{h}$ is supported in $\Omega_{x_{\rm l},x_{\rm r}}:=\Omega\cap \{x_{\rm l}<x<x_{\rm r}\}$ and  $n\cdot {\bf h}$ is supported on $\Gamma^{\rm N}\cap \{z < -h_1\}$, the solution $u\in \dot{H}^1(\Omega)$ to  \eqref{BVPNeumann} satisfies  $\nabla u\in H^{s}(\Omega)^2$, and moreover
			$$
			\Vert \nabla u \Vert_{\dot{H}^{s}(\Omega)} \leq C 
			\Vert  \mathbf{h}\Vert_{H^{s}(\Omega)},
			$$
			for some constant $C$ independent of ${\bf h}$.
		\end{lemma}
		
		\begin{proof}[Proof of Lemma \ref{Lemma: Neumann Neumann away from corner}] The proof is similar to the one of Lemma \ref{LMell2}, where we consider the same configuration as in Lemma \ref{Lemma: extension Neumann}. The key difference with the proof of Lemma \ref{LMell2}, is that we do not have the Poincaré inequality at our disposal and will therefore require some adaptations. 
			
			It is enough to prove the case when $s=k$ is an integer, since the general case follows by interpolation.  As usual, we take $h_1>0$ small enough in order for $\Gamma^{\rm N}\cap \{z>-h_1\}$ to be a union of vertical segments. We also introduce suitable cut-off functions to correct the solution $u$ by its mean (on a two dimensional domain or on a vertical line), so that we can apply \lq\lq Poincaré-Wirtinger type\rq\rq\: estimates for an unbounded domain. Let 	$\theta$ be a smooth compactly supported function on ${\mathbb R}_-$ such that $\theta(z)=1$ for $-\frac{h_1}{2}\leq z\leq 0$ and  $\theta(z)=0$ for
			$z\leq - h_1$.
			Let also $\chi_{-}$ be a smooth cut-off function such that $\chi_{-}=1$ on ${\mathcal E}_-$ and $\chi_{-}=0$ on ${\mathcal E}_+$, and $\chi_{+}$ be another cut-off function such that $\chi_{+}=1$ for $x<a$ and $\chi_{+}=0$ for $x>a+1$, for some $a>x_{\rm r}$.
			
			Denoting $\Omega_a=\Omega\cap \{x<a+1\}$, we define the averaged quantities $\widetilde{u}$ and $\langle u\rangle$ as in Lemma \ref{Lemma: Poincare type} and \eqref{defavv} respectively. We can now decompose $u$ as 
			\begin{align*}
				u
				= & \big[\chi_{+} \langle u\rangle + (1-\chi_{+})\widetilde{u}]+
				\theta \chi_{-} (u-	\langle u\rangle) 
				+
				\theta \chi_{+} (1-\chi_{-}) (u-	\langle u \rangle) \\
				&+ 
				(1-\theta) \chi_{+} (u-	\langle u \rangle) 
				+
				(1-\chi_{+}) (u - \widetilde{u})
				\\ 
				= & u_{\rm 0}+
				u_{\rm I} + u_{\rm II} + u_{\rm III} + u_{\rm IV}.
			\end{align*} 
			In order to prove that $u$ satisfies the estimate of the lemma for $s=k$ with $k\in {\mathbb N}$, $k\geq 1$, we proceed by induction. We therefore assume that $\nabla u\in H^{l-1}(\Omega)$, for $1\leq l$. We show in steps Steps 0 to 4 that under this assumption, the five components in the above decomposition have their gradient in $H^{l}(\Omega)$. This allows us to conclude in Step 5 by a simple induction.
			\\

			\noindent
			{\bf Step 0.} \textit{Estimate on $u_{\rm 0}=\chi_{+} \langle u\rangle + (1-\chi_{+})\widetilde{u}$}. Proceeding exactly as in Step 0 of the proof of Proposition \ref{Prop: NN high reg small angle}, we obtain that
			$\Vert \nabla u_{\rm 0} \Vert_{{H}^{l}(\Omega)}\lesssim \Vert \nabla u \Vert_{H^{l-1/2}(\Omega)}$. 
			\\

			\noindent
			{\bf Step 1.} \textit{Estimate of $u_{\rm I}=\theta \chi_{-} (u-	\langle u\rangle) $}. The function $u_{\rm I}$ solves the elliptic boundary value
			$$
			\begin{cases}
				\Delta u_{\rm I}=h_{\rm I}  &\mbox{ in } (x_{\rm L},x_{\rm l})\times (-h_1,0)\\
				\partial_{\rm n}u_{\rm I}=0 & \mbox{ on }   (x_{\rm L},x_{\rm l})\times \{0\} \quad\mbox{ and }\quad (x_{\rm L},x_{\rm l})\times \{-h_1\} \\
				\partial_{\rm n}u_{\rm I}=0  & \mbox{ on } \{x_{\rm L}\}\times (-h_1,0) \cup \{x_{\rm L}\}\times (-h_1,0),
			\end{cases}
			$$
			with
			$$h_{\rm I}=2\theta'  \partial_z u + \theta'' ( u- \langle u \rangle),$$
			where we used the fact that $\nabla\cdot {\bf h}=0$ on $ (x_{\rm L},x_{\rm l})\times (-h_1,0)$.
			We let $L=x_{\rm l}-x_{\rm L}$,  and consider as in the proof of Lemma \ref{LMell2} the Boussinesq extensions $u_{\rm I}^{\rm B}$ and $h_{\rm I}^{\rm B}$ (see Section \ref{sectBouss}) of  $u_{\rm I}$ and $h_{\rm I}$;  on the periodic strip  ${\mathcal S}_{-} = \R/(2L{\mathbb Z})\times (-h_1,0)$ these extensions solve
			\begin{equation*}
				\begin{cases}
					\Delta u_{\rm I}^{\rm B}=h_{\rm I}^{\rm B}  &\mbox{ in } {\mathcal S}_{-}\\
					\partial_{\rm n}u_{\rm I}^{\rm B}=0 & \mbox{ on }  \R/(2L{\mathbb Z})\times \{0\} \quad\mbox{ and }\quad \R/(2L{\mathbb Z})\times \{-h_1\}.
				\end{cases}
			\end{equation*}
			For all $1\leq l' \leq l$, estimates on $\nabla \partial_x^{l'} u^{\rm B}_{\rm II} $ can be obtained as follows
			\begin{align*} 
				\| \nabla \partial_x^{l'} u^{\rm B}_{\rm I} \|_{L^2({\mathcal S}_{-})}^2
				\lesssim & 
				\Big{|}
				\int_{{\mathcal S}_{-}} \partial_x^{l'-1} h_{\rm I}^{\rm B}  \partial_x^{l'+1} u^{\rm B}_{\rm I}
				\Big{|} 
				\\ 
				\lesssim &  
				\Big{(}
				\| \partial_z  \partial_x^{l'-1} u \|_{L^2(R_-)} 
				+
				\|  \partial_x^{l'-1}  ( u - \langle u\rangle)\|_{L^2(R_-)}  
				\Big{)}
				\| \partial_x^{l'+1} u^{\rm B}_{\rm I} \|_{L^2({\mathcal S}_{-})},
			\end{align*}
			where $R_-$ is the finite box $R_-=(x_{\rm L},x_{\rm l})\times (-h_1,0))$. Since $R_-\subset \Omega_a$, we can use Poincaré's inequality as in Step 0 to control $\partial_x^{l'-1}  ( u - \langle u\rangle)$ when $l'=1$; all the other terms are controlled directly to obtain $\| \nabla \partial_x^{l'} u^{\rm B}_{\rm I} \|_{L^2({\mathcal S}_{-})} \lesssim \Vert \nabla u\Vert_{H^{l-1}(\Omega)}$. The control of the terms $\nabla\partial_x^{l'}\partial_z^{l''} u^{\rm B}$ with $l'+l''\leq l$ is obtained classically using the equation, so that we finally obtain 
			\begin{align*}
				\| \nabla u_{\rm I} \|_{{H}^{l}(\Omega)}
				& \lesssim  
				\|  \nabla u  \|_{{H}^{l-1}(\Omega)}.
			\end{align*}

			\noindent
			{\bf Step 2.} \textit{Estimate of $u_{\rm II}=\theta \chi_{+} (1-\chi_{-}) (u-	\langle u \rangle)$}. The function $u_{\rm II}$ solves the elliptic boundary value problem
			$$
			\begin{cases}
				\Delta u_{\rm II}=h_{\rm II}  &\mbox{ in } (x_{\rm r},\infty)\times (-h_1,0)
				\\
				\partial_{\rm n}u_{\rm II}=0 & \mbox{ on }   (x_{\rm r}, \infty)\times \{0\} \mbox{ and }\quad (x_{\rm r}, \infty)\times \{-h_1\},\\
				\partial_{\rm n}u_{\rm II}=0  & \mbox{ on } \{x_{\rm r}\}\times (-h_1,0),
			\end{cases}
			$$
			with
			$$
			h_{\rm II} =
			2\nabla(\theta \chi_{+}) \cdot \nabla   u 
			+
			[\theta''\chi_{+}+\theta \chi_{+}''] ( u- \langle u\rangle).
			$$
			The proof is the same as in the previous step, and we obtain 
			$ \Vert u_{\rm II} \|_{{H}^{l}(\Omega)}
			\lesssim 
			\|  \nabla u  \|_{{H}^{l-1}(\Omega)}$.\\

			\noindent
			{\bf Step 3.}  \textit{Estimate of $u_{\rm III}=(1-\theta) \chi_{+} (u-	\langle u \rangle) $.} We argue as in Step 3. of the proof of Lemma \ref{LMell2} where one can find a regular bounded domain $\widetilde{\Omega}\subset \Omega$ whose boundary coincides with the boundary of $\Omega$ for $z<-\frac{h_1}{2}$ and $x \in  \mathrm{supp}(\chi_{+})$. Also let $g^{\rm ext} \in H^{k-\frac{1}{2}}(\partial\widetilde{\Omega})$ be a function that coincides with $(1-\theta) n\cdot {\bf h}$ for  $z<-\frac{h_1}{2}$ and is equal to $0$ for $z>-h_1/2$. The function $u_{\rm III}$ then solves a Neumann problem on a smooth domain, namely,
			$$
			\begin{cases}
				\Delta u_{\rm III}=h_{\rm III}  &\mbox{ in } \widetilde{\Omega}\\
				\partial_{\rm n}u_{\rm III}=g^{\rm ext}  & \mbox{ on }\partial\widetilde{\Omega},
			\end{cases}
			$$
			where 
			$$h_{\rm III}= (1-\theta)\chi_{+}\nabla \cdot \mathbf{h}+2\nabla\big{(} (1-\theta)\chi_{+} \big{)} \cdot \nabla u 
			+
			\big{(} 
			(1-\theta)\chi_{+}'' - \theta'' \chi_{+}
			\Big{)}(u -\langle u \rangle).$$
			The standard elliptic regularity for elliptic problems with Neumann boundary conditions on regular bounded domains then implies 
			\begin{align*}
				\| \nabla u_{\rm III} \|_{H^{l}(\widetilde \Omega)}
				& \lesssim  
				\| h_{\rm III} \|_{H^{l-1}(\widetilde \Omega)} + |g^{\rm ext}|_{H^{l-\frac{1}{2}}(\partial \widetilde{\Omega})}.
			\end{align*}
			Since $u_{\rm III}=0$ on $\Omega\backslash \widetilde{\Omega}$, from the definition of $h_{\rm III}$ and $g^{\rm ext}$, and from the trace theorem, we deduce
			\begin{align*}
				\| \nabla u_{\rm III} \|_{H^{l}(\Omega)} \lesssim  \|\nabla u\|_{H^{l-1}(\Omega)} 
				+ 
				\Vert  \mathbf{h}\Vert_{H^{l}(\Omega)}  .
			\end{align*}

			\noindent
			{\bf Step 4.}  \textit{Estimate of $u_{\rm IV}=(1-\chi_{+}) (u - \widetilde{u})$.} Proceeding as in Step 2 of the proof of Proposition \ref{Prop: NN high reg small angle}, we obtain
			$$
			\Vert \nabla u_{\rm IV}\Vert_{H^l(\Omega)}\lesssim
			\|\nabla u\|_{H^{l-1/2}(\Omega)} 
			+
			\Vert  \mathbf{h}\Vert_{H^{l}(\Omega)}.
			$$ 
			
			\noindent
			{\bf Step 5.} {\it Conclusion.}
			We deduce from the previous steps that
			$$
			\Vert \nabla u \Vert_{H^l(\Omega)}\lesssim 
			\|\nabla u\|_{H^{l-1/2}(\Omega)} +
			\Vert  \mathbf{h}\Vert_{H^{l}(\Omega)}.
			$$
			From the interpolation inequality $\|\nabla u\|_{H^{l-1/2}(\Omega)}\lesssim \|\nabla u\|_{H^{l-1}(\Omega)}^{1/2}\|\nabla u\|_{H^{l}(\Omega)}^{1/2}$ and Young's inequality, we infer that
			$$
			\Vert \nabla u \Vert_{H^l(\Omega)}\lesssim 
			\|\nabla u\|_{H^{l-1}(\Omega)} +
			\Vert  \mathbf{h}\Vert_{H^{l}(\Omega)}.
			$$
			By a finite induction on $1\leq l\leq k$ and starting from  \eqref{estk0} that corresponds to $k=0$, we obtain that 
			$$
			\Vert \nabla u \Vert_{H^k(\Omega)}\lesssim 
			\Vert  \mathbf{h}\Vert_{H^{k}(\Omega)}, 
			$$
			which concludes the proof of the lemma.

		\end{proof}

		Using Lemma \ref{Lemma: Neumann Neumann away from corner}, with ${\bf h}=-\nabla F$ and $F$ the Neumann extension of $f=G_0\psi$ given by Lemma \ref{Lemma: extension Neumann}, we get that $\psi^{\mathfrak h}=u+F$ satisfies $\nabla \psi^{\mathfrak h}\in H^s(\Omega)^2$
		and that 
		\begin{align*}
			\Vert \nabla \psi^{\mathfrak h}\Vert_{H^s(\Omega)}&\lesssim \Vert F \Vert_{H^{s+1}(\Omega)} +\vert G_0\psi \vert_{\dot{H}^{1/2}(\GD)'}\\
			&\lesssim \vert G_0\psi \vert_{H^{s-1/2}_{\rm c.c.}(\Gamma^{\rm D})}+\vert \psi\vert_{\dot{H}^{1/2}(\GD)}.
		\end{align*}
		By the trace theorem on polygons (e.g. Theorem 1.4.9 in \cite{Grisvard85}), we deduce that $\psi=(\psi^{\mathfrak h})_{\vert_{\GD}}$ is in $\dot{H}^{s+1/2}_{\rm c.c.}(\GD)$ and that 
		$\vert \psi\vert_{\dot{H}^{s+1/2}_{\rm c.c.}(\GD)} \lesssim \Vert \nabla \psi^{\mathfrak h}\Vert_{H^s(\Omega)}$, from which the proposition follows. 		%

	\end{proof}

	\subsection{Characterization of $\dot{\mathcal{H}}^{n/2}(\Gamma^{\rm D})$ and  $\mathcal{H}^{n/2}(\Gamma^{\rm D})$ when the contact angles are $\pi/2$}
	\label{Sec: Characterization with vertical walls}
	In Section \ref{Subsec: Well-posedness of Fritz John's model} we saw that the natural function spaces associated with the Fritz John's model are $\dot{\mathcal{H}}^{n/2}(\Gamma^{\rm D})$ and  $\mathcal{H}^{n/2}(\Gamma^{\rm D})$ defined as
	\begin{align*}
		\dot{\mathcal H}^{n/2}(\Gamma^{\rm D}) &=\{ f\in \dot{H}^{1/2}(\Gamma^{\rm D}), \quad \forall 1\leq j\leq n, \quad G_0^{j/2}f \in L^2(\Gamma^{\rm D})\},\\
		{\mathcal H}^{n/2}(\Gamma^{\rm D}) &= L^2(\Gamma^{\rm D})\cap \dot{\mathcal H}^{n/2}(\GD).
	\end{align*}	
	For small angles, we were able to show that   $\dot{\mathcal H}^{n/2}(\Gamma^{\rm D}) =\dot{H}^{n/2}(\Gamma^{\rm D})$ and  ${\mathcal H}^{n/2}(\Gamma^{\rm D}) =H^{n/2}(\Gamma^{\rm D})$. However this is not the case in general. In the case where the contact angles are assumed to be equal to $\pi/2$, these spaces are different because of compatibility conditions, as shown in the proposition below.	
	\begin{prop}\label{propHNdroit}
		Let $\Omega$, $\Gamma^{\rm D}$ and $\Gamma^{\rm N}$ be as in Assumption \ref{Assumption domain}, and assume moreover that all the contact angles are equal to $\pi/2$. Then for all $n\in {\mathbb N}^*$, one has $\dot{\mathcal H}^{n/2}(\Gamma^{\rm D}) =\dot{H}_{\rm c.c.}^{n/2}(\Gamma^{\rm D})$ and  ${\mathcal H}^{n/2}(\Gamma^{\rm D}) =H_{\rm c.c.}^{n/2}(\Gamma^{\rm D})$. 
	\end{prop}
	\begin{proof} 
		We need to show that $\dot{\mathcal H}^{n/2}(\Gamma^{\rm D}) =\dot{H}_{\rm c.c.}^{n/2}(\Gamma^{\rm D})$ for $n\geq 3$.  However, arguing as in the proof of Proposition \ref{propHn} we have that the inclusion $\dot{H}^{n/2}_{\rm c.c.}(\GD)\subset \dot{\mathcal H}^{n/2}(\GD)$ follows by the second point in Proposition \ref{Prop: Elliptic reg DN}. On the other hand, we saw that the reverse inclusion $\dot{\mathcal H}^{n/2}(\GD)\subset \dot{H}^{n/2}_{\rm c.c.}(\GD)$ follows by the estimate in Proposition \ref{Prop: Elliptic reg Neumann}  and therefore completes the proof. 
	\end{proof}

	\subsection{Higher order regularity of the Fritz John's problem with vertical walls}\label{Subsec: John high reg pi/2}
	In Corollary \ref{correg2} we proved a higher regularity of the floating body problem in $\mathbb{X}^n$ under the assumption that the Kirchhoff's potentials introduced in Definition \ref{defKirchhoff} satisfy $(\partial_z \mathbf{K}){|_{\Gamma^{\rm D}}} \in \mathcal{H}^{(n-1)/2}(\Gamma^{\rm D})$. In the case of vertical walls, we can use Proposition \ref{propHNdroit} to identify 
	${\mathcal H}^{(n-1)/2}$ with $H_{\rm c.c.}^{(n-1)/2}(\Gamma^{\rm D})$, and it possible to check whether this condition is satisfied. It turns out that the situation is different for the different components of ${\bf K}$. In the statement of the following lemma, we use the notation $ H^\infty_{\rm c.c.}(\Omega):=\bigcap\limits_{n\geq 1} H_{\rm c.c.}^{(n-1)/2}(\Gamma^{\rm D})$.
	\begin{lemma}\label{lemsingK1}
		Let $\Omega$ and $\Gamma^{\rm D}$ be as in Assumption \ref{Assumption domain}, and assume moreover that all the contact angles are equal to $\pi/2$. Let the Kirchhoff potentials $K_1$, $K_2$, and $K_3$ be defined as in Definition \ref{defKirchhoff}. Then 
		$$
		(\partial_z K_1){|_{\Gamma^{\rm D}}} \notin {H}^{1/2}(\Gamma^{\rm D})
		\quad\mbox{ and }\quad
		(\partial_z K_2){|_{\Gamma^{\rm D}}} \in H^\infty_{\rm c.c.}(\Omega);$$ 
		for $K_3$, the result depends on the position of the center of mass $G=(x_{G},z_{G})$,
		\begin{align*}
			(\partial_z K_3){|_{\Gamma^{\rm D}}} \notin {H}^{1/2}(\Gamma^{\rm D}) &\mbox{ if }z_{\rm G}\neq 0,\\
			(\partial_z K_3){|_{\Gamma^{\rm D}}}  \in H^\infty_{\rm c.c.}(\Omega) &\mbox{ if }z_{\rm G}\neq 0.
		\end{align*}   
	\end{lemma}
	\begin{remark}
		The proof provides the explicit expression of the singularity of $\partial_z K_1$ near the contact points. This is a great numerical interest to get a better precision on the computation of the right-hand side in Newton's equation \eqref{Newton2}.
	\end{remark}
	\begin{proof}
		We first consider the regularity of $K_2$. By definition,  $K_2$ is harmonic and vanishes on $\GD$, and its normal derivative vanishes on $\Gamma^{\rm b}$. Since the walls are assumed to be vertical near the contact points, one also has  $\partial_{\rm n}K_2=\pm{\bf e}_x\cdot{\bf e}_z=0$  on $\Gamma^{\rm N}\cap \{z>-h_1\}$ for $h_1>0$ small enough. It is therefore a direct consequence of  Lemma \ref{LMell2} that
		$K_2\in H^\infty(\Omega)$, and therefore $(\partial_z K_2)_{\vert_{\GD}}\in H^\infty(\GD)$. The fact that $(\partial_z K_2)_{\vert_{\GD}}$ is compatible at any order in the sense of Definition \ref{Def: Hcc}  is then again a consequence from the trace theorem on quadrants (Theorem 1.5.2.4 in \cite{Grisvard85}).

		For $K_1$, since $\partial_{\rm n} K_1=\kappa_1=\pm 1$ on $\Gamma^{\rm w}\cap \{z>-h_1\}$, we can decompose $K_1$ as
		$$
		K_1=K_{1,1}+K_{1,2}+K_{1,3}
		$$
		with  $K_{1,1}(x,z)=(x-x_{\rm l})\chi_{[0]}(x-x_{\rm l})+(x-x_{\rm r})\chi_{[0]}(x-x_{\rm r})$, with $\chi_{[0]}$ as in Notation \ref{notachi}, while $K_{1,2}=-(K_{1,1}(\cdot,0))^{\mathfrak h}$ and $K_{1,3}$ solves
		$$
		\begin{cases}
			\Delta K_{1,3} = h &\mbox{ in }\Omega\\
			K_{1,3} = 0 & \mbox{ on }\Gamma^{\rm D}\\
			\partial_{\rm n} K_{1,3} = g  & \mbox{ on }\Gamma^{\rm N},
		\end{cases}
		$$
		wihere $h=-\Delta K_{1,1}$ and $g=\kappa_2-\partial_{\rm n}K_{1,1}$ are smooth and supported far from the corners. In particular, both $K_{1,1}$ and $K_{1,3}$ are in $H^{\infty}(\Omega)$. On the contrary, since $\partial_x K_{1,1}(x_{\rm r},0)=\partial_x K_{1,1}(x_{\rm l},0)=1$, we know by Theorem \ref{Thm: sing decomp mixed pb} that
		$$
		K_{1,2}(x,z)=\Theta_{c_{\rm r}}{\mathcal S}_1(x-x_{\rm r},z)-\Theta_{c_{\rm l}}{\mathcal S}_1(x-x_{\rm l},z)+R,
		$$
		with $R\in H^{4^-}(\Omega)$. It follows, using the explicit expression of ${\mathcal S}_1$ computed in Remark \ref{Remark: Sing decomp}, that for all $x\in \GD$, 
		$$
		\partial_z K_1(x,0)= 
		\frac{2}{\pi}\big[ \Theta_{{\mathtt c}_{\rm r}}(x,0)\big{(}  \ln(x-x_{\rm r}) + 1\big{)}
		-\Theta_{{\mathtt c}_{\rm l}}(x,0)\big{(}  \ln(x_{\rm l}-x) + 1\big{)} +r,
		$$
		with $r\in H^{5/2^-}(\GD)$. It follows that $(\partial_z K_1)_{\vert_{\GD}}\not\in H^{1/2}(\GD)$.
		
		Let us finally consider the case of $K_3$. Since the wall is vertical near the contact points, we get that on $\Gamma^{\rm w}$, 
		$\partial_x K_3(x_{\rm l},z)=\partial_x K_3(x_{\rm r},z)=z-z_G$. Therefore, $\partial_x K_3$ does not vanish at the contact points if $z_G\neq 0$ and the situation is the same as for $K_1$. Now, if $z_G=0$, one can write
		$
		K_{3}=K_{3,1}+K_{3,2}
		$
		with $K_{3,1}=z(x-x_{\rm l})\chi_{[0]}(x-x_{\rm l})+z(x-x_{\rm r})\chi_{[0]}(x-x_{\rm r})$, which is in $H^\infty(\Omega)$ while $K_{3,2}$ solves
		$$
		\begin{cases}
			\Delta K_{3,2} = h &\mbox{ in }\Omega\\
			K_{3,2} = 0 & \mbox{ on }\Gamma^{\rm D}\\
			\partial_{\rm n} K_{3,2} = g  & \mbox{ on }\Gamma^{\rm N},
		\end{cases}
		$$
		where $h=-\Delta K_{3,1}$ and $g=\kappa_3-\partial_{\rm n}K_{1,1}$ are smooth and supported far from the corners, so that $K_{3,2}$ is also in $H^\infty(\Omega)$. One then concludes the proof as we did for $K_1$.
		
	\end{proof}
	
	As a consequence of the lemma, the solution of Fritz John's problem is more regular if only vertical translation, as well as rotations of  center $G$ if $z_G=0$, are allowed.  In such cases, the problem takes the simpler form
	\begin{equation}\label{FJabstract vertical motion}
		\partial_t \widetilde{\bf U}+{\bf {A}}\widetilde{\bf U}=0
	\end{equation}
	where $\widetilde{{\bf U}}=(\zeta,\widetilde{\mathtt X}^{\rm T},\psi,\widetilde{\mathtt V}^{\rm T})^{\rm T}$ with 
	$$
	\widetilde{\mathtt X}^{\rm T} = (0,{\mathtt X}_2,0) \quad\mbox{ and }\quad \widetilde{\mathtt V}^{\rm T}  = (0,{\mathtt V}_2,0),
	$$
	if only vertical translation is allowed, and
	$$
	\widetilde{\mathtt X}^{\rm T} = (0,{\mathtt X}_2,{\mathtt X}_3) \quad\mbox{ and }\quad \widetilde{\mathtt V}^{\rm T}  = (0,{\mathtt V}_2,{\mathtt V}_3),
	$$
	if $z_G=0$ and rotations of center $G$ are allowed. In all cases, the operator ${\bf A}$ is still given by \eqref{defA}. To state the following corollary, it is convenient to introduce the notation
	$$
	\mathbb{X}^n_{\rm Sob, c.c.} = H_{\rm c.c.}^{n/2}(\Gamma^{\rm D})\times {\mathbb R}^3\times \dot{H}_{\rm c.c.}^{(n+1)/2}(\Gamma^{\rm D})\times {\mathbb R}^3.
	$$
	\begin{cor}\label{Cor: vertical walls}  
		Let $\Omega$ and $\Gamma^{\rm D}$ be as in Assumption \ref{Assumption domain}, and assume moreover that all the contact angles are equal to $\pi/2$. Suppose also that the coefficients of $\mathcal{C}$ satisfy Assumption \ref{hypStab} and that we are in one of the following situations
		\begin{itemize}
			\item The object is only allowed to move vertically,
			\item The center of mass $G$ is such that $z_G=0$ and the object is only allowed to move by vertical translation and rotations around $G$,
		\end{itemize}
		and we consider the corresponding reduced problem \eqref{FJabstract vertical motion}.\\
		Then for all $n \in \N$ and  $\widetilde{\bf U}^{\rm in}\in {\mathbb X}_{\rm Sob, c.c.}^n$, there exists a unique solution $\widetilde{\bf U}\in \cap_{j=0}^n C^j({\mathbb R}_+;{\mathbb X}_{\rm Sob, c.c.}^{n-j})$ to  \eqref{FJabstract vertical motion} with initial condition $\widetilde{\bf U}(0)=\widetilde{\bf U}^{\rm in}$, and one also has  $\widetilde{\mathtt V}=\dot{\widetilde{\mathtt X}}\in C^{n+1}({\mathbb R}_+;{\mathbb R}^3)$; moreover, there exists $C>0$ independent of $\widetilde{\bf U}^{\rm in}$ such that
		$$
		\forall t\geq 0, \qquad
		\sum_{j=0}^n \Vert \partial_t^j \widetilde{\bf U}\Vert_{{\mathbb X}_{\rm Sob, c.c.}^{n-j}}(t) +\vert \widetilde{\mathtt V}^{(n+1)}(t)\vert \leq C \Vert \widetilde{\bf U}^{\rm in}\Vert_{{\mathbb X}_{\rm Sob, c.c.}^n}.
		$$
	\end{cor} 
	
	\begin{remark}
		Similarly, one obtains a reduced Cummins equation on $\widetilde{\mathtt X}$. If we consider a smooth initial wave field, compactly supported away from the object, then the solution $\widetilde{\mathtt X}$ to the reduced Cummins equation is $C^\infty$. However, as soon as horizontal translations are allowed, the solution is not better than $C^3$ because of the singular behavior of $K_1$ exhibited in Lemma \ref{lemsingK1}.
	\end{remark}

	\appendix 
	\section{Derivation of the linearized Newton's equations} \label{AppderN}

	When the object is in forced motion, that is, when the velocity ${\mathcal V}^{\rm w}$ of the boundary of  the object is a given function, then the equations \eqref{FJ1}-\eqref{Eq: Laplace equation} are sufficient to describe the motion of the waves generated by and interacting with the objects. This is no longer true when the object is freely floating: the velocity ${\mathcal V}^{\rm w}$ of the solid must then obtained through Newton's equations (or rather their linear approximation for the model considered here). To this end, let us introduce the following notations. 
	\begin{notation}\label{notaGetc}
		We make the following definitions:
		\begin{itemize}
			\item [-] We denote by $G_{\rm eq}=(x_{G,{\rm eq}}, z_{G,{\rm eq}})^{\rm T}$ the position of the center of gravity of the object at equilibrium, and by 
			$$G(t)=(x_G(t),z_G(t))^{\rm T}$$
			its position at time $t$. 
			\item [-] We also denote by $B$ the center of buoyancy of the object defined as
			$$\hspace{1.5cm}
			B=(x_B,z_B)^{\rm T}
			\quad\mbox{ with }\quad
			x_B=\frac{1}{|\Omega^{\rm w}|}\int_{\Omega^{\rm w}}x{\rm d}x{\rm d}z
			\quad\mbox{ and }\quad 
			z_B=\frac{1}{|\Omega^{\rm w}|}\int_{\Omega^{\rm w}}z {\rm d}x{\rm d}z,
			$$
			where $\Omega^{\rm w}$ is the volume of fluid displaced by the object in the equilibrium configuration, that is, the volume delimited by $\Gamma^{\rm w}$ and the horizontal axis $\{z=0\}$. 
			\item [-] We write $\widetilde{x}(t)$ and $\widetilde{z}(t)$  the horizontal and vertical displacements of the center of mass from its equilibrium position, so that  
			$$G(t)=G_{\rm eq}+(\widetilde{x}(t),\widetilde{z}(t))^{\rm T}.$$ 
			\item [-] We denote by $\theta(t)$ the rotation of the object with respect to its equilibrium position around the $(Oy)$ axis perpendicular to the plane: this means that in the $(Oxz)$ plane, we use the anti-trigonometrical orientation for $\theta$. 
			\item [-] We write 
			$${\mathtt X}(t)=(\widetilde{x}(t),\widetilde{z}(t),\theta(t))^{\rm T}\quad \text{and} \quad {\mathtt V}(t)=(\dot{\widetilde{x}}(t),\dot{\widetilde{z}}(t),\dot\theta(t))^{\rm T}.
			$$ 
			We also use the notation $$ {\mathcal V }_G(t)=(\dot{\widetilde{x}}(t),\dot{\widetilde{z}}(t))^{\rm T}\quad  \text{and} \quad \omega(t)=\dot\theta(t).  $$
			\item [-] For all $(x,z)\in \Gamma^{\rm w}$, we define $${\bf r}_{G_{\rm eq}}(x,z)=(x-x_{G,{\rm eq}},z-z_{G,{\rm eq}})^{\rm T}.$$
		\end{itemize}
	\end{notation}
	\begin{remark}
		The domain $\Omega^{\rm w}$ is the volume of fluid displaced by the object in the equilibrium configuration, that is, the volume delimited by $\Gamma^{\rm w}$ and the horizontal axis $\{z=0\}$. By Archimedes' law, we know that
		\begin{equation}\label{Archimedes}
			\rho |\Omega^{\rm w}|={\mathfrak m}
			\quad\mbox{ and }\quad
			x_B=x_{G_{\rm eq}},
		\end{equation}
		where ${\mathfrak m}$ is the mass of the object;  this means that the mass of the volume of fluid displaced by the object at equilibrium is equal to the mass of the object and that the center of buoyancy and the center of gravity must be on the same vertical plane.
	\end{remark}
	The velocity of a point at the boundary of the solid can be expressed in terms of the motion of its center of gravity ${\mathcal V}_G$ and its angular velocity $\omega$, which in turn are expressed through the linearized Newton equation. Before we turn to the derivation, we need to express the velocity on the boundary of the object.

	\subsection{Velocity of a point of the boundary of the object}
	We recall that we denote the position at time $t$ of the center of gravity of the object by $G(t)=G_{\rm eq}+(\widetilde{x}(t),\widetilde{z}(t))^{\rm T}$, where $\widetilde{x}(t)$ and $\widetilde{z}(t)$ denote the horizontal and vertical displacement of the center of mass from its equilibrium position. We also denote by $\theta(t)$ the rotation of the object with respect to its equilibrium position around the $(Oy)$ axis perpendicular to the plane: this means that in the $(Oxz)$ plane, we use the anti-trigonometrical orientation for $\theta$. Finally, denoting by $M(t)=(X(t),Z(t))^{\rm T}$ the position at time $t$ of a particular point of the boundary which occupies the position $M_{\rm eq}=(X_{\rm eq},Z_{\rm eq})^{\rm T}$ at equilibrium, we have
	\begin{align}
		\nonumber
		\begin{pmatrix}
			X(t)-x_{G}(t) \\ Z(t)-Z_{G}(t) 
		\end{pmatrix}
		&=\begin{pmatrix}
			\cos(\theta(t) -\theta(0)) & \sin(\theta(t)-\theta(0)) \\ 
			-\sin(\theta(t)-\theta(0)) & \cos(\theta(t)-\theta(0))
		\end{pmatrix}
		\begin{pmatrix}
			X(0)-x_{G}(0) \\ Z(0)-Z_{G}(0) 
		\end{pmatrix}  \\
		\label{relvel}
		& 
		=\begin{pmatrix}
			\cos(\theta(t) ) & \sin(\theta(t)) \\ 
			-\sin(\theta(t)) & \cos(\theta(t))
		\end{pmatrix}
		\begin{pmatrix}
			X_{\rm eq}-x_{G,{\rm eq}} \\ Z_{{\rm eq}}-z_{G,{\rm eq}} 
		\end{pmatrix}  .
	\end{align}
	Time differentiating the first identity, one deduces an expression for the velocity ${\mathcal V}^{\rm w}$  of a point of the boundary
	\begin{equation}\label{eqUsol}
		\forall (x,z)\in \Gamma^{\rm w}, \qquad {\mathcal V}^{\rm w}(t,x,z)={\mathcal V}_{G}(t)-\omega(t) {\bf r}_{G(t)}^\perp(x,z),
	\end{equation}
	where ${\mathcal V}_{G}(t)=(\dot{\widetilde{x}(}t),\dot{\widetilde{z}}(t))^{\rm T}$ and $\omega(t)=\dot\theta(t)$ denote the velocity of the center of mass the angular velocity, while 
	${\bf r}_{G(t)}(x,z)=(x-x_{G(t)},z-z_{G(t)})^{\rm T}$.
	
	\subsection{Newton's laws}
	Both ${\bf U}_{G}$ and $\omega$ are determined by Newton's laws; in order to state these laws in the present context, let us note that the resulting force and torque exerted by the atmosphere and the fluid on the object are, respectively
	$$
	{\bf F}=\int_{\Gamma^{\rm w}(t)}(P-P_{\rm atm}) {\bf n}^{\rm w}(t) {\rm d}\Gamma^{\rm w}(t) \quad\mbox{ and }\quad
	T=-\int_{\Gamma^{\rm w}(t)}(P-P_{\rm atm}) {\bf r}_{G(t)}^\perp\cdot {\bf n}^{\rm w}(t) {\rm d}\Gamma^{\rm w}(t), 
	$$
	where $\Gamma^{\rm w}(t)$ is the immersed part of the object at time $t$, ${\bf n}^{\rm w}(t)$ the outward unit vector to $\Gamma^{\rm w}(t)$. Newton's laws then take the following form
	$$
	\begin{cases}
		\mathfrak{m}
		\dot{\mathcal V}_{G}
		& =
		-\mathfrak{m} {\mathtt g} \mathbf{e}_z + {\mathbf F} ,
		\\ 
		\mathfrak{i} \dot{\omega} & = T,
	\end{cases}
	$$
	where ${\mathfrak m}$ is the mass of the object and ${\mathfrak i}$ its inertia coefficient.
	
	\subsection{Linearized Newton's laws}
	Since we are interested in small motions of the object around its equilibrium position, we derive linear approximations of the above formulas. Throughout this section, we write $A\sim B$ if $A=B$ up to quadratic and higher order terms in $\widetilde{x}$, $\widetilde{z}$ and $\theta$. We first get
	\begin{equation}\label{eqUsollin}
		{\mathcal V}^{\rm w}(t,x,z)\sim{\mathcal V}_{G}(t)-\omega(t) {\bf r}_{G_{\rm eq}}^\perp(x,z),
	\end{equation}
	Since the pressure is given by Bernoulli's law,
	$$
	-\frac{1}{\rho}(P-P_{\rm atm)}=\partial_t \phi+{\mathtt g}z,
	$$
	we can decompose
	$$
	{\bf F}=-\rho \int_{\Gamma^{\rm w}(t)} \partial_t \phi{\bf n}^{\rm w}(t) {\rm d}\Gamma_j^{\rm w}(t)
	-
	\rho \int_{\Gamma^{\rm w}(t)} {\mathtt g}z {\bf n}^{\rm w}(t) {\rm d}\Gamma^{\rm w}(t),
	$$
	and a similar decomposition for $T$. In the linear approximation, the integral of the first term can be taken over the position $\Gamma^{\rm w}$ of the immersed part of the object at equilibrium,
	$$
	-\rho \int_{\Gamma^{\rm w}(t)} \partial_t \phi{\bf n}^{\rm w}(t) {\rm d}\Gamma^{\rm w}(t) \sim
	- \rho \int_{\Gamma^{\rm w}} \partial_t \phi{\bf n}^{\rm w} {\rm d}\Gamma^{\rm w},
	$$
	but this is not the case for the second component. In order to describe its contribution, one must derive an approximation of the position of $\Gamma^{\rm w}(t)$ in terms of $\widetilde{x}$, $\widetilde{z}$, and $\theta$. Neglecting quadratic and higher order terms in \eqref{relvel}, one gets
	\begin{align*}
		X(t) & \sim X_{\rm eq} + \widetilde{x}(t)+\theta(t) (Z_{\rm eq}-z_{G,{\rm eq}}),  \\
		Z(t)& \sim Z_{\rm eq} +  \widetilde{z}(t)-\theta(t) (X_{\rm eq}-x_{G,{\rm eq}});
	\end{align*}
	similarly, one checks that ${\bf n}^{\rm w}(t) \sim {\bf n}^{\rm w}-\theta(t) ({\bf n}^{\rm w} )^\perp$. Therefore, the second component of ${\bf F}$ can be approximated as
	$$
	-\rho \int_{\Gamma^{\rm w}(t)} {\mathtt g}z {\bf n}^{\rm w}(t) {\rm d}\Gamma^{\rm w}(t)
	\sim
	-\rho \int_{\Gamma^{\rm w}} {\mathtt g} \big[ z + \widetilde{z}(t) - \theta(t) (x-x_{G,{\rm eq}})\big] {\bf n}^{\rm w}{\rm d}\Gamma^{\rm w} +   \theta(t)  \rho \int_{\Gamma^{\rm w}} {\mathtt g}  z ({\bf n}^{\rm w})^\perp{\rm d}\Gamma^{\rm w}.
	$$
	If we denote by $\Omega^{\rm w}$ the region delimited by $\Gamma_j^{\rm w}$ and the horizontal axis $\{ z=0 \} $, which corresponds to the area of fluid displaced by the solid in the equilibrium configuration, we have by Green's identity and Archimedes' law \eqref{Archimedes} that
	$$
	-\rho \int_{\Gamma^{\rm w}(t)} {\mathtt g}z {\bf n}^{\rm w}(t) {\rm d}\Gamma^{\rm w}(t)
	\sim \big( {\mathfrak m} {\mathtt g} -\rho {\mathtt g}(x_{\rm r}-x_{\rm l})\widetilde{z} 
	+ \rho {\mathtt g (}\int_{x_{\rm l}}^{x_{\rm r}}(x-x_{\rm eq}){\rm d}x) \theta \big) {\bf e}_z.
	$$
	From these approximations of the two components of ${\mathbf F}$, we infer that
	$$
	{\mathbf F} \sim
	-\rho \int_{\Gamma^{\rm w}} \partial_t \phi{\bf n}^{\rm w} 
	+
	\big ( {\mathfrak m} {\mathtt g} -\rho {\mathtt g}(x_{\rm r}-x_{\rm l})\widetilde{z}
	+ \rho {\mathtt g (}\int_{x_{\rm l}}^{x_{\rm r}}(x-x_{\rm eq})) \theta \big) {\bf e}_z .
	$$
	Proceeding similarly for the torque, we first observe  that ${\bf r}_{G(t)}(X,Z)={\bf r}_{G_{\rm eq}}(X_{\rm eq},Z_{\rm eq})-\theta {\bf r}_{G_{\rm eq}}(X_{\rm eq},Z_{\rm eq})^\perp$, and therefore ${\bf r}_{G(t)}(X,Z)\cdot {\bf n}^{\rm w}(t)^\perp\sim{\bf r}_{G_{\rm eq}}(X_{\rm eq},Z_{\rm eq})\cdot ({\bf n}^{\rm w})^\perp$, so that
	$$
	T \sim \int_{\Gamma^{\rm w}} \partial_t \phi {\bf r}_{G}^\perp\cdot {\bf n}^{\rm w} {\rm d}\Gamma^{\rm w}
	+\rho \int_{\Gamma^{\rm w}} {\mathtt g} \big[ z + \widetilde{z}(t) - \theta(t) (x-x_{G,{\rm eq}})\big]  {\bf r}_{G}^\perp\cdot {\bf n}^{\rm w} {\rm d}\Gamma^{\rm w}
	.
	$$
	Since $\int_{\Gamma^{\rm w}} z {\bf r}_G^\perp\cdot {\bf n}^{\rm w}=-\int_{\Omega^{\rm w}}(x-x_{G,{\rm eq}})=|\Omega^{\rm w}|(x_B-x_{G_{\rm eq}})=0$ by Archimedes' law, and using Green's identity, we get
	\begin{align*}
		T \sim &\int_{\Gamma^{\rm w}} \partial_t \phi {\bf r}_{G}^\perp\cdot {\bf n}^{\rm w} {\rm d}\Gamma^{\rm w}
		+\rho {\mathtt g}(\int_{x_{\rm l}}^{x_{\rm r}}(x-x_{G,{\rm eq}}){\rm d}x) \widetilde{z} \\
		&-\rho {\mathtt g}\big(\int_{x_{\rm l}}^{x_{\rm r}}(x-x_{G,{\rm eq}})^2{\rm d}x 
		+ \int_{\Omega^{\rm w}}(z-z_{G,{\rm eq}}){\rm d}x{\rm d}z\big) \theta.
	\end{align*}
	Denoting ${\mathtt X}(t)=(\widetilde{x}(t),\widetilde{z}(t),\theta(t))^{\rm T}$ and ${\mathtt V}(t)=\dot{\mathtt X}(t)$
	and with $\boldsymbol{\kappa}=(\kappa_1,\kappa_2,\kappa_3)^{\rm T}$ as defined in \eqref{defkappaC}, the linearized Newton's equations can therefore be written as
	$$
	{\mathcal M}\dot{\mathtt V}=-\rho \int_{\Gamma^{\rm w}}\partial_t \phi \boldsymbol{\kappa}{\rm d}\Gamma^{\rm w}-{\mathcal C}{\mathtt X},
	$$
	with ${\mathcal M}={\rm diag}({\mathfrak m},{\mathfrak m},{\mathfrak i})$ and 
	$$
	{\mathcal C}=\begin{pmatrix}  0 & 0 & 0 \\ 
		0 & c_{zz} & -c_{z\theta} \\
		0 & -c_{\theta z} & c_{\theta\theta}
	\end{pmatrix},
	$$
	and
	\begin{align*}
		c_{zz}&=\rho {\mathtt g}(x_{\rm r}-x_{\rm l}), \\
		c_{z\theta}&=c_{\theta z}=\rho {\mathtt g} \int_{x_{\rm l}}^{x_{\rm r}}(x-x_{G,{\rm eg}}){\rm d}x, \\
		c_{\theta\theta}&=\rho {\mathtt g} \int_{x_{\rm l}}^{x_{\rm r}}(x-x_{G,{\rm eq}})^2{\rm d}x 
		+ {\mathfrak m}{\mathtt g}(z_B-z_{G,{\rm eq}}),
	\end{align*} 
	where we used Archimedes' law and the definition of the center of buoyancy $B$ to simplify the expression for $c_{\theta\theta}$. 
	\begin{remark}\label{analC}
		{\bf i.} If the object is symmetric with respect to the vertical axis passing through $x_{G_{\rm eq}}$, then $c_{z\theta}=c_{\theta z}=0$.\\
		{\bf ii.} Under the condition
		$
		c_{zz}c_{\theta\theta}- c_{z\theta}^2>0
		$,
		the matrix ${\mathcal C}$ is symmetric positive, and there exists $c_0>0$ such that ${\mathtt X}\cdot {\mathcal C}{\mathtt X}\geq c_0 ( \widetilde{z}^2 + \theta^2)$.\\
		{\bf iii.} One can rewrite $c_{\theta\theta}$ under the form
		$$
		c_{\theta\theta}={\mathfrak m}{\mathtt g}(z_{\rm MG}-z_{G,{\rm eq}}) 
		\quad \mbox{ with }\quad z_{\rm MG}=z_B+\frac{1}{|\Omega^{\rm w}|}\int_{x_{\rm l}}^{x_{\rm r}}(x-x_{G,{\rm eq}})^2{\rm d}x;
		$$
		the quantity $z_{\rm MG}$ is called \emph{meta-center} \cite{Moore67}. The positivity condition of the previous point can therefore be expressed as a condition on the relative position of the meta-center and the vertical coordinate of the center of gravity. If the object is symmetric, it simply states that the meta-center must be above the center of gravity.
	\end{remark}
	\color{black}

	Lastly, we recall that ${\mathtt V}$ are coupled with the fluid equations via ${\mathcal V}^{\rm w}$, where we obtained in \eqref{eqUsollin} the following linear approximation:
	\begin{equation}\label{eqUsol0}
		\forall (x,z)\in \Gamma^{\rm w},\qquad
		{\mathcal V}^{\rm w}(t,x,z)={\mathcal V}_G(t)-\omega(t) {\bf r}_{G_{\rm eq}}^\perp(x,z).
	\end{equation}

	\section{Proof of Lemmas \ref{LMell1} and \ref{Lemma: extension Neumann}}\label{appcompatible}
	
	The goal of this section is to prove the extension Lemmas \ref{LMell1} and \ref{Lemma: extension Neumann}, whose statements are reproduced below for the sake of clarity. We recall that $\Omega_{x_{\rm l},x_{\rm r}}=\Omega \cap \{(x,z), x_{\rm l}<x<x_{\rm r}\}$.
	\begin{lemma}\label{LMell1bis}
		Let $\Omega$, $\Gamma^{\rm D}$ and $\Gamma^{\rm N}$ be as in Assumption \ref{Assumption domain}, and assume moreover that all the contact angles are equal to $\pi/2$.
		There exists $h_1>0$ such that for all $s\geq 0$ and $\psi\in \dot{H}_{\rm c.c.}^{s+\frac{1}{2}}(\Gamma^{\rm D})$, there exists a function $\psi^{\rm ext}\in \dot{H}^{s+1}(\Omega)$, harmonic in $\Omega\backslash\overline{\Omega_{x_{\rm l},x_{\rm r}}}$,  such that $\psi^{\rm ext}=\psi$ on $\Gamma^{\rm D}$ and $\partial_n \psi^{\rm ext}=0$ on $\Gamma^{\rm N}\cap \{z>-h_1\}$, and such that
		\begin{equation}\label{Eq: Est on nabla psi ext in Hk App. C}
			\Vert \nabla \psi^{\rm ext}\Vert_{H^s(\Omega)}
			\leq C \vert \psi\vert_{\dot{H}_{\rm c.c.}^{s+\frac{1}{2}}(\Gamma^{\rm D})},
		\end{equation}
		for some constant $C$ independent of $\psi$.
	\end{lemma}
	
	\begin{lemma}\label{Lemma: extension Neumann2} 
		Let $\Omega$, $\Gamma^{\rm D}$ and $\Gamma^{\rm N}$ be as in Assumption \ref{Assumption domain}, and assume moreover that all the contact angles are equal to $\pi/2$.
		There exists $h_1>0$ such that for all $s\geq 1/2$ and  all  $ f\in H_{\rm c.c.}^{s - \frac{1}{2}}(\Gamma^{\rm D})$, there exists a function $F\in H^{s+1}(\Omega)$ such that $\Delta F$ is in $H^s(\Omega)$ and vanishes outside $\Omega_{x_{\rm l},x_{\rm r}}$, and   such that $\partial_{\rm n}  F= f$ on $\Gamma^{\rm D}$ and $\partial_{\rm n} F = 0$ on $ \Gamma^{\rm N}\cap \{z > -h_1\}$, and satisfying the estimate
		\begin{equation}\label{Eq: Est on nabla F ext in Hk2} 
			\| F \|_{H^{s+1}(\Omega) } 
			\leq C | f|_{H^{s-1/2}_{\rm c.c.}(\Gamma^{\rm D})} 
		\end{equation}
		for a constant $C>0$ independent of $f$.
	\end{lemma}   
	
	The idea of the proof of Lemma \ref{LMell1bis}  is to use the Boussinesq extension of the restriction of $\psi$ to each connected component of $\Gamma^{\rm D}$. This Boussinesq extension can be classically extended through a Poisson kernel whose regularity depends on the compatibility conditions used to define the space $\dot{H}_{\rm c.c.}^{s+1/2}(\GD)$. The proof of Lemma \ref{Lemma: extension Neumann2} follows identical lines, but the Poisson kernel is not the same.
	
	We first introduce in Section \ref{appsectsp} the spaces $\dot{H}^{s+1/2}$ on the full line and on the torus, and then introduce the Boussinesq extension in Section \ref{sectBouss}. We then derive conditions on a function defined on an interval or a half-line that ensure that its Boussinesq extension is regular. We then prove Lemma \ref{LMell1bis} in Section \ref{appprooflm} and Lemma \ref{Lemma: extension Neumann2} in Section \ref{appLMell2}.

	\subsection{The spaces $\dot{H}^{s+1/2}({\mathbb R}_{2L})$ and $\dot{H}^{s+1/2}({\mathbb R})$	}\label{appsectsp}
	
	For the sake of clarity, we denote ${\mathbb R}_{2L}={\mathbb R}/{(2L{\mathbb Z})}$. If $f$ is a locally integrable function on ${\mathbb R}_{2L}$, we classically define its Fourier coefficients by
	$$
	\forall n\in {\mathbb Z}, \quad \widehat{f}(n)=\frac{1}{L}\int_{-L}^L f(x)e^{-i  n k_Lx}{\rm d}x,
	\quad\mbox{ with }\quad
	k_L:=\frac{\pi}{L}.
	$$
	We can now define $\dot{H}^{s+1/2}({\mathbb R}_{2L})$.
	\begin{Def}\label{Def: HsR 2L}
		Let $L>0$ and  $s\geq0$. We define the space  $\dot{H}^{s+1/2}({\mathbb R}_{2L})$ as
		$$
		\dot{H}^{s+1/2}({\mathbb R}_{2L})
		= \{ f\in L^1_{\rm loc}({\mathbb R}_{2L}) \: : \: \vert f \vert_{\dot{H}^{s+1/2}({\mathbb R}_{2L})}<\infty\},
		$$
		where the semi-norm $\vert f \vert_{\dot{H}^{s+1/2}({\mathbb R}_{2L})}$ is defined by
		$$
		\vert f \vert_{\dot{H}^{s+1/2}({\mathbb R}_{2L})}^2 = \sum_{n\in {\mathbb Z}} |n k_L|^2(1+|n k_L|^2)^{s-1/2}  |\widehat{f}(n)|^2.
		$$
	\end{Def}
	%
	%
	%
	The case of the full line is treated below and is defined via the Fourier transform: 
	\begin{equation}\label{Def: Fourier trasform}
		\hspace{-1cm}\forall \xi \in \R, \quad \hat{f} (\xi) = \int_{\R} e^{-  i x \xi } f(x) \: \mathrm{d}x.
	\end{equation}
	\begin{Def}\label{Def: HsR}
		Let $s\geq 0$. We define the space $\dot{H}^{s+1/2}(\R)$ by
		\begin{equation*}
			\dot{H}^{s+1/2}(\R) = \{f \in \mathcal{S}'(\R) \: : \: |f|_{\dot{H}^{s+1/2}(\R)} <\infty\},
		\end{equation*}
		where the semi-norm $|f|_{\dot{H}^{s+1/2}(\R)}$ is defined by
		\begin{equation*}
			|f|_{\dot{H}^{s+1/2}(\R)}^2 = \int_{\R} |\xi|^2 (1+|\xi|^2)^{s-1/2}|\hat{f}(\xi)|^2 \: \mathrm{d}\xi
			=\vert \partial_x f\vert_{H^{s-1/2}(\R)}^2.
		\end{equation*}
	\end{Def}
	We will now identify the necessary conditions on $f\in \dot{H}^{s+1/2}(I)$ in order to define an extension  on $\dot{H}^{s+1/2}({\mathbb R}_{2L})$ when $I$ is finite and on $\dot{H}^s(\R)$ when $I$ is a half-line. Then we will use this characterization to study the regularity of solutions to Laplace problems with regular boundary data.

	\subsection{The Boussinesq extension}\label{sectBouss}
	If $I$ is an open interval or a half-line, and $f\in \widetilde{L}_{\rm loc}^1(I)$, we can define its Boussinesq extension $f^{\rm B}$ as follows: reflect $f$ across one extremity of $I$ and, when $I$ is finite, periodize this extension. Such an extension has, for instance, been used in \cite{AlazardBurqZuily_Wall_16} in the context of water waves.
	\begin{Def}\label{DefBoussinesqext} The Boussinesq extension $f^{\rm B}$ of a function $f\in \widetilde{L}_{\rm loc}^1(I)$ is defined by:\\
		{\bf i.} If $I=(a,b)$, with $a<b$, we write $L=b-a$ and associate an almost everywhere unique function $f^{\rm B}\in L^1({\mathbb R}_{2L})$ such that
		\begin{equation*}
			\qquad f^{\rm B}=f \qquad \text{on}  \quad I
		\end{equation*}
		and on $(2a-b,a)$,$f^{\rm B}$ is the reflection of $f$ with respect to $a$: 
		\begin{equation*}
			\quad \qquad  f^{\rm B}(x) = f(2a-x) \qquad \text{for}  \quad x\in (2a-b,a).
		\end{equation*}
		{\bf ii.} If $I=(a,\infty)$ (resp. $(-\infty,a)$), with $a\in {\mathbb R}$, then  we denote by $f^{\rm B}\in L_{\rm loc}^1({\mathbb R})$ the function deduced from $f$ by an even reflection  with respect to $a$.
	\end{Def}
	%
	%
	%
	Clearly, the fact that $f\in \dot{H}^{s+1/2}(I)$ does not imply that $f^{\rm B}\in \dot{H}^{s+1/2}({\mathbb R}_{2L})$ when $I$ is finite or $\dot{H}^{s+1/2}({\mathbb R})$ when $I$ is a half-line.  This is, however, true if certain conditions are imposed on $f$ at the boundary of $I$. 
	\begin{Def}\label{Def: Compatability conditions}
		Let $s\geq 0$. If $I$ is a non empty interval or a half-line and $a\in {\mathbb R}$ is an extremity of $I$, we say that $f\in \dot{H}^{s+1/2}(I)$ is compatible at $a$ if the following holds
		$$
		\begin{cases}
			\partial_x^{2l+1} f(a)=0 & \mbox{ for all } l\in {\mathbb N} \mbox{ such that } 2l+1<s,\\
			| x-a|^{-1/2}\partial_x^s f \in L^2(I) & \mbox{ if } s\in 2{\mathbb N}+1.
		\end{cases}
		$$ 
	\end{Def}
	We can now state the following result on the regularity of the Boussinesq extension.
	\begin{prop}\label{propositionregext}
		Let $I$ be a nonempty open interval or a half-line, $s\geq 0$, and  $f\in \dot{H}^{s+1/2}(I)$. The following holds,
		\begin{itemize}
			\item  [1.] If $I=(a,b)$, with $a<b$ finite numbers, and if $f$ is compatible at order $s$ at $a$ and $b$, then its
			Boussinesq extension $f^{\rm B}$ belongs to $\dot{H}^{s+1/2}({\mathbb R}_{2L})$, with $L=b-a$, and there holds,
			\begin{equation}\label{Eq: fB extension}
				\vert f^{\rm B}\vert_{\dot{H}^{s+1/2}({\mathbb R}_{2L})} \lesssim 
				\begin{cases}
					\vert f\vert_{\dot{H}^{s+1/2}(I)} &\mbox{ if } s\not\in 2{\mathbb N}+1,\\
					\vert f\vert_{\dot{H}^{s+1/2}(I)} +  \big\vert( |x-a|^{-1/2}+|x-b|^{-1/2})\partial_x^s f \big\vert_{L^2(I)}&\mbox{ if } s\in 2{\mathbb N}+1.
				\end{cases}
			\end{equation} 
			\item [2.] If $I=(a,\infty)$ (resp. $(-\infty,a)$), with $a\in {\mathbb R}$, and if $f$ is compatible at order $s$ at $a$, then its Boussinesq extension $f^{\rm B}$ belongs to $\dot{H}^{s+1/2}({\mathbb R})$, and there holds,
			\begin{equation}\label{Eq: fB inf extension}
				\vert f^{\rm B}\vert_{\dot{H}^{s+1/2}(\R)} \lesssim 
				\begin{cases}
					\vert f\vert_{\dot{H}^{s+1/2}(I)} &\mbox{ if } s\not\in 2{\mathbb N}+1,\\
					\vert f\vert_{\dot{H}^{s+1/2}(I)} +\big\vert |x-a|^{-1/2}\partial_x^s f \big\vert_{L^2(I)}&\mbox{ if } s\in 2{\mathbb N}+1.
				\end{cases}
			\end{equation}
		\end{itemize}
	\end{prop}
	\begin{remark}\label{remregBouss}
		The same kind of results of course hold with standard Sobolev spaces $H^{s+1/2}(I)$ instead of $\dot{H}^{s+1/2}(I)$.
	\end{remark} 
	
	\begin{proof}
		We consider only the case where $I$ is bounded because the case of the half-line is a straightforward adaptation.
		We let $I = (0,L)$ for simplicity. We then divide the proof into several cases, depending on the regularity index $s$. But first, we make a classical observation that relates the two norms. In particular, for any $n\in \Z$,  $\delta \in I$, and $r \in (0,1)$ there holds,
		\begin{align}\label{Eq: Stricharz}
			|nk_L|^2(1+(nk_L)^{2r-2}) 
			\leq C
			\int_{0}^{\delta}	\frac{|1-e^{-in k_L  h}|^2}{h^{2r+1}} \: \mathrm{d} h,
		\end{align}
		for some constant $C$ depending only on $r$ and $\delta$, so that by Fubini and Parseval's identity we get
		\begin{align*}
			\vert f^{\rm B}\vert_{\dot{H}^r({\mathbb R}_{2L})}^2
			&  \lesssim 
			\sum \limits_{n\in \Z}	\int_{0}^{\delta}	\frac{|1-e^{-in k_L  h}|^2}{h^{2r+1}} |\widehat{f^{\rm B}}(n)|^2 \: \mathrm{d} h
			\\ 
			& = 
			\int_{-L}^{L} \int_{y}^{\delta + y} \frac{|f^{\rm B}(x) -  f^{\rm B}(y)|^2}{|x-y|^{2r+1}} \: \mathrm{d}x \mathrm{d} y.
		\end{align*}
		We now prove \eqref{Eq: fB extension}, considering separate cases for $s\geq 0$.%
		\\ 
		
		\noindent	
		{\bf{Step 1}.} Case $s\in [0,1/2)$. From the previous observation, with $r=s+1/2$, and the definition of $f^{\rm B}$ we find that 
		\begin{align*}
			\vert f^{\rm B}\vert_{\dot{H}^{s+1/2}({\mathbb R}_{2L})}^2
			&  \lesssim 
			\vert f\vert_{\dot{H}^{s+1/2}(I)}^2
			+
			\int_{-\delta}^{0} \int_{0}^{\delta } \frac{|f^{\rm B}(x) -  f^{\rm B}(y)|^2}{|x-y|^{2s+2}} \: \mathrm{d}x \mathrm{d} y
			\\ 
			& 
			\hspace{0.5cm}
			+
			\int_{L-\delta}^{L} \int_{L}^{L+\delta } \frac{|f^{\rm B}(x) -  f^{\rm B}(y)|^2}{|x-y|^{2s+2}} \: \mathrm{d}x \mathrm{d} y.
		\end{align*}
		Since the second and third terms of the right-hand side can be treated similarly, we just focus on the second one. Since for all $y\in (-\delta,0)$, one has $f^{\rm B}(y)=f(-y)$ and because $(x+y)^{-(2s+2)}\leq (x-y)^{-(2s+2)}$ when $x,y\in (0,\delta)$, we can perform  the change of variable $y\mapsto -y$ in the integral to obtain 
		\begin{align*}
			\int_{-\delta}^{0} \int_{0}^{\delta } \frac{|f^{\rm B}(x) -  f^{\rm B}(y)|^2}{|x-y|^{2s+2}} \: \mathrm{d}x \mathrm{d} y
			& = \int_{0}^{\delta} \int_{0}^{\delta } \frac{|f(x) -  f(-y)|^2}{|x+y|^{2s+2}} \: \mathrm{d}x \mathrm{d} y 
			\\ 
			& \leq \vert f\vert_{\dot{H}^{s+1/2}(I)}^2,
		\end{align*}
		so that we can conclude that $\vert f^{\rm B}\vert_{\dot{H}^{s+1/2}({\mathbb R}_{2L})}
		\lesssim 
		\vert f\vert_{\dot{H}^{s+1/2}(I)}$.\\
		
		\noindent	
		{\bf{Step 2}.} Case $s\in [1/2,1)$. We observe for $s=1/2$ that 
		\begin{align*}
			\vert f^{\rm B}\vert_{\dot{H}^1(\R_{2L})}^2
			=  
			\sum \limits_{n\in \Z} (nk_L)^2|\widehat{f^{\rm B}}(n)|^2
			= 
			\sum \limits_{n\in \Z} 
			|\widehat{\partial_x f^{\mathrm{B}}}(n)|^2,
		\end{align*}
		so that, by Parseval identity, $\vert f^{\rm B}\vert_{\dot{H}^1(\R_{2L})}\leq |\partial_x f\vert_{L^2(I)}=\vert f\vert_{\dot{H}^1(I)}$. For $s\in(1/2,1)$ we remark that on $(-L,L)$, one has $\partial_x f^{\rm B}=(\partial_x f)^{\rm odd}$, where  $(\partial_x f)^{\rm odd}$ denotes the $2L$ periodization of the odd reflection of $f$ with respect to the origin. We then argue, as in the previous step, to see that
		\begin{align*}
			\vert f^{\rm B}\vert_{\dot{H}^{s+1/2}_{\R_{2L}}}^2
			&  \lesssim 
			\vert f\vert_{\dot{H}^{s+1/2}(I)}^2
			+
			\int_{-\delta}^{0} \int_{0}^{\delta } \frac{|(\partial_x f)^{\rm odd}(x) -  (\partial_xf)^{\rm odd}(y)|^2}{|x-y|^{2s+2}} \: \mathrm{d}x \mathrm{d} y
			\\ 
			& 
			\hspace{0.5cm}
			+
			\int_{L-\delta}^{L} \int_{L}^{L+\delta } \frac{| (\partial_x f)^{\rm odd}(x) -  (\partial_x f)^{\rm odd}(y)|^2}{|x-y|^{2s+2}} \: \mathrm{d}x \mathrm{d} y.
		\end{align*}
		For the second term, we find that
		\begin{align*}
			\int_{-\delta}^{0} \int_{0}^{\delta } \frac{|(\partial_x f)^{\rm odd}(x) -  (\partial_xf)^{\rm odd}(y)|^2}{|x-y|^{2s+2}} \: \mathrm{d}x \mathrm{d} y
			& \lesssim \int_0^{\delta} \frac{1}{|x|^{2(s-1/2)}} |\partial_x f|^2 \: \mathrm{d}x, 
		\end{align*}
		and since $0<s-1/2<\frac{1}{2}$ we may use Hardy's inequality (see for instance Theorem 1.4.4.4. in \cite{Grisvard85}) to control the last term by $|\partial_x f|_{\dot{H}^{s-1/2}(I)}^2$. The last term is treated similarly.\\

		\noindent
		{\bf{Step 3}.} Case $s = 1$. Arguing as in the previous case, we find that
		\begin{equation*}
			\vert f^{\rm B}\vert_{\dot{H}^{3/2}(\R_{2L})}^2
			\lesssim |\partial_x f|_{\dot{H}^{\frac{1}{2}}(I)}^2  
			+
			|x^{-\frac{1}{2}}\partial_xf|^2_{L^2([0,\delta])}
			+
			|(L-x)^{-\frac{1}{2}}\partial_xf|^2_{L^2([L-\delta,L])}.
		\end{equation*}
		
		\noindent
		{\bf{Step 4}.} Case $s \in (1, 3/2)$. We proceed exactly as in Step 2, but since one now has $\frac{1}{2}<s-1/2<1$, Hardy's inequality requires that $\partial_x f(0) = \partial_x f(L)=0$. \\ 
		
		\noindent
		{\bf{Step 5}.} Case $s\geq 3/2$. First observe that if $\partial_xf(0)=\partial_x f(1)=0$ then $\partial_x^2 f^{\rm B}=(\partial_x^2 f)^{\rm B}$. For $s=3/2$, one can then easily compute that 
		\begin{align*}
			\vert f^{\rm B}\vert_{\dot{H}^2(\R_{2L})}^2
			&=  
			\sum \limits_{n\in \Z} (nk_L)^2 (1+(nk_L)^2)|\widehat{f^{\rm B}}(n)|^2\\
			&= 
			\sum \limits_{n\in \Z} 
			|\widehat{\partial_x f^{\mathrm{B}}}(n)|^2
			+
			\sum \limits_{n\in \Z} 
			|\widehat{(\partial_x^2f)^{\mathrm{B}}}(n)|^2.
		\end{align*}
		The first term corresponds to the case $s=1/2$, and for the second one, we can use Parseval's identity to finally conclude that $\vert f^{\rm B}\vert_{\dot{H}^2(\R_{2L})}\lesssim \vert \partial_x f\vert_{H^1(I)}$.
		For $3/2<s<5/2$, the same analysis as in Step 1 can then be performed.  For $5/2\leq s<7/2$, one proceeds as in Steps 2, 3, and 4, adding the additional condition on $\partial^3_x f(0)=\partial^3_x f(L)=0$ for $s\geq 3$. This scheme is repeated for $2k\leq s +1/2 < 2k+2$ for $k\geq 2$.
		\\

	\end{proof}
	
	\subsection{Proof of Lemma \ref{LMell1bis}}\label{appprooflm}

	First, we do the proof in the case where ${\mathcal E}_-= (x_{\rm L},x_{\rm l})$  is finite while  ${\mathcal E}_+ = (x_{\rm r}, \infty)$  is infinite. All the other configurations can be treated with simple adaptations. We also denote by $\psi_-$ and $\psi_+$ the restriction of $\psi$ on $\mathcal{E}_-$ and ${\mathcal E}_+$ respectively.
	Secondly, we construct explicit harmonic extensions of $\psi_-$ and $\psi_+$ whose horizontal derivatives vanish on the vertical lines issued from the contact points. These extensions are then used to construct $\psi^{\rm ext}$.\\

	\noindent
	{\bf Step 1.} \textit{Extension of $\psi_- \in \dot{H}^{s + \frac{1}{2}}(\mathcal{E}_-)$, with $s\geq 0$}.
	Under the assumptions of the lemma, we know from Proposition \ref{propositionregext} that the Boussinesq extension $\psi_-^{\rm B}$  belongs to $\dot{H}^{s+\frac{1}{2}}(\R_{2L})$ with $L = x_{\rm l} - x_{\rm L}$. Let $H_0>0$ be large enough so that $\Omega$ is contained in the strip ${\mathbb R}\times (-H_0,0)$, and introduce the periodic strip ${\mathcal S}_{-}=\R_{2L}\times (-H_0,0)$. Then we define the extension $\psi_-^{\rm ext}:{\mathcal S}_{-} \rightarrow \R$ via the Poisson kernel:
	\begin{align*}
		\psi_-^{\rm ext}(x,z)
		= \exp(z |\mathrm{D}|)\psi_-^{\rm B}(x)  = 
		\sum \limits_{n \in \Z} e^{z|n|} \widehat{\psi_-^{\rm B}}(n) e^{ink_L x}.
	\end{align*}
	Clearly, $\psi_-^{\rm ext}$ is harmonic. Since $\psi_-^{\rm B}$ is even we have that $\psi_-^{\rm ext}$ reduces to a cosine series and gives
	\begin{equation}\label{extneum finite}
		\partial_x \psi_-^{\rm ext}=0 \mbox{ on } \{x_{\rm L}\}\times {\mathbb R}_-\cup \{x_{\rm l}\}\times {\mathbb R}_-.
	\end{equation}
	By classical properties of Poisson kernels, we also have $\Vert \nabla \psi_-^{\rm ext}\Vert_{H^s({\mathcal S}_{-})}\lesssim |\psi_-^{\rm B}|_{\dot{H}^{s + \frac{1}{2}}(\R_{2L})}$. 
	We then infer from Proposition \ref{propositionregext} that 
	$$
	\Vert \nabla \psi_-^{\rm ext}\Vert_{H^{s}( {\mathcal S}_{-})} 
	\lesssim
	\begin{cases}
		\vert \psi_-\vert_{ \dot{H}^{s+\frac{1}{2}}(\mathcal{E}_-)} &\mbox{ if } s \not\in2{\mathbb N}+1,\\
		\vert \psi_-\vert_{ \dot{H}^{s+\frac{1}{2}}(\mathcal{E}_-)} +\vert (|x-x_{\rm L}|^{-\frac{1}{2}}+|x-x_{\rm l}|^{-\frac{1}{2}})\partial_x^{s}\psi_-\vert_{L^2(\mathcal{E}_-)} &\mbox{ if } s \in2{\mathbb N}+1.
	\end{cases}
	$$
	
	\noindent
	{\bf Step 2.} \textit{Extension of $\psi_+ \in \dot{H}^{k + \frac{1}{2}}({\mathcal E}_+)$}. We define the strip ${\mathcal S}_{+}= \R\times (-H_0,0)$. Then we define the extension $\psi_+^{\rm ext}:{\mathcal S}_{+} \rightarrow \R$ as
	\begin{align*}
		\psi_+^{\rm ext}(x,z)
		= \exp(z |\mathrm{D}|)\psi_+^{\rm B}(x)  = 
		\mathcal{F}^{-1}\big{(} e^{z |\xi|} \widehat{\psi_+^{\rm B}}(\xi) \big{)}(x).
	\end{align*}
	Clearly, we have by Plancherel's theorem  and Proposition \ref{propositionregext} that
	\begin{align*}
		\Vert \nabla \psi_+^{\rm ext}\Vert_{H^s({\mathcal S}_{+})}^2 
		& \lesssim
		|\psi_+^{\rm B}|^2_{\dot{H}^{s + \frac{1}{2}}(\R)}
		\\ 
		& \lesssim 
		\begin{cases}
			\vert \psi_-\vert_{ \dot{H}^{s+\frac{1}{2}}(\mathcal{E}_-)} &\mbox{ if } s \not\in2{\mathbb N}+1,\\
			\vert \psi_-\vert_{ \dot{H}^{s+\frac{1}{2}}(\mathcal{E}_-)} +\vert |x-x_{\rm r}|^{-\frac{1}{2}}\partial_x^{s}\psi_-\vert_{L^2(\mathcal{E}_-)} &\mbox{ if } s \in2{\mathbb N}+1,
		\end{cases}
	\end{align*}
	while the condition 
	\begin{equation}\label{extneum finite psi_ext2}
		\partial_x \psi_+^{\rm ext}=0 \mbox{ on } \{x_{\rm r}\}\times {\mathbb R}_-
	\end{equation}
	can be verified by direct computations. 
	\\

	\noindent
	{\bf Step 3.} \textit{Extension of $\psi\in \dot{H}_{\rm c.c.}^{s + \frac{1}{2}}(\Gamma^{\rm D})$}. We will now glue the two functions above together to construct the final extension on $\Gamma^{\rm D}$. To do so, we introduce a smooth cut-off function $\chi_{-}: [x_{\rm L}, \infty) \rightarrow [0,1]$ such that  $\chi_{-}(x)=1$ for $x\in \mathcal{E}_-$, $\chi_{-}(x)=0$ for $x\in  {\mathcal E}_+$. Then define $\psi^{\rm ext}$ as
	\begin{align}
		\label{defpsiext}
		\psi^{\rm ext} (x,z)&=\chi_{-}(x)\psi_-^{\rm ext}(x,z)+(1-\chi_{-}(x))\psi_+^{\rm ext}(x,z). 
	\end{align}
	Clearly, we have that $\partial_{\rm n} \psi^{\rm ext}=0$ on $\Gamma^{\rm N}\cap \{z>-h_1\}$ for $h_1>0$ is chosen small to ensure that $\Gamma^{\rm N}\cap \{z>-h_1\}$ is a union of vertical segments. It therefore only remains to prove  \eqref{Eq: Est on nabla psi ext in Hk App. C}. To do so, we add and subtract $c_1$ which is the average of $\psi_-^{\rm B}$ over one period and $c_2$ being the average of $\psi_+^{\rm B}$ on some finite interval:
	\begin{equation}\label{ST21}
		\nabla \psi^{\rm ext}={\mathbf v}_{\rm I}+{\mathbf v}_{\rm II},
	\end{equation}
	with
	\begin{align*}
		{\mathbf v}_{\rm I}=&\chi_{-}(x) \nabla \psi_-^{\rm ext}+(1-\chi_{-}(x)) \nabla \psi_+^{\rm ext}-(\nabla \chi_{-} )(c_2-c_1),\\
		{\mathbf v}_{\rm II}= &(\nabla \chi_{-})(\psi_-^{\rm ext}-c_1)- (\nabla \chi_{-})(\psi_+^{\rm ext}(x,z)-c_2).
	\end{align*}
	From Steps 1 and 2, and Definition \ref{defGN}, one directly gets that
	\begin{equation}\label{ST22}
		\Vert {\mathbf v}_{\rm I}\Vert_{H^{s}(\Omega)}\lesssim 
		\vert \psi\vert_{\dot{H}_{\rm c.c.}^{s+\frac{1}{2}}(\Gamma^{\rm D})}. 
	\end{equation} 
	For the estimate on ${\mathbf v}_{\rm II}$ we use that $\nabla\chi_{-}$ is supported on $(x_{\rm l}, x_{\rm r})$ to obtain
	$$
	\Vert {\mathbf v}_{\rm II}\Vert_{H^{s}(\Omega)}
	\lesssim
	\Vert \psi_-^{\rm ext}-c_1 \Vert_{H^{s}((x_{\rm l}, x_{\rm r})\times (-H_0,0))}+\Vert \psi_+^{\rm ext}-c_2 \Vert_{H^{s}((x_{\rm l}, x_{\rm r})\times (-H_0,0))}.
	$$
	Both terms are estimated in the same way, where we use the identity
	\begin{equation*}
		\psi_j^{\rm ext} =  \psi_j^{\rm B}(x) - \int_{z}^0 \partial_z \psi_j^{\rm ext}(x,z')\: \mathrm{d}z',
	\end{equation*}
	for $j=1,2$ to deduce that
	\begin{align*}
		\Vert \psi_j^{\rm ext}-c_j \Vert_{H^{s}((x_{\rm l}, x_{\rm r})\times (-H_0,0))} 
		& \lesssim \vert \psi_j^{\rm B}-c_j \vert_{L^2(x_{\rm l}, x_{\rm r})}+\Vert \nabla\psi_j^{\rm ext}\Vert_{H^{s}((x_{\rm l}, x_{\rm r})\times (-H_0,0))},
	\end{align*}
	where we also used the Minkowski inequality. To conclude, we use Theorem 3.5 in \cite{LeoniTice19}, which states for all finite interval $I$ one has 
	$$\vert f - \frac{1}{|I|}\int_I f \vert_{L^2(I)}\lesssim \vert f\vert_{\dot{H}^{\frac{1}{2}}(I)}.$$
	Then, since the constants $c_j$ appearing in the definition of $\dot{H}^{k+\frac{1}{2}}(\Gamma^{\rm D})$ do not depend on the choice of the interval (see Proposition 3 in \cite{LannesMing24}), we get together with Step 1 that 
	\begin{equation}\label{ST23}
		\Vert {\mathbf v}_{\rm II}\Vert_{H^{s}(\Omega)}\lesssim 
		\vert \psi\vert_{\dot{H}_{\rm c.c.}^{s+\frac{1}{2}}(\Gamma^{\rm D})}. 
	\end{equation} 
	Estimate \eqref{Eq: Est on nabla psi ext in Hk} therefore follows from \eqref{ST21}, \eqref{ST22}, and \eqref{ST23}. This concludes the proof of Lemma \ref{LMell1}.

	\subsection{Proof of Lemma \ref{Lemma: extension Neumann2}}\label{appLMell2}

	Let $f_\pm$ be the restriction of $f$ on $\mathcal{E}_\pm$. Then we argue as in the proof of Lemma \ref{LMell1}, keeping the same notations; the only change is on the nature of the extension of the boundary data, since we now have to deal with non-homogeneous Neumann rather than Dirichlet conditions. As above, $h_1>0$ is chosen such that $\Gamma^{\rm N}\cap \{z>-h_1\}$ is a union of vertical segments. We also take $H_0$ large enough to have $\Omega$ strictly included in the strip $\R\times (-H_0,0)$.   \\ 
	
	\noindent
	{\bf Step 1.} \textit{Neumann extension of $f_- \in H^{s-\frac{1}{2}}(\mathcal{E}_-)$}. A function $F_-$, harmonic in the periodic strip ${\mathcal S}_-=\R_{2L}\times (-H_0,0)$ and such that $\partial_z F_-=f_-$ on $\mathcal{E}_-$ is given by 
	\begin{align*}
		F_-(x,z)
		& = 
		\Big{(}\frac{\sinh((z+H_0) |\mathrm{D}|)}{\cosh(H_0|\mathrm{D}|)}\frac{1}{|\mathrm{D}|}f_1^{\rm B}\Big{)}(x). 
	\end{align*}
	Also, proceeding as in Step 1 in the proof of Lemma \ref{LMell1bis} in Section \ref{appprooflm}, we observe that
	\begin{equation}\label{extneum finite neumann 1}
		\partial_x F_- =0 \mbox{ on } \{x_{\rm L}\}\times (-H_0,0)\cup \{x_{\rm l}\}\times (-H_0,0),
	\end{equation}
	%
	%
	%
	and that 
	\begin{align*}
		\|  F_-\|_{H^{s+1}({\mathcal S}_{-})} 
		& \lesssim  |f_-^{\rm B}|_{{H}^{s-1/2}(\R_{2L})} \\
		&\lesssim  \vert f_- \vert_{{H}^{s-1/2}_{\rm c.c.}({\mathcal E}_-)} . 
	\end{align*}
	where we also used Proposition \ref{propositionregext} and Remark \ref{remregBouss} for the final estimate.\\
	
	\noindent
	{\bf Step 2.} \textit{Neumann extension of $f_+ \in H^{s-\frac{1}{2}}({\mathcal E}_+)$}. The same formula as in the previous step furnishes a function $F_+$, harmonic in the strip ${\mathcal S}_+=\R\times (-H_0,0)$, such that $\partial_z F_+=f_+$ on ${\mathcal E}_+$, and that satisfies $\partial_x F_+ =0 \mbox{ on } \{x_{\rm r}\}\times (-H_0,0)$ as well as the estimate
	$$  \Vert F_+\|_{H^{s+1}({\mathcal S}_{-})} 
	\lesssim \vert f_+ \vert_{{H}^{s-1/2}_{\rm c.c.}({\mathcal E}_-)}.
	$$

	\noindent
	{\bf Step 3.} \textit{Neumann extension of $f \in H^{s -\frac{1}{2}}(\Gamma^{\rm D})$}. As in Step 3 of the proof of Lemma \ref{LMell1bis}, we introduce a smooth cut-off function $\chi_{-}: [x_{\rm L}, \infty) \rightarrow [0,1]$ such that  $\chi_{-}(x)=1$ for $x\in \mathcal{E}_-$, $\chi_{-}(x)=0$ for $x\in  {\mathcal E}_+$. We then define $F$ as
	\begin{align} 
		F (x,z)
		& =
		\chi_{-}(x)F_-(x,z) 
		+
		(1-\chi_{-}(x))F_+(x,z).
	\end{align}
	This extension satisfies $\partial_{\rm n} F = 0$ on $\Gamma^{\rm N}\cap \{z>-h_1\}$ and by construction there holds
	\begin{equation*}
		\quad \partial_z  F = f_\pm
		\quad \text{on} \quad 
		\mathcal{E}_\pm,
	\end{equation*}
	and moreover $\Delta F\in H^{s}(\Omega)$ and vanishes outside $\Omega_{x_{\rm l},x_{\rm r}}$.
	We are therefore left to prove  \eqref{Eq: Est on nabla F ext in Hk2}, which follows easily by the previous two steps.

	\bibliographystyle{plain}
	\bibliography{Biblio.bib}
	
\end{document}